\documentclass[preprint,1p,times]{elsarticle}

\pdfoutput=1
\usepackage{amsmath,amssymb,amsthm}
\usepackage{mathtools}
\usepackage[vcentermath,enableskew]{youngtab}\Yboxdim10pt
\usepackage{tikz-cd}
\usepackage{hyperref}

\newcommand\kk{\Bbbk}
\newcommand\unit{\text{\usefont{U}{bbm}{m}{n}{1}}}
\newcommand\NN{\mathbb{N}}
\newcommand\ZZ{\mathbb{Z}}
\newcommand\QQ{\mathbb{Q}}
\newcommand\fD{\mathfrak{D}}
\newcommand\fS{\mathfrak{S}}
\newcommand\sP{\mathsf{P}}
\newcommand\sR{\mathsf{R}}
\newcommand\sS{\mathsf{S}}
\newcommand\sT{\mathsf{T}}

\newcommand\cA{\mathcal{A}}
\newcommand\cB{\mathcal{B}}
\newcommand\cC{\mathcal{C}}
\newcommand\cD{\mathcal{D}}
\newcommand\cE{\mathcal{E}}
\newcommand\cH{\mathcal{H}}
\newcommand\cI{\mathcal{I}}
\newcommand\cV{\mathcal{V}}
\newcommand\cM{\mathcal{M}}
\newcommand\cS{\mathcal{S}}
\newcommand\cG{\mathcal{G}}

\newcommand\rS{\text{\usefont{U}{rsfs}{m}{n}{S}}}

\newcommand\id{\mathrm{id}}
\newcommand\op{\mathrm{op}}
\newcommand\Id{\mathrm{Id}}
\newcommand\term[1]{\emph{#1}}
\DeclareMathOperator\Tr{Tr}
\DeclareMathOperator\Ker{Ker}
\DeclareMathOperator\Coker{Coker}
\DeclareMathOperator\Image{Image}
\DeclareMathOperator\Hom{Hom}
\DeclareMathOperator\End{End}
\DeclareMathOperator\Ext{Ext}
\DeclareMathOperator\Irr{Irr}

\DeclareMathOperator\Tab{\mathrm{Tab}}
\DeclareMathOperator\STab{\mathrm{STab}}
\DeclareMathOperator\cHom{\mathcal{H}\mathit{om}}
\DeclareMathOperator\cEnd{{\mathcal{E}\mathit{nd}}}
\DeclareMathOperator\cAdj{{\mathcal{A}\mathit{dj}}}

\DeclareMathOperator\Ann{\mathrm{Ann}}
\newcommand\onetoone{\stackrel{1:1}{\longleftrightarrow}}

\newcommand\Mod{\mathcal{M}\mathit{od}}

\newcommand\lMod[1]{#1\text{-}\Mod}
\newcommand\rMod[1]{\Mod\text{-}#1}

\newcommand\bMod[2]{#1\text{-}\Mod\text{-}#2}

\DeclareMathOperator\Rep{\mathcal{R}\mathit{ep}}

\newcommand\uH{\underline H}
\newcommand\ugamma{\underline{\smash\gamma}}
\newcommand\um{\underline m}
\newcommand\uM{\underline M}
\newcommand\uS{\underline S}
\newcommand\uGamma{\underline\Gamma}
\newcommand\uTab{\underline\Tab}

\newcommand\scdots{{\scriptstyle\cdots}}
\newcommand\qbinom[2]{\genfrac{[}{]}{0pt}{}{#1}{#2}}
\DeclarePairedDelimiter\abs{\lvert}{\rvert}
\DeclarePairedDelimiterX\set[2]{\{}{\}}{#1\,\delimsize\vert\,#2}
\DeclarePairedDelimiter\qqn{\llbracket}{\rrbracket}

\newcommand\mathcircled[1]{\text{\textcircled{$\scriptstyle#1$}}}
\newcommand\circleA{\mathcircled1}
\newcommand\circleB{\mathcircled2}
\newcommand\circleC{\mathcircled3}
\newcommand\circleD{\mathcircled4}
\newcommand\circleE{\mathcircled5}
\newcommand\circleF{\mathcircled6}
\newcommand\circleG{\mathcircled7}
\newcommand\circleH{\mathcircled8}

\newcommand\circleI{\mathcircled i}
\newcommand\circleJ{\mathcircled j}

\makeatletter
\newcommand\yng@margin{\ifmmode\mathchoice{\,}{}{}{}\fi}
\let\yng@original=\yng
\def\yng(#1){\yng@margin\yng@original(#1)\yng@margin}
\let\young@original=\young
\def\young(#1){\yng@margin\young@original(#1)\yng@margin}
\makeatother

\newtheorem{theorem}{Theorem}[section]

\newtheorem{proposition}[theorem]{Proposition}
\newtheorem{lemma}[theorem]{Lemma}
\newtheorem{corollary}[theorem]{Corollary}

\theoremstyle{definition}
\newtheorem{definition}[theorem]{Definition}
\newtheorem{notation}[theorem]{Notation}

\theoremstyle{remark}
\newtheorem{remark}[theorem]{Remark}
\newtheorem{example}[theorem]{Example}

\newcommand\theoremwithoutproof[1]{%
\newenvironment{#1*}[1][]{\begin{#1}[##1]\pushQED\qed}{\popQED\end{#1}}%
}
\theoremwithoutproof{theorem}
\theoremwithoutproof{proposition}
\theoremwithoutproof{lemma}
\theoremwithoutproof{corollary}
\let\theoremwithoutproof\relax

\begin{document}

\begin{frontmatter}

\title{A cellular approach to the Hecke--Clifford superalgebra}

\author[utms]{Masaki Mori}
\ead{mori@ms.u-tokyo.ac.jp}

\address[utms]{Graduate School of Mathematical Sciences, The University of Tokyo, Tokyo 153, Japan}

\journal{arXiv.org}

\begin{abstract}
The Hecke--Clifford superalgebra is a super-analogue of
the Iwahori--Hecke algebra of type $\mathsf{A}$.
The classification of its simple modules
is done by Brundan, Kleshchev and Tsuchioka using
a method of categorification of affine Lie algebras.
In this paper, we introduce another way to produce its
simple modules with a generalized theory of cellular algebras
which is originally developed by Graham and Lehrer.
In our construction the key is that there is a right action of
the Clifford superalgebra on the super-analogue of the Specht module.
With the help of the notion of the Morita context, a simple module
of the Hecke--Clifford superalgebra is made from that of
the Clifford superalgebra.
\end{abstract}

\end{frontmatter}

\section*{Introduction}

The purpose of this paper is to classify the simple modules
over the Hecke--Clifford superalgebra by use of an extended theory of cellular algebras.
The original theory of cellular algebras
is developed by Graham and Lehrer~\cite{GrahamLehrer96}
as an axiomatization of various algebras arising as endomorphism algebra
on natural representation of classical groups and quantum groups:
the symmetric group algebra,
the Brauer algebra, the partition algebra,
the Iwahori--Hecke algebra, the Birman--Murakami--Wenzl algebra and so on so forth.
First recall the notion of cellular algebra
with more general one introduced by Du and Rui~\cite{DuRui98}.
The definition below is based on that given by K\"onig and Xi~\cite{KonigXi98},
which is equivalent but slightly ring-theoretic than the original one.
Let $(\Lambda,\le)$ be a partially ordered set.
In this introduction we assume that the set $\Lambda$ is finite for simplicity.

\begin{definition}
Let $A$ be an algebra over a commutative ring $\kk$.
$A$ is called a \term{standardly based algebra} on $\Lambda$
if it is equipped with a particular basis over $\kk$
\[\set{a^\lambda_{ij}\in A}{\lambda\in\Lambda,i\in I(\lambda),j\in J(\lambda)}\]
parametrized by families of finite sets $I(\lambda)$ and $J(\lambda)$ for each $\lambda$
which satisfies the following properties.
\begin{enumerate}
\item For each $\lambda\in\Lambda$,
the $\kk$-submodule $A^{<\lambda}\subset A$ spanned by
\[\set{a^\mu_{ij}\in A}{\mu<\lambda,i\in I(\mu),j\in J(\mu)}\]
is a 2-sided ideal of $A$.
\item For each $\lambda\in\Lambda$, there exist
a left $A$-module $M_\lambda=\kk\set{m^\lambda_i}{i\in I(\lambda)}$ and
a right $A$-module $N_\lambda=\kk\set{n^\lambda_j}{j\in J(\lambda)}$,
which also have parametrized bases, such that
\begin{align*}
M_\lambda\otimes_\kk N_\lambda&\to A/A^{<\lambda},\\
m^\lambda_i\otimes n^\lambda_j&\mapsto a^\lambda_{ij}
\end{align*}
is a homomorphism between $(A,A)$-bimodules.
\end{enumerate}
$A$ is also called a \term{cellular algebra}
if $I(\lambda)=J(\lambda)$ for all $\lambda$ and the map
$a^\lambda_{ij}\mapsto a^\lambda_{ji}$
defines an anti-involution on the algebra $A$.
\end{definition}
We here do not pay much attention to anti-involutions,
so standardly based algebras are fundamental for us.
Intuitively the cell $\kk\set{a^\lambda_{ij}}{i\in I(\lambda),j\in J(\lambda)}$
for each $\lambda\in\Lambda$
is made to imitate the structure of matrix algebra,
so that the modules $M_\lambda$ and $N_\lambda$ respectively correspond to
column and row vector spaces.
As a semisimple algebra decompose into a direct sum of matrix algebras,
a cellular algebra has a filtration whose successive quotients are such cells.

One of the most striking result of the theory is
the classification of simple modules performed as follows.
First we can show that there is a canonical $A$-bilinear form
\[(,)\colon N_\lambda\times M_\lambda\to\kk\]
between $M_\lambda$ and $N_\lambda$ for each $\lambda$.
Now suppose $\kk$ is a field and let
\[L_\lambda\coloneqq M_\lambda\big/\set{x\in M_\lambda}{(y,x)=0\text{ for all }y\in N_\lambda}\]
for each $\lambda$. Graham and Lehrer~\cite[Theorem~3.4]{GrahamLehrer96}
prove that an $A$-module $L_\lambda$ is either zero or simple,
and the set $\set{L_\lambda}{\lambda\in\Lambda,L_\lambda\neq0}$
consists of pairwise distinct all simple $A$-modules.
This is an analogue of the fact that each matrix component of a semisimple algebra
produces its simple module.

However, this strategy does not work well in representation theory of superalgebras;
there are no known non-trivial cellular superalgebras in the original definition.
This is essentially because there is another kind of simple superalgebras
in addition to matrix algebras, namely matrix algebras over the Clifford superalgebra.
The key idea is that we allow a generalized cellular algebra
to have such a new kind of cells.

The construction above of simple modules, though the developers of the theory might have not noticed,
implicitly use the notion of Morita context which connect
the two algebras $A/A^{<\lambda}$ and $\kk$.
\begin{definition}\label{def:Morita_between_algebras}
A \term{Morita context} between algebras $A$ and $B$ is
a pair of an $(A,B)$-bimodule $M$ and a $(B,A)$-bimodule $N$
equipped with bimodule homomorphisms
$\eta\colon M\otimes_B N\to A$ and $\rho\colon N\otimes_A M\to B$
which satisfy the associativity laws
\begin{align*}
\eta(x\otimes y)\cdot x' &= x\cdot\rho(y\otimes x'),&
\rho(y\otimes x)\cdot y' &= y\cdot\eta(x\otimes y')
\end{align*}
for every $x,x'\in M$ and $y,y'\in N$.
\end{definition}
This is a weaker version of the Morita equivalence~\cite{Morita58} and studied in detail by
Nicholson and Watters~\cite{NicholsonWatters88}.
Morita's original theorem says that if both $\eta$ and $\rho$ are surjective
then we have a category equivalence between
the module categories of these algebras,
and every category equivalence is obtained in this form.
For such data, we can prove the following statement
\begin{theorem}
Let $I\subset A$, $J\subset B$ be the images of $\eta$ and $\rho$ respectively.
Let $\Irr(A)$ be the set of isomorphism class of simple $A$-modules,
and let $\Irr^I(A)$ be its subset consisting of simple modules $V$
such that $IV=V$. We similarly define $\Irr^J(B)$.
For a $B$-module $W$, let $DW$ be the image of the $A$-homomorphism
\begin{align*}
M\otimes_B W&\to\Hom_B(N,W)\\
m\otimes w&\mapsto(n\mapsto\rho(n\otimes m)w).
\end{align*}
Then $W\mapsto DW$ induces a one-to-one correspondence
$\Irr^I(A)\onetoone\Irr^J(B)$.
\end{theorem}
We will prove it in Theorem~\ref{thm:morita_simple_1:1}
in a more general setting:
we also treat a Morita context between two abelian categories instead of
that between two algebras $A$ and $B$, so that
it is redefined as that between their module categories $\lMod{A}$ and $\lMod{B}$.
We do this process for two reasons.
First since our purpose
is a classification of simple objects in the module category,
it is more essential to deal directly with the module category $\lMod{A}$
rather than the algebra $A$ itself.
Second we expect that our strategy works in more general settings
outside representation theory of algebras.

Anyway, note that for a standardly based algebra, for each $\lambda\in\Lambda$
the embedding
$M_\lambda\otimes N_\lambda\hookrightarrow A/A^{<\lambda}$
and the bilinear form $N_\lambda\otimes_A M_\lambda\to\kk$ make
pair $(M_\lambda,N_\lambda)$ into a Morita context between
the algebras $A/A^{<\lambda}$ and $\kk$,
and $L_\lambda$ above is just $D\kk$ where $\kk$ is viewed as
a trivial $\kk$-module.
The classification of simple modules of a cellular algebra
is a consequence of this theorem.

For a general Morita context we do not need that one algebra is a base ring $\kk$.
Hence by replacing $\kk$ with a more general one, such as the Clifford superalgebra,
we can define generalized cellular algebras
in order to obtain a similar method of classification
which we can apply to more various things.
In this paper we introduce the notion of \term{standardly filtered algebra}
over a family of algebras $\{B_\lambda\}_{\lambda\in\Lambda}$;
see Definition~\ref{def:standardly_filtered_algebra}.
A standardly filtered algebra $A$ also consists of
a Morita context $(M_\lambda,N_\lambda)$ for each $\lambda\in\Lambda$,
between quotient algebras $A/A^{<\lambda}$ of $A$ and $B_\lambda/B'_\lambda$ of $B_\lambda$.
Let $B''_\lambda/B'_\lambda\subset B_\lambda/B'_\lambda$
the image of the Morita context map $N_\lambda\otimes_A M_\lambda\to B_\lambda/B'_\lambda$,
and write $\Irr^{B''_\lambda}_{B'_\lambda}(B_\lambda)\coloneqq\Irr^{B''_\lambda/B'_\lambda}(B_\lambda/B'_\lambda)$.
These data induce the following classification
which generalizes \cite[Theorem~3.4]{GrahamLehrer96}.

\begin{theorem}
The Morita contexts induce a one-to-one correspondence
\[\Irr(A)\onetoone\bigsqcup_{\lambda\in\Lambda}\Irr^{B''_\lambda}_{B'_\lambda}(B_\lambda).\]
\end{theorem}
In the classical case, each $B_\lambda$ is taken to be a base field $\kk$
so that $\Irr^{B''_\lambda}_{B'_\lambda}(B_\lambda)$ is either
$\{\kk\}$ or $\varnothing$.
Thus in this case $\Irr(A)$ is in bijection with some subset of $\Lambda$.
If so, we simply say that $A$ is a standardly filtered algebra
over $\kk$ on the set $\Lambda$, similarly as before.
In any case the classification of simple modules of $A$
can be reduced to those of $B_\lambda$'s via this correspondence.
In this paper we also introduce the notion of generalized standardly based algebra
and that of generalized cellular algebra
over the family $\{B_\lambda\}$, not over the single base ring $\kk$.
It seems better to list several examples
rather than to introduce its detailed definition.
In many cases a standardly filtered algebra
is produced from a category as follows.

\begin{lemma}\label{lem:category_produces_filter}
Let $\cA$ be a $\kk$-linear (super)category. For each $\lambda\in\Lambda$, let us take 
an object $X_\lambda\in\cA$ and
a subalgebra $B_\lambda\subset\End_\cA(X_\lambda)$.
Let $\cA^\lambda\subset\cA$ be a 2-sided ideal of $\cA$
generated by $X_\lambda$; that is,
\[\cA^\lambda(X,Y)\coloneqq\Hom_\cA(X_\lambda,Y)\circ\Hom_\cA(X,X_\lambda)
\subset\Hom_\cA(X,Y)\]
for each pair of objects $X,Y\in\cA$.
Now suppose that
\[\Hom_\cA(X_\mu,X_\lambda)=\sum_{\nu\le\lambda,\mu}\cA^\nu(X_\mu,X_\lambda)\]
for each pair of $\lambda,\mu\in\Lambda$ and
\[\End_\cA(X_\lambda)=B_\lambda+\sum_{\mu<\lambda}\cA^\mu(X_\lambda,X_\lambda)\]
for each $\lambda\in\Lambda$.
Then for every $\omega\in\Lambda$,
$\End_\cA(X_\omega)$ is a standardly filtered algebra over $\{B_\lambda\}$.
Here the Morita contexts for $A$ above is given by
\begin{align*}
A/A^{<\lambda}&\coloneqq\End_{\cA/\cA^{<\lambda}}(X_\omega),&
M_\lambda&\coloneqq\Hom_{\cA/\cA^{<\lambda}}(X_\lambda,X_\omega),\\
B_\lambda/B'_\lambda&\coloneqq\End_{\cA/\cA^{<\lambda}}(X_\lambda),&
N_\lambda&\coloneqq\Hom_{\cA/\cA^{<\lambda}}(X_\omega,X_\lambda)
\end{align*}
where $\cA^{<\lambda}\coloneqq\sum_{\mu<\lambda}\cA^\mu$.
The homomorphisms equipped on this Morita context is just the composition of morphisms.
\end{lemma}
\begin{example}
The Iwahori--Hecke algebra $H_n(q)$ of type $\mathsf{A}_{n-1}$
for $q\in\kk$ is standardly filtered over $\kk$
on the set of compositions of $n$.
If moreover $q$ is an invertible element, it is also standardly based
and the index set can be restricted to partitions of $n$.
For the proof we take the module category of $H_n$ as $\cA$ and
for a composition $\lambda$ we pick up the corresponding parabolic module as $X_\lambda$.
Then the Morita contexts above is given by the Specht modules.
See Section~\ref{sec:cellular_hecke}.
In particular, for $q=1$ the symmetric group algebra $\kk\fS_n$ is
also standardly based.
\end{example}
\begin{example}
The Hecke--Clifford superalgebra $H^c_n(a;q)$, which is our main target,
is also standardly filtered on the same sets over the Clifford superalgebras
with different quadratic form for each composition.
The key point is that there is a right action of
this Clifford superalgebra on the super-analogue of the Specht module.
It is also standardly based in a generalized sense if $2aq\in\kk^\times$ and
the $q$-characteristic of $\kk$ is greater than $n/2$.
The strategy of the proof is same as that for the Iwahori--Hecke algebra.
See Section~\ref{sec:cellular_hecke_clifford}.
For the special case $q=1$, we obtain that the Sergeev superalgebra
$W_n(a)=C_n(a)\rtimes\fS_n$ is
also standardly filtered over the Clifford superalgebras.
\end{example}
\begin{example}
The Temperley--Lieb algebra $\mathrm{TL}_n(t)$ is standardly based
over $\kk$ on the set of natural numbers $\{0,1,\dotsc,n\}$.
We can simply take the Temperley--Lieb category as $\cA$,
and for $k\le n$ the ``$k$-points'' object as $X_k$.
\end{example}
\begin{example}
The partition algebra $P_n(t)$ is standardly based over $\{\kk\fS_k\}_{0\le k\le n}$,
the symmetric group algebras.
Note the natural inclusion $\kk\fS_k\subset P_k(t)$.
We take Deligne's
``representation category $\Rep(\fS_t)$ of the symmetric group $\fS_t$ for $t\in\kk$''
(see \cite{Deligne07}) as $\cA$
and similarly ``$k$-points'' object $[1]^{\otimes k}$ as $X_k$.
In addition, by using that each $\kk\fS_k$ is standardly based over $\kk$ on the set of partitions of $k$,
we obtain that $P_n(t)$ is also standardly based over $\kk$
on the set of all partitions of $k\le n$.
It gives a simple alternative proof of \cite{Xi99}.
We can prove similar results for
the Brauer algebra and the walled Brauer algebra
using the Deligne's category for ``the orthogonal group $O_t$''
and ``the general linear group $GL_t$'' respectively.
\end{example}

As we listed in the examples above,
the \term{Iwahori--Hecke algebra} $H_n=H_n(q)$ is one of the most important example
of a cellular algebra whose cellular basis
is given by Kazhdan and Lusztig's canonical basis~\cite{KazhdanLusztig79}
or Murphy's basis~\cite{Murphy92,Murphy95}.
It first comes from a study of
flag varieties over the finite fields,
and also appears as an endomorphism algebra of a certain
representation of the quantum general linear group via an analogue of
the Schur--Weyl duality,
then considered as a $q$-analogue of the symmetric group algebra.
Now suppose $\kk$ is a field and $q\in\kk$ be a non-zero element.
When $q$ is not a root of unity, its representations are
very similar to those of the symmetric group in characteristic zero,
and a concrete construction of the simple modules
called Young's seminormal form is given by Hoefsmit~\cite{Hoefsmit74}.
For modular representations $q=\sqrt[e]{1}$,
its simple modules are studied by Dipper and James~\cite{DipperJames86,DipperJames87}
in a cellular way.
In addition to cellular representation theory
there is a beautiful approach on the classification of simple modules
made by Lascoux, Leclerc, Thibon~\cite{LascouxLeclercThibon96},
Ariki~\cite{Ariki96}, Grojnowski~\cite{Grojnowski99},
Brundan~\cite{Brundan98}, Kleshchev~\cite{Kleshchev95} and others,
called the \term{categorification}.
Based on their works it is proven that the union set $\bigsqcup_{n\in\NN}\Irr(H_n)$
of simple modules of $H_n$ for all $n\in\NN$ has a structure of
Kashiwara crystal~\cite{Kashiwara02} over the quantum affine algebra
$U_v(\widehat{\mathfrak{sl}}_e)$ of type $\mathsf{A}^{(1)}_{e-1}$,
and is isomorphic to the crystal basis $B(\Lambda_0)$
of the irreducible representation $V(\Lambda_0)$ whose highest weight
is the fundamental weight $\Lambda_0$
(\cite[Theorem~14.2, 14.3]{Grojnowski99} for $\lambda=\Lambda_0$).
One can obtain each simple module by applying Kashiwara operators $\tilde f_i$
on the trivial module of $H_0=\kk$. However this construction is
too abstract and hard to compute in practical use.
Compared with this Lie theory, the cellular theory has advantages that
we can construct simple modules in a concrete way,
and that we can apply it even when $\kk$ is a more general commutative ring:
we only requires that $q\in\kk$ is invertible.
\begin{theorem}
Suppose $q\in\kk$ is invertible. Then there is a one-to-one correspondence
\[\Irr(H_n)\onetoone\bigsqcup_{\lambda\colon\textup{partition}}\Irr^{\kk f_\lambda}(\kk).\]
Here $f_\lambda\coloneqq[\lambda_1-\lambda_2]![\lambda_2-\lambda_3]!\dotsm[\lambda_r]!\in\kk$
for each partition $\lambda=(\lambda_1,\lambda_2,\dots,\lambda_r)$
where $[k]!$ denotes the $q$-factorial.
\end{theorem}

Our main target in this paper is the \term{Hecke--Clifford superalgebra}
$H^c_n=H^c_n(a;q)$ for $a,q\in\kk$,
which is a super version of the Iwahori--Hecke algebra.
It is introduced by Olshanski~\cite{Olshanski92}
as a partner of the quantum Queer superalgebra via the Schur--Weyl duality
and is considered as a $q$-analogue of the wreath product
$W_n=C_n\rtimes\fS_n$ of the Clifford superalgebra,
which is called the \term{Sergeev superalgebra}.
It is known that the spin representation theory of the symmetric group $\fS_n$
is controlled by $W_n$; see \cite{BrundanKleshchev02} or \cite{Kleshchev05}.
Young's seminormal form of $H^c_n$ for characteristic zero case is independently founded by
Hill, Kujawa and Sussan~\cite{HillKujawaSussan11} and Wan~\cite{Wan10},
and the categorification method for $q=\sqrt[e]{1}$ is developed
by Brundan and Kleshchev~\cite{BrundanKleshchev01} for odd $e$
and by Tsuchioka~\cite{Tsuchioka10} for even $e$
using the quantum affine algebra
of type $\mathsf{A}^{(2)}_{e-1}$ and of type $\mathsf{D}^{(2)}_{e/2}$ respectively.
Hence this paper fills the missing one: the cellular representation
theory of $H^c_n$.
In our cellular method the classification can be done in a very weak assumption
same as the case of the Iwahori--Hecke algebra above.
This is our main theorem.
\begin{theorem}
Suppose $q\in\kk$ is invertible. Then there is a one-to-one correspondence
\[\Irr(H^c_n)\onetoone\bigsqcup_{\lambda\colon\textup{partition}}\Irr^{\Delta_\lambda+\Theta_\lambda}_{\Theta_\lambda}(\Gamma_\lambda).\]
Here $\Gamma_\lambda$ is the Clifford superalgebra defined on the quadratic form
with respect to the scalars $a\qqn{\lambda_1},a\qqn{\lambda_2},\dots,a\qqn{\lambda_r}$
where $\qqn{k}$ denotes the $q^2$-integer,
and $\Delta_\lambda,\Theta_\lambda\subset\Gamma_\lambda$ are its 2-sided ideals
defined in Section~\ref{sec:cellular_hecke_clifford}.
\end{theorem}

This paper is organized as follows.
There are roughly two parts.
First we extend
the theory of cellular algebras so that we can apply it to our target,
the Hecke--Clifford superalgebra.
We start by reviewing the enriched category theory in Section~\ref{sec:enriched}
in order to treat representation category of superalgebra.
Next in Section~\ref{sec:ideal} and \ref{sec:morita}
we define Morita contexts between two categories,
and develop the theory of classification of simple objects.
We then in Section~\ref{sec:cellular} introduce the notion of standardly filtered algebra
by use of Morita contexts.

Using these results we apply our generalized cellular theory
to the two algebras.
In Section~\ref{sec:cellular_hecke}
we reconstruct the cellular representation theory of the Iwahori--Hecke algebra
from a generalized viewpoint to make it suitable for our purpose.
Finally in Section~\ref{sec:cellular_hecke_clifford}
we proceed to the study of the cellular representation theory of the Hecke--Clifford superalgebra
and completes the classification of its simple modules.

\subsubsection*{Acknowledgments}
First and foremost I would like to thank my advisor Hisayosi Matumoto
for his useful comments and great patience.
I am also deeply grateful to Susumu Ariki
whose excellent works suggested me this exciting subject.
Finally and above all, I would wish to acknowledge my wonderful wife Yasuko
for her support of my life.

\section{Basics on enriched categories}\label{sec:enriched}
Throughout in this paper, we fix a commutative ring $\kk$.
Tensor products over $\kk$ are simply denoted by $\otimes$.
We denote by $\kk^\times$ the set of invertible elements in $\kk$.
We here recall the basic notions of enriched categories in a special case.
For details we refer the reader to the textbook \cite{Kelly82}.

\subsection{Enriched categories}
We denote by $\cM$ the symmetric tensor category of $\kk$-modules,
by $\cS$ that of $\kk$-supermodules and by $\cG$ that of graded $\kk$-modules.
So they consist of $\kk$-modules $V=\bigoplus_{i\in I}V_i$ graded by the abelian group
$I=\{1\}$, $\ZZ/2\ZZ$ and $\ZZ$ respectively, and $\kk$-homomorphisms
which respect these gradings.
When we take an element $x\in V$, we always assume that $x$ is a homogeneous element.
For such $x$, we denote by $\abs{x}\in I$ the degree of $x$.
The symmetries on them are defined as
\begin{align*}
V\otimes W&\to W\otimes V\\
x\otimes y&\mapsto(-1)^{\abs{x}\abs{y}}y\otimes x
\end{align*}
using the Koszul sign $(-1)^{\abs{x}\abs{y}}$.
If you want to use the na\"\i ve symmetry in the graded case,
you should concentrate on evenly graded spaces for convention
(actually we can take an arbitrary abelian group $I$
with a homomorphism $I\to\{\pm1\}$).

Now let $\cV$ be one of $\cM$, $\cS$ or $\cG$.
In each case, a $\cV$-category is called a \term{$\kk$-linear category},
a \term{$\kk$-linear supercategory} or a \term{$\kk$-linear graded category}.
Shortly, a $\cV$-category $\cC$ consists of a hom object $\Hom_\cC(X,Y)\in\cV$ for each pair of objects $X,Y\in\cC$
instead of an ordinary hom set.
By taking the degree-zero part $\Hom_{\cC_0}(X,Y)\coloneqq\Hom_\cC(X,Y)_0$
we obtain the underlying $\cM$-category $\cC_0$.
When we write $f\colon X\to Y$ we mean that $f$ is a homogeneous element
of $\Hom_\cC(X,Y)$.
For $\cV$-categories $\cC$ and $\cD$,
the \term{tensor product $\cV$-category} $\cC\boxtimes\cD$
and the \term{opposite $\cV$-category} $\cC^\op$ are
defined through the symmetry on $\cV$.
Their morphisms are in the form $f\boxtimes g$ and $f^\op$ respectively
and the compositions are given by
\begin{align*}
(f_1\boxtimes g_1)\circ(f_2\boxtimes g_2)&\coloneqq(-1)^{\abs{g_1}\abs{f_2}}
(f_1\circ f_2)\boxtimes(g_1\circ g_2),\\
f_1^\op\circ f_2^\op&\coloneqq(-1)^{\abs{f_1}\abs{f_2}}(f_2\circ f_1)^\op.
\end{align*}

Similarly a $\cV$-functor $F\colon\cC\to\cD$ is a collection of
a degree preserving homomorphism
$\Hom_\cC(X,Y)\to\Hom_\cD(FX,FY)$ for each pair of $X,Y\in\cC$.
By taking its degree-zero part we obtain its underlying
usual functor $F_0\colon\cC_0\to\cD_0$.
A $\cV$-natural transformation $F\to G$ is defined as
a usual natural transformation between the underlying functors $F_0\to G_0$
which satisfies an additional condition.

We denote by $\cHom(\cC,\cD)_0$ the usual category consisting of
$\cV$-functors and $\cV$-natural transformations between them.
The set (or the class) of natural transformations between
$\cV$-functors $F,G\colon\cC\to\cD$
is denoted by $\Hom_{\cC,\cD}(F,G)_0$ for short.
We can also complete this category to a $\cV$-category $\cHom(\cC,\cD)$
(except that it may not be locally small) by letting its hom object
$\Hom_{\cC,\cD}(F,G)$ as a equalizer of the parallel morphisms
\[\prod_{X\in\cC}\Hom_\cC(FX,GX)\rightrightarrows\prod_{Y,Z\in\cC}
\Hom_\cV\bigl(\Hom_\cC(Y,Z),\Hom_\cD(FY,GZ)\bigr).\]
We represent its homogeneous element as $h\colon F\to G$.
The reader should pay attention to that
its naturality means
\[(hY)\circ(Ff)=(-1)^{\abs{h}\abs{f}}(Gf)\circ(hX)\]
for every $f\colon X\to Y$.
Note that a $\cV$-algebra $A$ (a monoid object in $\cV$)
is nothing but a $\cV$-category $\cC$ with a single object $*\in\cC$
such that $A\simeq\End_\cC(*)$.
Then the category of left $A$-modules, which we denote by $\lMod{A}$,
is just the functor category $\cHom(\cC,\cV)$.
The \term{tensor product $\cV$-algebra} $A\otimes B$ and
the \term{opposite $\cV$-algebra} $A^\op$
are special cases of the operations for $\cV$-categories above.
We also denote by $\rMod{A}\simeq\lMod{A^\op}$ the category of right $A$-modules.

\subsection{Alternative definitions}
Now for a while consider the super case $\cV=\cS$.
$\cS$ has the \term{parity change functor}
$\Pi\colon\cS\to\cS$ which exchanges the grading
\[(\Pi V)_0\coloneqq V_1,\quad(\Pi V)_1\coloneqq V_0\]
for supermodule $V=V_0\oplus V_1$.
In a general $\cS$-category $\cC$, for an object $Y\in\cC$
its parity change $\Pi Y\in\cC$, if exists,
is defined as a representation of the $\cS$-functor
\[\Hom_\cC(X,\Pi Y)\simeq\Pi\Hom_\cC(X,Y).\]
If every object in $\cC$ has its parity change,
$\Pi$ can be defined as an $\cS$-functor $\Pi\colon\cC\to\cC$ and
we also have an $\cS$-natural isomorphism
\[\Pi\Hom_\cC(X,Y)\simeq\Hom_\cC(\Pi X,Y).\]
We say that such $\cC$ is \term{$\Pi$-closed}.
Instead of to treat the theory of enriched categories directly,
we can view a $\Pi$-closed $\cS$-category as a usual category
with additional informations as follows.
\begin{lemma}
Giving a $\Pi$-closed $\cS$-category $\cC$ is equivalent to
giving an $\cM$-category $\cC_0$ with an $\cM$-functor $\Pi\colon\cC_0\to\cC_0$
and an isomorphism $\xi\colon\Pi^2\simeq\Id_{\cC_0}$ such that
$\xi\Pi=\Pi\xi$ as $\cS$-natural isomorphisms $\Pi^3\to\Pi$.
\end{lemma}
\begin{proof}
Clearly a $\Pi$-closed $\cS$-category induces such a datum.
Conversely let $\cC_0$ be an $\cM$-category equipped with
$\Pi\colon\cC_0\to\cC_0$ and $\xi\colon\Pi^2\simeq\Id_{\cC_0}$.
For each $X,Y\in\cC_0$, let $\Hom_\cC(X,Y)$ be a supermodule defined by
\[\Hom_\cC(X,Y)_0\coloneqq\Hom_{\cC_0}(X,Y),\quad\Hom_\cC(X,Y)_1\coloneqq\Hom_{\cC_0}(X,\Pi Y).\]
Then we can define the composition
$\Hom_\cC(Y,Z)\otimes\Hom_\cC(X,Y)\to\Hom_\cC(X,Z)$ by
\begin{align*}
\Hom_{\cC_0}(Y,Z)\otimes\Hom_{\cC_0}(X,Y)&\to\Hom_{\cC_0}(X,Z),\\
\Hom_{\cC_0}(Y,\Pi Z)\otimes\Hom_{\cC_0}(X,Y)&\to\Hom_{\cC_0}(X,\Pi Z),\\
\Hom_{\cC_0}(Y,Z)\otimes\Hom_{\cC_0}(X,\Pi Y)&\simeq\Hom_{\cC_0}(\Pi Y,\Pi Z)\otimes\Hom_{\cC_0}(X,\Pi Y)\\
&\to\Hom_{\cC_0}(X,\Pi Z),\\
\Hom_{\cC_0}(Y,\Pi Z)\otimes\Hom_{\cC_0}(X,\Pi Y)&\simeq\Hom_{\cC_0}(\Pi Y,\Pi^2 Z)\otimes\Hom_{\cC_0}(X,\Pi Y)\\
&\simeq\Hom_{\cC_0}(\Pi Y,Z)\otimes\Hom_{\cC_0}(X,\Pi Y)\\
&\to\Hom_{\cC_0}(X,Z).
\end{align*}
The condition $\xi\Pi=\Pi\xi$ is needed for that composition of
three odd morphisms is associative.
\end{proof}

\begin{lemma}
Let $\cC$ and $\cD$ be $\Pi$-closed $\cS$-categories.
Then giving an $\cS$-functor $F\colon\cC\to\cD$ is equivalent to giving
an $\cM$-functor $F_0\colon\cC_0\to\cD_0$ between their underlying $\cM$-categories
equipped with an isomorphism $\alpha\colon F_0\Pi\simeq\Pi F_0$ which makes the diagram below
commutes:
\[\begin{tikzcd}
F_0\Pi^2 \arrow{r}{\alpha\Pi} \arrow{rd}[swap]{F_0\xi} &
\Pi F_0\Pi \arrow{r}{\Pi\alpha} &
\Pi^2 F_0 \arrow{ld}{\xi F_0} \\&
F_0\rlap{.}
\end{tikzcd}\]
\end{lemma}
\begin{proof}
The one direction is clear.
So let $F_0\colon\cC_0\to\cD_0$ be an $\cM$-functor equipped with
an isomorphism $\alpha\colon F_0\Pi\simeq\Pi F_0$.
On objects simply let $FX\coloneqq F_0X$.
Then we can define the degree preserving map $F\colon\Hom_\cC(X,Y)\to\Hom_\cC(FX,FY)$ by
\begin{align*}
\Hom_{\cC_0}(X,Y)&\to\Hom_{\cC_0}(FX,FY),\\
\Hom_{\cC_0}(X,\Pi Y)&\to\Hom_{\cC_0}(FX,F\Pi Y)\simeq\Hom_{\cC_0}(FX,\Pi FY).
\end{align*}
The commutativity of the diagram above is used to ensure that $F$
preserves composition of two odd morphisms.
\end{proof}

\begin{lemma}
Let $F$ and $G$ be $\cS$-functors $\cC\to\cD$ between $\Pi$-closed $\cS$-categories.
Then a natural transformation $h\colon F_0\to G_0$ is $\cS$-natural if and only if
the square
\[\begin{tikzcd}
F_0\Pi \arrow{r}{h\Pi} \arrow{d}[swap]{\alpha} &
G_0\Pi \arrow{d}{\alpha} \\
\Pi F_0 \arrow{r}[swap]{\Pi h} &
\Pi G_0
\end{tikzcd}\]
commutes.
\end{lemma}
\begin{proof}
The $\cS$-naturality just says that the two parallel maps
$\Hom_\cC(X,Y)\rightrightarrows\Hom_\cC(FX,GY)$, which are induced by $h$, coincide.
By the usual naturality it is satisfied for the even parts.
The condition above is equivalent to that it also holds for the odd parts.
\end{proof}

\begin{remark}
This characterization of $\Pi$-closed $\cS$-category is appeared
in \cite{KangKashiwaraTsuchioka11,KangKashiwaraOh13,KangKashiwaraOh13_2}
as their definition of \term{supercategory}.
In \cite{KangKashiwaraOh13_2}, our $\cS$-category is called
a \term{$1$-supercategory}.
\end{remark}

Next consider the graded case $\cV=\cG$.
Analogously we have the \term{$k$-th degree shift functor} $\Sigma^k$
defined by $(\Sigma^k V)_i\coloneqq V_{i+k}$ for $V\in\cG$.
We say that a $\cG$-category is \term{$\Sigma$-closed} if
each $Y\in\cC$ has its degree shift $\Sigma^k Y$ defined by
\[\Sigma^k\Hom_\cC(X,Y)\simeq\Hom_\cC(X,\Sigma^k Y).\]
$\Sigma$-closed $\cG$-category can be also characterized as follows.
The proofs are similar as before so we omit them.
\begin{lemma*}
Giving a $\Sigma$-closed $\cG$-category $\cC$ is equivalent to
giving a $\cM$-category $\cC_0$ with a $\cM$-functor
$\Sigma\colon\cC_0\to\cC_0$ which is an equivalence.
\end{lemma*}

\begin{lemma*}
Let $\cC$ and $\cD$ be $\Sigma$-closed $\cG$-categories.
Then giving a $\cG$-functor $F\colon\cC\to\cD$ is equivalent to giving
a $\cM$-functor $F_0\colon\cC_0\to\cD_0$
equipped with an isomorphism $\alpha\colon F_0\Sigma\simeq\Sigma F_0$.
\end{lemma*}

\begin{lemma*}
Let $F$ and $G$ be $\cG$-functors $\cC\to\cD$ between $\Sigma$-closed $\cG$-categories.
Then a natural transformation $f\colon F_0\to G_0$ is $\cG$-natural if and only if
the square
\[\begin{tikzcd}
F_0\Sigma \arrow{r}{f\Sigma} \arrow{d}[swap]{\alpha} &
G_0\Sigma \arrow{d}{\alpha} \\
\Sigma F_0 \arrow{r}[swap]{\Sigma f} &
\Sigma G_0
\end{tikzcd}\]
commutes.
\end{lemma*}

\subsection{Limits in enriched category}

Let $\cC$ be a $\cV$-category.
We say a usual functor $\cI\to\cC_0$ from a small category $\cI$
to the underlying category of $\cC$ a \term{diagram} in $\cC$.
So a diagram consists of $Y_i\in\cC$ for each $i\in\cI$
and a degree-zero morphism $Y_i\to Y_j$ for each arrow $i\to j$ in $\cI$.
For a diagram $\{Y_i\}_{i\in\cI}$, its \term{(conical) $\cV$-limit}
is an object $\varprojlim_i Y_i\in\cC$ equipped with a family of
degree-zero morphisms $\varprojlim_i Y_i\to Y_i$ for each $i$
which satisfies the $\cV$-natural isomorphism
\[\Hom_\cC(X,\varprojlim_i Y_i)\simeq\varprojlim_i\Hom_\cC(X,Y_i).\]
Note that the usual limit only implies
\[\Hom_{\cC_0}(X,\varprojlim_i Y_i)\simeq\varprojlim_i\Hom_{\cC_0}(X,Y_i).\]
Since taking the degree-zero part $\cV\to\cM$; $V\mapsto V_0$ preserves limits,
the $\cV$-limit of an diagram is also its usual limit.
The converse does not hold in general but
there are no differences between them in a suitable condition.
\begin{lemma}
Suppose $\cC$ is an $\cM$-category (resp.\ a $\Pi$-closed $\cS$-category,
a $\Sigma$-closed $\cG$-category).
Then for any diagram its limit in $\cC_0$ is automatically
its $\cV$-limit in $\cC$.
\end{lemma}
\begin{proof}
It is trivial for the $\cV=\cM$ case.
When $\cV=\cS$, we have
\begin{align*}
\Hom_\cC(X,\varprojlim_i Y_i)&\simeq
\Hom_{\cC_0}(X,\varprojlim_i Y_i)\oplus\Hom_{\cC_0}(\Pi X,\varprojlim_i Y_i)\\
&\simeq\varprojlim_i\Hom_{\cC_0}(X,Y_i)\oplus\varprojlim_i\Hom_{\cC_0}(\Pi X,Y_i)\\
&\simeq\varprojlim_i\Hom_\cC(X,Y_i).
\end{align*}
The similar proof works for $\cV=\cG$
since a limit commutes with direct sums in a Grothendieck category $\cG$.
\end{proof}
The dual notion \term{(conical) $\cV$-colimit} of a diagram is introduced similarly.
Next we introduce the notion of abelian $\cV$-category.
\begin{definition}
We say that a $\cS$-category (resp.\ a $\cG$-category) $\cC$ is \term{abelian}
if it is $\Pi$-closed (resp.\ $\Sigma$-closed)
and its underlying $\cM$-category $\cC_0$ is abelian in an ordinary sense.
\end{definition}
By the lemma above, in an abelian $\cV$-category
we can use several categorical notions such as
zero object, direct sum, kernel, cokernel, monomorphism, epimorphism
and exactness defined via enriched Hom functors
without any modifications.
We are also allowed to operate homological computations as follows.
\begin{lemma}
Let $\cC$ be an abelian $\cV$-category.
Then $P\in\cC$ is projective in $\cC_0$ if and only if
the functor $\Hom_\cC(P,{\bullet})\colon\cC\to\cV$ is exact.
\end{lemma}
\begin{proof}
The ``if'' part follows from that taking the degree-zero part
$V\mapsto V_0$ is exact.
The ``only if'' part can be proven similarly as the lemma above.
\end{proof}

We here study limits and colimits in a functor category.
Let $\{F_i\colon\cC\to\cD\}$ be a diagram in $\cHom(\cC,\cD)$.
If the $\cV$-limit $\varprojlim_i F_i X$ exists for each $X\in\cC$,
then the $\cV$-limit $\varprojlim_i F_i$ also exists and is defined as
\[(\varprojlim_i F_i)X\coloneqq\varprojlim_i F_i X.\]
Dually $\cV$-colimits of this diagram is also computed value-wise.
We remark that the composition of $\cV$-functors
\begin{align*}
\cHom(\cD,\cE)\boxtimes\cHom(\cC,\cD)&\to\cHom(\cC,\cE),\\
F\boxtimes G&\mapsto F G
\end{align*}
is also defined as a $\cV$-functor.
In particular, if $\cD$ is abelian then so is $\cHom(\cC,\cD)$.
By definition, the right multiplication
\[\bullet G\colon\cHom(\cD,\cE)\to\cHom(\cC,\cE)\]
is both $\cV$-continuous (i.e.\ preserves $\cV$-limits) and
$\cV$-cocontinuous (i.e.\ does $\cV$-colimits).
In contrast, the left multiplication
\[F\bullet\colon\cHom(\cC,\cD)\to\cHom(\cC,\cE)\]
preserves certain $\cV$-limits or $\cV$-colimits for all $\cC$ if and only if $F$ does.

Suppose that $\cV$-categories $\cC$ and $\cD$ are both abelian,
and $\cC$ has enough projectives.
Then for each $F\colon\cC\to\cD$,
its $i$-th left derived $\cV$-functor $L_iF\colon\cC\to\cD$ makes sense as usual.
$L_i$ can be viewed as a $\cV$-endofunctor on $\cHom(\cC,\cD)$,
which is also abelian,
and from a short exact sequence $0\to F\to G\to H\to0$
of $\cV$-functors they yield a long exact sequence
\[\dotsb\to L_2H\to L_1F\to L_1G\to L_1H\to L_0F\to L_0G\to L_0 H\to0.\]
Dually, when $\cC$ has enough injectives
we define the $i$-th right derivation $R^i$.

We are most interested in the zeroth derivation $L_0$.
By definition there is a canonical $\cV$-natural transformation $L_0 F\to F$.
$L_0F\colon\cC\to\cD$ is right exact and
of course $L_0F\simeq F$ when $F$ is already right exact.
For another right exact $\cV$-functor $G\colon\cC\to\cD$, the map
\[\Hom_{\cC,\cD}(G,L_0F)\to\Hom_{\cC,\cD}(G,F)\]
is an isomorphism since its inverse map is given by
\[\Hom_{\cC,\cD}(G,F)\to\Hom_{\cC,\cD}(L_0G,L_0F)\simeq
\Hom_{\cC,\cD}(G,L_0F).\]
Thus $L_0F$ is considered as the most applicative
right exact approximation of $F$.

\subsection{Adjunctions}

An adjunction from $\cC$ to $\cD$ is a pair of
adjoint $\cV$-functors $F\colon\cC\to\cD$ and $F^\vee\colon\cD\to\cC$
where $F$ is left adjoint to $F^\vee$.
That is, it is called so if there is a $\cV$-natural isomorphism
\[\Hom_\cD(FX,Y)\simeq\Hom_\cC(X,F^\vee Y).\]
Then $F$ is $\cV$-cocontinuous and $F^\vee$ is $\cV$-continuous.
It is also characterized by degree-zero natural transformations
$\delta\colon\Id_\cC\to F^\vee F$ and
$\epsilon\colon F F^\vee\to\Id_\cD$
which satisfy the zig-zag identities
\begin{align*}
\id_F &= (F\xrightarrow{F\delta} F F^\vee F\xrightarrow{\epsilon F} F),&
\id_{F^\vee} &= (F^\vee\xrightarrow{\delta F^\vee} F^\vee F F^\vee\xrightarrow{F^\vee\epsilon} F^\vee).
\end{align*}
$\delta$ and $\epsilon$ are called the unit and the counit of the adjunction respectively.
For a $\cV$-functor $F\colon\cC\to\cD$, the rest datum
$(F^\vee,\delta,\varepsilon)$ which makes an adjunction
is uniquely determined up to unique isomorphism if it exists.
In order to keep notations simple we say that ``$F$ is an adjunction from $\cC$ to $\cD$''
when $F\colon\cC\to\cD$ has a fixed right adjoint $\cV$-functor $F^\vee$.
We denote by $\cAdj(\cC,\cD)$ the full subcategory
of $\cHom(\cC,\cD)$ consisting of adjunctions.

For an adjunction $F\colon\cC\to\cD$, the left multiplication
$F\bullet$ is left adjoint to $F^\vee\bullet$
while the right multiplication
$\bullet F$ is right adjoint to $\bullet F^\vee$.
Thus for two parallel adjunction $F,G\colon\cC\to\cD$,
we have a canonical isomorphism
\[\Hom_{\cC,\cD}(F,G)\simeq\Hom_{\cD,\cC}(G^\vee,F^\vee).\]

We here list few examples of adjunction category.
\begin{example}
Let $A$ and $B$ be $\cV$-algebras.
For an $(A,B)$-bimodule $M$, the $\cV$-functors between their module categories
\begin{align*}
F\colon\lMod{B}&\to\lMod{A},& F^\vee\colon\lMod{A}&\to\lMod{B},\\
W&\mapsto M\otimes_B W,& V&\mapsto \Hom_A(M,V)
\end{align*}
form an adjunction from $B\text{-}\Mod$ to $\lMod{A}$.
Conversely let $F\colon B\text{-}\Mod\to\lMod{A}$ be an arbitrary adjunction.
When we put $M\coloneqq FB$ the multiplication on $B$ from right
makes $M$ a right $B$-module. We have
\[F^\vee V\simeq\Hom_B(B,F^\vee V)\simeq\Hom_A(FB,V)=\Hom_A(M,V)\]
so every adjunction between the module categories can be obtained in this way.
In addition, $\cAdj(B\text{-}\Mod,\lMod{A})$ is equivalent to $\bMod AB$,
the category of $(A,B)$-bimodules.
In particular it is abelian and the embedding to
$\cHom(B\text{-}\Mod,A\text{-}\Mod)$ is right exact,
but not left exact in general.
\end{example}

\begin{example}
Suppose $\kk$ is a field,
and for $A$ and $B$ above let $\lMod{A}^f$, $\lMod{B}^f$ be the categories of
their finite dimensional left modules.
Suppose that $N$ is a $(B,A)$-bimodule which satisfies these finiteness conditions:
\begin{enumerate}
\item if a right $A$-module $V$ is finite dimensional then so is $\Hom_{A^\op}(V,N)$,
\item if a left $B$-module $W$ is finite dimensional then so is $\Hom_B(W,N)$,
\item $N$ is locally finite dimensional, that is,
it is the union of its finite dimensional $(B,A)$-submodules.
\end{enumerate}
We denote by $V^\vee\coloneqq\Hom_\cV(V,\kk)$ the dual space of
a finite dimensional vector space.
Then the functors
\begin{align*}
F\colon\lMod{B}^f&\to\lMod{A}^f,& F^\vee\colon\lMod{A}^f&\to\lMod{B}^f,\\
W&\mapsto\Hom_B(W,N)^\vee,& V&\mapsto \Hom_{A^\op}(V^\vee,N)
\end{align*}
form an adjunction via the natural isomorphism
\begin{align*}
\Hom_A(\Hom_B(W,N)^\vee,V)&\simeq
\Hom_{A^\op}(V^\vee,\Hom_B(W,N))\\&\simeq
\Hom_B(W,\Hom_{A^\op}(V^\vee,N)).
\end{align*}
In this case $N$ is recovered from $F$ by the formula
$N\simeq\varinjlim_{B\twoheadrightarrow B'}(FB')^\vee$
where $B'$ runs over all finite dimensional quotient algebras of $B$.
Every adjunctions are obtained in this way
and $\cAdj(\lMod{B}^f,\lMod{A}^f)$ is equivalent to the opposite
of the category of such $(B,A)$-bimodules.
This category does not necessarily have kernels.
\end{example}

\begin{example}
For a small $\cV$-category $\cA$,
we call a $\cV$-functor $\cA\to\cV$
``a left $\cA$-module'', and denote by $\lMod\cA\coloneqq\cHom(\cA,\cV)$.
Similarly, for another $\cV$-category $\cB$,
a right $\cB$-module (resp.\ an $(\cA,\cB)$-bimodule) is
just a $\cV$-functor $\cB^\op\to\cV$
(resp.\ $\cB^\op\boxtimes\cA\to\cV$).
Since a $\ZZ$-linear category is a ``ring with several objects''
as Mitchell~\cite{Mitchell72} noticed, this terminology is a generalization for
usual algebras.

For an $(\cA,\cB)$-bimodule $M$ and an $(\cA,\cC)$-bimodule $N$,
we can form an $(\cB,\cC)$-bimodule $\Hom_\cA(M,N)$ defined as
\begin{align*}
\cC^\op\boxtimes\cB&\to\cV,\\
Z\boxtimes Y&\mapsto
\Hom_{\cA,\cV}(M(Y,{\bullet}),N(Z,{\bullet})).
\end{align*}
On the other hand, if $P$ is an $(\cB,\cC)$-bimodule,
there is an $(\cA,\cC)$-bimodule denoted by $M\otimes_\cB P$,
which sends $Z\boxtimes X\in\cC^\op\boxtimes\cA$ to the coequalizer of the parallel maps
\[\bigoplus_{Y',Y''\in\cB}
M(Y',X)\otimes\Hom_\cB(Y'',Y')\otimes P(Z,Y'')\rightrightarrows
\bigoplus_{Y\in\cB}
M(Y,X)\otimes P(Z,Y).\]
Similarly as above, every adjunction $F\colon\cB\text{-}\Mod\to\cA\text{-}\Mod$
is represented by some $(\cA,\cB)$-bimodule $M$ using $\otimes$ and $\Hom$.
The identity functor on $\cA\text{-}\Mod$ corresponds
to the $(\cA,\cA)$-bimodule $\Hom_\cA({\bullet},{\bullet})$.
\end{example}

\subsection{The category of adjunctions}

First of all, we make sure a well-known fact
that adjunctions are closed under colimits, especially cokernels.
Let $\{F_i\}$ be a diagram in $\cAdj(\cC,\cD)$.
For each arrow $i\to j$ between indices, the $\cV$-natural transformation $F_i\to F_j$
induces the corresponding $F_j^\vee\to F_i^\vee$.
Thus they form the diagram $\{F^\vee_i\}$ in $\cHom(\cD,\cC)$
whose arrows are reversed.

\begin{proposition}\label{prop:colim_adj}
Let $\{F_i\}$ be as above, and
suppose that the $\cV$-colimit $F\coloneqq\varinjlim_i F_i$ and the $\cV$-limit
$F^\vee\coloneqq\varprojlim_i F^\vee_i$ are both exist.
Then $F$ is left adjoint to $F^\vee$.
Moreover the canonical morphisms $F_i\to F$ and $F^\vee\to F_i^\vee$ are
mapped to each other by the isomorphism
\[\Hom_{\cC,\cD}(F_i,F)\simeq\Hom_{\cD,\cC}(F^\vee,F_i^\vee).\]
\end{proposition}

We here give two proofs for this easy but important result.
\begin{proof}[First proof]
First we prove the lemma by studying the functors value-wise.
The statements are obvious by
the $\cV$-natural isomorphism
\[\Hom_\cD(FX,Y)\simeq\varprojlim_i \Hom_\cD(F_i X,Y)
\simeq\varprojlim_i \Hom_\cC(X,F_i^\vee Y)
\simeq\Hom_\cC(X,F^\vee Y).
\qedhere\]
\end{proof}
\begin{proof}[Second proof]
The second is a ``2-categorical'' proof.
For each arrow $i\to j$, we have a commutative diagram
\[\begin{tikzcd}
\Id_\cC \arrow{r} \arrow{d} &
F^\vee_j F_j \arrow{r} \arrow{d} &
F^\vee_j F \arrow{d} \\
F^\vee_i F_i \arrow{r} &
F^\vee_i F_j \arrow{r} &
F^\vee_i F
\end{tikzcd}\]
so it induces the unit
$\Id_\cC\to\varprojlim_i F^\vee_i F\simeq F^\vee F$.
The counit $F F^\vee\to\Id_\cD$ is defined analogously.
To prove that $F\to F F^\vee F\to F$ is equal to $\id_F$,
it suffices to show that its pullback $F_i\to F\to F F^\vee F\to F$
is equal to $F_i\to F$ for each $i$.
It follows from the commutativity of the diagram below:
\[\begin{tikzcd}
F \arrow{d} &
F_i \arrow{l} \arrow{r} \arrow{d} \arrow[bend left=20]{rr}{\id} &
F_i F_i^\vee F_i \arrow{r} \arrow{d} &
F_i \arrow{d} \\
F F^\vee F \arrow[bend right=20]{rrr} &
F_i F^\vee F \arrow{l} \arrow{r} &
F_i F_i^\vee F \arrow{r} &
F\rlap{.}
\end{tikzcd}\]
Similarly $F^\vee\to F^\vee F F^\vee\to F^\vee$ is also
equal to the identity so $F$ and $F^\vee$ are adjoint to each other.
The diagram above also shows us that
$F_i\to F$ also coincides with the composite
$F_i\to F_i F^\vee F\to F_i F_i^\vee F\to F$
which is induced by $F^\vee\to F_i^\vee$.
\end{proof}

\begin{remark}\label{rem:2-cat}
Although we are now studying the 2-category of
$\cV$-categories,
the second proof also works for an arbitrary 2-category
such that the right multiplication of a 1-cell is both $\cV$-continuous and $\cV$-cocontinuous.
\end{remark}

We will soon use the next corollary.
\begin{corollary}\label{cor:exseq_of_adjs}
Suppose that both $\cC$ and $\cD$ are abelian.
A sequence of adjunctions $F\to G\to H\to 0$ is exact in $\cHom(\cC,\cD)$
if and only if so is the corresponding sequence
$0\to H^\vee\to G^\vee\to F^\vee$ in $\cHom(\cD,\cC)$.
In particular, $F\to G$
is epic if and only if the corresponding
$G^\vee\to F^\vee$ is monic.
\end{corollary}

As we have seen in the examples,
the $\cV$-category $\cAdj(\cC,\cD)$ may have limits
which differ from those taken in $\cHom(\cC,\cD)$.
We here give a simple sufficient condition for
$\cAdj(\cC,\cD)$ to be abelian.

\begin{lemma}\label{lem:ker_of_adj}
Let $\cC$ and $\cD$ be abelian $\cV$-categories,
and suppose that $\cC$ has enough projectives and
$\cD$ has enough injectives.
Let $F,G\colon\cC\to\cD$ be parallel adjunctions with
a degree-zero natural transformation $F\to G$.
Let $K\coloneqq\Ker(F\to G)$ and $C\coloneqq\Coker(G^\vee\to F^\vee)$.
Then $L_0K$ is left adjoint to $R^0C$, so it is an adjunction.
Moreover, $L_0K$ is the kernel of $F\to G$ in $\cAdj(\cC,\cD)$.
\end{lemma}
\begin{proof}
Let $X\in\cC$, $Y\in\cD$ and take a projective resolution $P'\to P\to X\to0$ and
an injective resolution $0\to Y\to Q\to Q'$ respectively.
By definition $\Hom_\cD((L_0K)X,Y)$ is the kernel of the map
\[\Hom_\cD(KP,Q)\to\Hom_\cD(KP',Q)\oplus\Hom_\cD(KP,Q').\]
Since $P$ is projective and $Q$ is injective, each term can be represented as
\begin{align*}
\Hom_\cD(KP,Q)&\simeq\Coker(\Hom_\cD(GP,Q)\to\Hom_\cD(FP,Q))\\
&\simeq\Coker(\Hom_\cC(P,G^\vee Q)\to\Hom_\cC(P,F^\vee Q))\\
&\simeq\Hom_\cC(P,CQ).
\end{align*}
Hence we have a natural isomorphism $\Hom_\cD((L_0K)X,Y)\simeq\Hom_\cC(X,(R^0C)Y)$
since its right-hand side also has a similar representation.
For any $H\in\cAdj(\cC,\cD)$ there is a natural isomorphism
\[\Hom_{\cC,\cD}(H,L_0K)\simeq\Hom_{\cC,\cD}(H,K)\simeq
\Ker(\Hom_{\cC,\cD}(H,F)\to\Hom_{\cC,\cD}(H,G)).\]
Thus $L_0K$ is the kernel of $F\to G$ taken in $\cAdj(\cC,\cD)$.
\end{proof}

\begin{proposition}\label{prop:adj_is_abelian}
Let $\cC$ and $\cD$ be as above. Then  $\cAdj(\cC,\cD)$ is abelian and
the embedding $\cAdj(\cC,\cD)\hookrightarrow\cHom(\cC,\cD)$ is right exact.
\end{proposition}
\begin{proof}
Clearly it is closed under $\Pi$ or $\Sigma$.
By Proposition~\ref{prop:colim_adj} and Lemma~\ref{lem:ker_of_adj}
it also has finite direct sums, kernels and cokernels.
In $\cAdj(\cC,\cD)$ the image of the morphism $F\to G$ is isomorphic to its coimage,
since they give same value $\Image(FP\to GP)$ on enough projectives $P\in\cC$.
The cokernel of a morphism in $\cAdj(\cC,\cD)$ is equal
to that taken in $\cHom(\cC,\cD)$ so the embedding is right exact.
\end{proof}

\section{Ideal functors in abelian categories}\label{sec:ideal}
From now on, we omit all prefixes ``$\cV$-''
in order to avoid redundant descriptions,
so ``a category'' means a $\cV$-category,
``a functor'' a $\cV$-functor, etc.

We here study some kind of endofunctors
which we call \term{ideal functors}.
These are analogues of 2-sided ideals in a ring.
Later it is used to divide the category into two parts by a Morita context.

\subsection{Ideal functors}
Let $\cC$ be an abelian category (i.e.\ an abelian $\cV$-category)
and consider the category $\cEnd(\cC)\coloneqq\cHom(\cC,\cC)$
which is also abelian
and has the specific object $\Id_\cC$, the identity functor.

\begin{definition}
A subfunctor $I\subset\Id_\cC$
is called an \term{ideal functor} on $\cC$ if its cokernel
$T_I\coloneqq\Coker(I\hookrightarrow\Id_\cC)$
is an adjunction.
The cokernel of the corresponding morphism $T_I^\vee\hookrightarrow\Id_\cC$
(monic by Corollary~\ref{cor:exseq_of_adjs}) is denoted by $I^\circ$.
\end{definition}

\begin{example}
Consider the case that $\cC$ is a module category $A\text{-}\Mod$.
Then quotient adjunctions of $\Id_{\lMod{A}}$
are in bijection with quotient $(A,A)$-bimodules of $A$, that is,
$A/I$ for a 2-sided ideal $I\subset A$.
Thus the corresponding ideal functor maps an $A$-module $V$
to the kernel of the map \[V\twoheadrightarrow T_I V\coloneqq A/I\otimes_A V\simeq V/IV,\]
namely $IV$.
This is why we call such kind of functor an ``ideal functor''.
In this case the right adjoint $T_I^\vee$ can be represented as
\[T_I^\vee V\coloneqq\Hom_A(A/I,V)\simeq\set{v\in V}{Iv=0}.\]
\end{example}

\begin{example}
When $\cC=\cA\text{-}\Mod$ is a module category of a category $\cA$,
an ideal functor on $\cC$ is also corresponds to a 2-sided ideal $\cI\subset\cA$.
Here a 2-sided ideal in a category is a collection of spaces of morphisms
\[\cI(V,W)\subset\Hom_\cA(V,W)\]
for each pair of $V,W\in\cA$,
which is closed under compositions with all morphisms in $\cA$.
From such an ideal we can form a quotient category $\cA/\cI$,
whose hom sets are defined by pairwise quotient.
\end{example}

\begin{example}\label{ex:rad_is_ideal}
The socle $\mathrm{Soc}(X)$ of an object $X\in\cC$ is
the sum of all simple subobjects of $X$.
Dually, its top $\mathrm{Top}(X)$ is defined as $X/\mathrm{Rad}(X)$
where the radical $\mathrm{Rad}(X)$ is the intersection
of all its maximal subobjects of $X$.
If these objects exist for every $X\in\cC$,
then $\mathrm{Soc}$, $\mathrm{Top}$ and $\mathrm{Rad}$
can be defined as endofunctors on $\cC$.

Suppose that for all object $X\in\cC$,
$\mathrm{Top}(X)$ and $\mathrm{Soc}(X)$
are both semisimple.
Then the functor $\mathrm{Top}$ is left adjoint to $\mathrm{Soc}$, thus
$\mathrm{Rad}=\Ker(\Id_\cC\twoheadrightarrow\mathrm{Top})$ is an ideal functor.
\end{example}

\begin{example}
If $I$ is an ideal functor on $\cC$,
$(I^\circ)^\op$ is an ideal functor on $\cC^\op$.
\end{example}

On a general abelian category,
a typical example is the image of an adjunction.

\begin{proposition}\label{prop:image_of_adjunction_is_ideal}
Let $F\colon\cC\to\cC$ be an adjunction with a degree-zero
natural transformation $F\to\Id_\cC$.
Then $I\coloneqq\Image(F\to\Id_\cC)$ is an ideal functor.
$I^\circ$ is also the image of the corresponding natural transformation
$\Id_\cC\to F^\vee$.
\end{proposition}
\begin{proof}
By Proposition~\ref{prop:colim_adj}, the cokernel
$T_I=\Coker(F\to\Id_\cC)$
has the right adjoint functor
$T_I^\vee\coloneqq\Ker(\Id_\cC\to F^\vee)$
so it is an adjunction.
Thus by definition $I$ is an ideal functor.
Since $0\to T_I^\vee\to\Id_\cC\to F^\vee$ is exact
by Corollary~\ref{cor:exseq_of_adjs},
$F^\vee$ contains $I^\circ=\Coker(T_I^\vee\hookrightarrow\Id_\cC)$
as a subfunctor.
\end{proof}

\begin{remark}
Suppose $\cC$ has enough projectives and injectives.
Then for an ideal functor $I$ on $\cC$,
$L_0I$ is an adjunction and
$L_0I\twoheadrightarrow I$ is epic
by Proposition~\ref{prop:adj_is_abelian}.
Thus in this case every ideal functor is obtained as the image of an adjunction.
When $\cC=A\text{-}\Mod$, $L_0I$ for a 2-sided ideal $I\subset A$
is just the tensor functor $I\otimes_A{\bullet}$
where $I$ is viewed as an $(A,A)$-bimodule.

On the other hand, let $A$ be a polynomial algebra $\kk[x_1,x_2,\dotsc]$
in infinitely many variables over a field $\kk$
and consider the case $\cC=\lMod{A}^f$,
the category of its finite dimensional modules.
Then $I\colon V\mapsto\sum_i x_iV$ is clearly an ideal functor on $\lMod{A}^f$
but there are no adjunctions which cover $I$.
\end{remark}

We here list basic properties of ideal functors.
For simplicity if there is a canonical isomorphism $F\to G$ between objects
which is clear from the context, we write $F=G$ for short.

\begin{lemma}\label{lem:T_I_is_idempotent}
Let $I\subset\Id_\cC$ be an ideal functor. Then
\begin{enumerate}
\item the morphisms $T_I=T_I\cdot\Id_\cC\to T_I^2$
and $T_I=\Id_\cC\cdot T_I\to T_I^2$ coincide,
\item $T_I^\vee T_I=T_I=T_I^2$,
\item $IT_I=0=I^\circ T_I$.
\end{enumerate}
\end{lemma}
\begin{proof}
(1) follows from that the epimorphism $\Id_\cC\twoheadrightarrow T_I$
equalize these two morphisms.
Since $\Id_\cC\twoheadrightarrow T_I$ factors through
$\Id_\cC\to T_I^\vee T_I\hookrightarrow T_I$,
the monomorphism $T_I^\vee T_I\hookrightarrow T_I$ is also epic.
So it must be an isomorphism since the functor category is abelian.
Similarly $T_I^\vee=T_IT_I^\vee$ holds
and we obtain
\[T_I=T_I^\vee T_I=T_I T_I^\vee T_I=T_I^2\]
so (2) holds. (3) is just a rephrasing of (2).
\end{proof}

In general an ideal functor is neither left exact nor right exact.
Still, we can prove the following useful properties.
\begin{proposition}
An ideal functor preserves all images.
\end{proposition}
\begin{proof}
Let $I$ be an ideal functor on $\cC$.
A functor is called \term{mono} (resp.\ \term{epi})
if it preserves all monomorphisms (resp. epimorphisms).
Since $\Id_\cC$ is clearly a mono functor,
so is its subfunctor $I$.
$I$ is also epi because
$\Id_\cC$ is epi and the cokernel $T_I$ is right exact;
apply the nine lemma.
Thus $I$ preserves all epi-mono factorizations.
\end{proof}

Recall that for a possibly infinite family of subobjects $\{X_i\subset X\}$,
their sum, if exists, is the minimum subobject $\sum_i X_i\subset X$
which contains all $X_i$.

\begin{lemma}\label{lem:IZ_in_Y}
Let $X\in\cC$ be an object and $Y,Z\subset X$ its subobjects.
Then $IZ\subset Y$ if and only if
$Z\subset\Ker(X\twoheadrightarrow X/Y\twoheadrightarrow I^\circ(X/Y))$
(in other words, $X/Z$ is a quotient of $I^\circ(X/Y)$).
\end{lemma}
\begin{proof}
Suppose $IZ\subset Y$, or equivalently, the composite
$IZ\hookrightarrow Z\to X/Y$ is zero.
Then $Z\to X/Y$ factors through
$Z\twoheadrightarrow T_I Z$.
Hence $Z\to X/Y\twoheadrightarrow I^\circ(X/Y)$
factors through $I^\circ T_I Z=0$
so it must be zero.
Taking the opposite category
we can dually prove the other implication.
\end{proof}
\begin{proposition}
An ideal functor commutes with summation.
\end{proposition}
\begin{proof}
Let $I$ be an ideal functor on $\cC$.
Take an object $X\in\cC$ and a family of subobjects
$\{X_i\subset X\}$ whose sum $\sum_i X_i$ exists.
Note that $X_i\hookrightarrow\sum_i X_i$ induces
$IX_i\hookrightarrow I\sum_i X_i$.
Thus if $\sum_i IX_i$ exists, it is contained in $I\sum_i X_i$.

Conversely, let $Y\subset X$ be a subobject which contains every $IX_i$.
Let
\[Y'\coloneqq\Ker(X\twoheadrightarrow X/Y\twoheadrightarrow I^\circ(X/Y)).\]
Then $X_i\subset Y'$ for all $i$ by the ``only if '' part of Lemma~\ref{lem:IZ_in_Y},
so $\sum_i X_i\subset Y'$.
On the other hand, its ``if'' part says that $IY'\subset Y$.
Thus $I\sum_i X_i\subset IY'\subset Y$,
so $\sum_i IX_i$ exists and actually
\[\sum_i IX_i=I\sum_i X_i.\qedhere\]
\end{proof}

\subsection{Subcategories defined by ideal functors}

In this subsection we fix an ideal functor $I$ on $\cC$.
Using an ideal functor, we define two full subcategory of $\cC$ in the following manner.

\begin{lemma}\label{lem:annihilated}
For an object $X\in\cC$, the following conditions are equivalent.
\begin{enumerate}
\item $IX=0$ ($\iff X=T_I X$),
\item $I^\circ X=0$ ($\iff T_I^\vee X=X$).
\end{enumerate}
\end{lemma}
\begin{proof}
Similar to the proof of $T_I=T_I^2$ in Lemma~\ref{lem:T_I_is_idempotent}.
\end{proof}

\begin{definition}
An object $X\in\cC$ is called
\term{$I$-annihilated} if it satisfies the conditions above.
We denote by $\cC_I$ the full subcategory of $\cC$ consisting of
$I$-annihilated objects.
\end{definition}

The other subcategory is defined analogously.

\begin{definition}
An object $X\in\cC$ is called
\begin{enumerate}
\item \term{$I$-accessible} if $IX=X$ ($\iff T_I X=0$),
\item \term{$I$-torsion-free} if $X=I^\circ X$ ($\iff T_I^\vee X=0$).
\end{enumerate}
We denote by $\cC^I$ the full subcategory of $\cC$
consisting of $I$-accessible $I$-torsion-free objects.
\end{definition}

By definition it is clear that these subcategories are closed
under the parity change $\Pi$ or the degree shift $\Sigma$. 

\begin{example}
Let $A$ be a ring and $I\subset A$ a 2-sided ideal.
Then an $A$-module $V$ is $I$-annihilated if and only if
it can be defined over the quotient ring $A/I$.
Thus there is a canonical category equivalence $(\lMod{A})_I\simeq\lMod{(A/I)}$.
\end{example}

\begin{example}
Suppose that every object in $\cC$ is of finite length.
Then the assumption in Example~\ref{ex:rad_is_ideal} is satisfied.
For an ideal functor $I=\mathrm{Rad}$, we have that
$\cC^{\mathrm{Rad}}=\{0\}$ (Nakayama's lemma)
and $\cC_{\mathrm{Rad}}$ consists of all semisimple objects in $\cC$.
\end{example}

Clearly the intersection of $\cC^I$ and $\cC_I$ is $\{0\}$.
In addition, there are no non-zero morphisms between objects in these categories.

\begin{lemma}
Let $X,Y,Z\in\cC$ and suppose that
$X$ is $I$-accessible, $Y$ is $I$-torsion-free and $Z$ is $I$-annihilated.
Then $\Hom_\cC(X,Z)=0$ and $\Hom_\cC(Z,Y)=0$.
\end{lemma}
\begin{proof}
Follows from
\begin{gather*}
\Hom_\cC(X,Z)=\Hom_\cC(X,T_I^\vee Z)\simeq\Hom_\cC(T_I X,Z)=0,\\
\Hom_\cC(Z,Y)=\Hom_\cC(T_I Z,Y)\simeq\Hom_\cC(Z,T_I^\vee Y)=0.
\qedhere
\end{gather*}
\end{proof}

The important property is that simple objects in $\cC$ are divided into these subcategories.
The proof of the lemma below is obvious.

\begin{lemma*}
When $X\in\cC$ is simple, these three conditions are all equivalent.
\begin{enumerate}
\item $X$ is $I$-accessible,
\item $X$ is $I$-torsion-free,
\item $X$ is not $I$-annihilated.
\qedhere
\end{enumerate}
\end{lemma*}

\begin{notation}
We denote by $\Irr\cC$
the isomorphism class of simple objects in $\cC$.
For an ideal functor $I$ on $\cC$,
we also denote by $\Irr\cC^I$ and $\Irr\cC_I$
the subsets of $\Irr\cC$ whose members are simple objects contained
in respective subcategories.
\end{notation}
By the lemma, we have a decomposition $\Irr\cC=\Irr\cC^I\sqcup\Irr\cC_I$.
Note that the definitions of $\Irr\cC^I$ and $\Irr\cC_I$
need both the category $\cC$ and
the ideal functor $I$,
not only the subcategories themselves.

\begin{proposition}\label{prop:limit_properties_of_subcats}
\ 
\begin{enumerate}
\item $I$-accessible objects are closed under
quotients, extensions and direct sums,
\item $I$-torsion-free objects are closed under
subobjects, extensions and direct products,
\item $I$-annihilated objects are closed under
subobjects, quotients, direct products and direct sums.
\end{enumerate}
In particular, $\cC^I$ is an exact subcategory of $\cC$ in Quillen's sense,
and $\cC_I$ is an abelian subcategory.
\end{proposition}
\begin{proof}
Follow from that $T_I$ is cocontinuous
and that $T_I^\vee$ is continuous.
\end{proof}
Obviously $\Irr\cC_I$ is equal to
the isomorphism class of simple objects in an abelian category $\cC_I$,
so this notation makes no confusions.
Let us denote by the embedding $\cC_I\hookrightarrow\cC$ of abelian category by
$\Phi_I$.
Namely, for $X\in\cC_I$, we explicitly write $\Phi_IX\in\cC$
when we need to emphasis the categories in which these objects belong.

\begin{lemma}
$\Phi_I$ has both the left adjoint functor $\Phi_I^\wedge$
and the right adjoint functor $\Phi_I^\vee$
such that $\Phi_I\Phi_I^\wedge=T_I$, $\Phi_I\Phi_I^\vee=T_I^\vee$
and $\Phi_I^\wedge\Phi_I=\Phi_I^\vee\Phi_I=\Id_{\cC_I}$.
\end{lemma}
\begin{proof}
Since $T_I X$ and $T_I^\vee X$ for any $X\in\cC$ are $I$-annihilated,
the functors $\Phi_I^\wedge$ and $\Phi_I^\vee$ can be defined
as the restriction of $T_I$ and $T_I^\vee$ respectively.
For any $Y\in\cC_I$, we have naturally
\begin{gather*}
\Hom_\cC(X,\Phi_I Y)=\Hom_\cC(X,T_I^\vee\Phi_IY)\simeq
\Hom_\cC(T_I X,\Phi_IY)=\Hom_{\cC_I}(\Phi_I^\wedge X,Y),\\
\Hom_\cC(\Phi_IY,X)=\Hom_\cC(T_I\Phi_IY,X)\simeq
\Hom_\cC(\Phi_IY,T_I^\vee X)=\Hom_{\cC_I}(Y,\Phi_I^\vee X).
\end{gather*}
Thus these two functors are respective adjoints of the embedding.
\end{proof}
\begin{corollary*}
For an ideal functor $J$ on another abelian category $\cD$,
the functor
\begin{align*}
\cHom(\cC_I,\cD_J)&\to\cHom(\cC,\cD),\\
F&\mapsto \Phi_JF\Phi_I^\wedge
\end{align*}
is fully faithful.
Its image is equivalent to the full subcategory
\[\set{G\colon\cC\to\cD}{G=T_JGT_I}\subset\cHom(\cC,\cD)\]
and the inverse is induced by
\begin{align*}
\cHom(\cC,\cD)&\to\cHom(\cC_I,\cD_J),\\
G&\mapsto \Phi_J^\wedge G\Phi_I.
\end{align*}
In particular, $F$ is an adjunction if and only if
so is $\Phi_JF\Phi_I^\wedge$.
\end{corollary*}

\begin{remark}
$\cC_I$ is characterized up to equivalence by these data:
let $\cE$ be an abelian category with an adjoint $\Phi\colon\cE\to\cC$
which also has a left adjoint functor $\Phi^\wedge$,
and suppose that the counit $\Phi^\wedge\Phi\to\Id_\cE$ is an isomorphism
and the unit $\Id_\cC\to\Phi\Phi^\wedge$ is epic.
Then $\cE$ is canonically equivalent to $\cC_I$
where $I\coloneqq\Ker(\Id_\cC\twoheadrightarrow\Phi\Phi^\wedge)$.
\end{remark}

\subsection{Ideal operations}
In this subsection we consider operations against ideal functors:
summation, product and quotient.
These are analogues of those against usual 2-sided ideals in rings.
Firstly we introduce summation of ideal functors.

\begin{proposition}
Let $\{I_i\}$ be a family of ideal functors on $\cC$,
and suppose that their sum $\sum_i I_i\subset\Id_\cC$
and the intersection $\bigcap_i T_{I_i}^\vee\subset\Id_\cC$ exist.
Then $\sum_i I_i$ is again an ideal functor.
\end{proposition}
\begin{proof}
The cokernel of $\sum_i I_i\hookrightarrow\Id_\cC$
is the pushout of adjunctions $\Id_\cC\twoheadrightarrow T_{I_i}$ under $\Id_\cC$,
which is left adjoint to the pullback $\bigcap_i T_{I_i}^\vee$
by Proposition~\ref{prop:colim_adj}.
\end{proof}

\begin{remark}
In contrast, the intersection of ideal functors is not an ideal functor.
This is because even in the module category of a ring,
the equation $(I\cap J)V=IV\cap JV$ does not hold in general.
\end{remark}

In particular, a finite sum always exists.
We can represent finite sum in another way as follows.

\begin{proposition}
Let $I$ and $J$ be ideal functors on $\cC$.
Then $T_IT_J=T_{I+J}$.
\end{proposition}
\begin{proof}
We have a commutative diagram
\[\begin{tikzcd}
0 \arrow{r} &
I \arrow{r} \arrow[two heads]{d} &
\Id_\cC \arrow{r} \arrow[two heads]{d} &
T_I \arrow{r} \arrow[two heads]{d} &
0 \\
0 \arrow{r} &
IT_J \arrow{r} &
T_J \arrow{r} &
T_IT_J \arrow{r} &
0
\end{tikzcd}\]
where the rows are exact since the right multiplication of $T_J$ is exact,
and $I\to IT_J$ is epic since $I$ preserves images.
These properties imply that the right square is cocartesian.
In other words, $T_IT_J$ is the pushout of $T_I$ and $T_J$ under $\Id_\cC$.
\end{proof}

Secondly we consider the product of two ideal functors defined by composition.

\begin{lemma}
For ideal functors $I$ and $J$ on $\cC$, we have $T_I J=JT_{IJ}$.
\end{lemma}
Here we let $T_{IJ}\coloneqq\Coker(IJ\hookrightarrow\Id_\cC)$
though we have not yet prove that $IJ$ is an ideal functor.
\begin{proof}
Let us consider the diagram
\[\begin{tikzcd}
0 \arrow{r} &
IJ \arrow{r} \arrow{d}{\id} &
J \arrow{r} \arrow[hook]{d} &
T_I J \arrow{r} \arrow[hook]{d} &
0 \\
0 \arrow{r} & IJ \arrow{r} & \Id_\cC \arrow{r} & T_{IJ} \arrow{r} & 0
\end{tikzcd}\]
whose rows are exact.
The commutative square induces $T_I J\to T_{IJ}$
which is monic by the four lemma.
On the other hand, we have another commutative diagram
\[\begin{tikzcd}
J \arrow[two heads]{r} \arrow[hook]{d} &
J T_{IJ} \arrow[hook]{d} \\
\Id_\cC \arrow[two heads]{r} &
T_{IJ}\rlap{.}
\end{tikzcd}\]
where $J\to JT_{IJ}$ is epic and $JT_{IJ}\to T_{IJ}$ is monic.
Thus we have two epi-mono factorization of the diagonal morphism
so that its images must be equal.
\end{proof}

\begin{lemma}
For $I$ and $J$ as above, we have $I^\circ J=JI^\circ$.
\end{lemma}
\begin{proof}
Follows in a similar way from that both are the image of $J\hookrightarrow\Id_\cC\twoheadrightarrow I^\circ$.
\end{proof}

\begin{proposition}
Let $I$ and $J$ be ideal functors on $\cC$.
Then $IJ$ is also an ideal functor such that $(IJ)^\circ=J^\circ I^\circ$.
\end{proposition}
\begin{proof}
Let $T_{IJ}\coloneqq\Coker(IJ\hookrightarrow\Id_\cC)$
and $T_{IJ}^\vee\coloneqq\Ker(\Id_\cC\twoheadrightarrow J^\circ I^\circ)$.
It suffices to prove that these functors actually form an adjunction.
By the lemma above, we have $IJ T_{IJ}=IT_IJ=0$.
So $JT_{IJ}$ is $I$-annihilated and this implies $JI^\circ T_{IJ}=I^\circ JT_{IJ}=0$.
Thus $I^\circ T_{IJ}$ is $J$-annihilated so $J^\circ I^\circ T_{IJ}=0$.
Equivalently, we have $T_{IJ}^\vee T_{IJ}=T_{IJ}$.
We can prove $T_{IJ}T_{IJ}^\vee=T_{IJ}^\vee$ in a similar way.
Using these isomorphisms, we can define the unit $\Id_\cC\twoheadrightarrow T_{IJ}=T_{IJ}^\vee T_{IJ}$
and the counit $T_{IJ}T_{IJ}^\vee=T_{IJ}^\vee\hookrightarrow\Id_\cC$.
Now it is obvious that these morphisms satisfy the zig-zag identities.
\end{proof}

Lastly we study about quotient of ideal functors.

\begin{definition}
Let $I$ and $J$ be two ideal functors on $\cC$ such that $J\subset I\subset\Id_\cC$.
Since $\cC_J$ is closed under subobjects, $X\in\cC_J$ implies $IX\in\cC_J$.
We denote this restricted functor of $I$ by
$I_J\colon\cC_J\to\cC_J$.
\end{definition}
Recall that we denote by $\Phi_J\colon\cC_J\to\cC$
the embedding of abelian category, and
the endofunctor category can be also exactly embedded as
\begin{align*}
\cEnd(\cC_J)&\to\cEnd(\cC),\\
F&\mapsto\Phi_J F\Phi^\wedge_J.
\end{align*}
\begin{lemma}
Let $I$ and $J$ be as above. Then
\[
\Phi_J I_J\Phi^\wedge_J=IT_J
=\Ker(T_J\twoheadrightarrow T_I)\simeq I/J.
\]
\end{lemma}
\begin{proof}
By definition, $\Phi_J I_J=I\Phi_J$ so $\Phi_J I_J\Phi^\wedge_J=IT_J$.
Since $T_I T_J=T_{I+J}=T_I$, it is equal to
\[IT_J=\Ker(T_J\twoheadrightarrow T_I T_J)=\Ker(T_J\twoheadrightarrow T_I).\]
The last isomorphism is obvious.
\end{proof}
\begin{proposition}
For $I$ and $J$ as above, $I_J$ is an ideal functor on $\cC_J$.
Moreover every ideal functors on $\cC_J$ is obtained in this way,
and ideal functors on $\cC_J$ are in one-to-one correspondence
with those on $\cC$ which contain $J$.
\end{proposition}
\begin{proof}
Via the above embedding, $\Coker(I_J\hookrightarrow\Id_{\cC_J})$ is mapped to
an adjunction $\Coker(IT_J\hookrightarrow T_J)=T_I$.
Thus $I_J$ is an ideal functor and $I$ is recovered from $I_J$ in this way.

Conversely, suppose that $K$ is an ideal functor on $\cC_J$.
Then $T_{\tilde K}\coloneqq\Phi_J T_K\Phi_J^\wedge$ is a quotient adjunction of $T_J$ so
$\tilde K\coloneqq\Ker(\Id_\cC\twoheadrightarrow T_J\twoheadrightarrow T_{\tilde K})$
is an ideal functor which contains $J$.
Moreover we have
\[\tilde K\Phi_J=\Ker(\Phi_J\twoheadrightarrow T_{\tilde K}\Phi_J)
=\Ker(\Phi_J\twoheadrightarrow \Phi_J T_K)=\Phi_J K,\]
which means that $\tilde K_J\simeq K$.
Thus these operations gives a one-to-one correspondence up to unique isomorphism.
\end{proof}

Now the next statements are clear.

\begin{proposition*}
Let $I,J\subset\Id_\cC$ be ideal functors. Then
\begin{enumerate}
\item $\cC_{I+J}=\cC_I\cap\cC_J$,
\item $\cC^{IJ}=\cC^I\cap\cC^J$,
\item if $J\subset I$ then $(\cC_J)_{I_J}=\cC_I$ and $(\cC_J)^{I_J}=\cC^I\cap\cC_J$.
\qedhere
\end{enumerate}
\end{proposition*}

\subsection{Compatibility with extension}
In this subsection we fix an abelian category $\cC$
and an ideal functor $I\subset\Id_\cC$.
For each pair of $X,Y\in\cC$, let us denote
$\Ext^i_\cC(X,Y)_0\coloneqq\Ext^i_{\cC_0}(X,Y)$ the Ext group taken in $\cC_0$.
It is defined as the set of equivalence classes of exact sequences
\[0\to Y\to E_i\to\dots\to E_1\to A\to 0.\]
We also define the graded Ext group $\Ext^i_\cC(X,Y)$ as
\[\Ext^i_\cC(X,Y)\coloneqq\Ext^i_{\cC_0}(X,Y)\oplus\Ext^i_{\cC_0}(X,\Pi Y)\]
for the super case $\cV=\cS$ and
\[\Ext^i_\cC(X,Y)\coloneqq\bigoplus_k\Ext^i_{\cC_0}(X,\Sigma^k Y)\]
for the graded case $\cV=\cG$.

Every exact sequences in $\cC_I$ are still exact in $\cC$,
so we have a canonical map \[\Ext^i_{\cC_I}(X,Y)\to\Ext^i_\cC(\Phi_I X,\Phi_I Y).\]
However this map is rarely an isomorphism because
when we take an exact sequence in $\cC$
whose both ends are in $\cC_I$,
the rest terms do not need to belong in $\cC_I$.
In this subsection we give some characterizations
of that $\Phi_I$ preserves $\Ext$ functors.
These results are a reformulation of those in \cite{AuslanderPlatzeckTodorov92}.

In a module category case,
the condition for $\Ext^1$ is well-known:
it is equivalent to that the ideal is idempotent.
We can easily generalize this fact as follows.
\begin{proposition}\label{prop:ext_1}
The followings are equivalent.
\begin{enumerate}
\item $\Ext^1_{\cC_I}(X,Y)\simeq\Ext^1_\cC(X,Y)$
for every $X,Y\in\cC_I$,
\item $\cC_I$ is closed under extensions,
\item $I^2=I$.
\end{enumerate}
\end{proposition}
\begin{proof}
(1) $\Leftrightarrow$ (2) is immediate by definition.
Now let $0\to X\to Y\to Z\to 0$ be a short exact sequence in $\cC$
and suppose $X,Z\in\cC_I$.
Then clearly $I^2Y=0$ so (3) implies (2).
Conversely assume that (3) fails so that $I^2\subsetneq I$.
Then there exists $X\in\cC$ such that
$I^2X\subsetneq IX$.
By $T_II=IT_{I^2}$ the sequence
\[0\to T_IIX\to T_{I^2}X\to T_IX\to0\]
is exact.
Both the left and the right term is in $\cC_I$ but the middle is not
since $I^2X\subsetneq IX$ implies $IT_{I^2}X=T_IIX\neq0$.
Hence (2) does not hold, so we have (2) $\Leftrightarrow$ (3).
\end{proof}

\begin{lemma}\label{lem:enough_proj}
Suppose that $\cC$ has enough projectives.
\begin{enumerate}
\item If $P\in\cC$ is projective, then so is $\Phi_I^\wedge P\in\cC_I$.
\item $\cC_I$ also has enough projectives.
\item Each projective object in $\cC_I$ is a direct summand
of $\Phi_I^\wedge P$ for some projective object $P\in\cC$.
\end{enumerate}
\end{lemma}
\begin{proof}
(1) follows from that its right adjoint $\Phi_I$ is exact.
For any $X\in\cC_I$, there is a projective object $P\in\cC$ and
an epimorphism $P\twoheadrightarrow\Phi_I X$ which induces
$\Phi_I^\wedge P\twoheadrightarrow X$ so (2) and (3) follow from (1).
\end{proof}

In the rest of this subsection, we require that
$\cC$ has enough projectives and injectives
in order to use its homological properties.
Then we have that $\Ext_\cC^i(X,{\bullet})$ and $\Ext_\cC^i({\bullet},Y)$ are
the $i$-th left derived functors of $\Hom_\cC(X,{\bullet})$
and $\Hom_\cC({\bullet},Y)$ respectively.

\begin{lemma}\label{lem:ext_preserving_X}
Let $2\le k\le\infty$.
For $X\in\cC$, the followings are equivalent.
\begin{enumerate}
\item $\Ext^i_{\cC_I}(\Phi_I^\wedge X,Y)\simeq\Ext^i_\cC(X,\Phi_I Y)$
for any $Y\in\cC_I$ and $0\le i<k$,
\item $\Ext^i_\cC(X,\Phi_I Q)=0$ for any injective $Q\in\cC_I$ and $1\le i<k$,
\item $(L_i\Phi_I^\wedge)X=0$ for any $1\le i<k$,
\item $(L_0I)X\simeq IX$ and $(L_iI)X=0$ for any $1\le i<k-1$.
\end{enumerate}
\end{lemma}
\begin{proof}
Let $P_k\to\dots\to P_1\to P_0\to X\to0$ be a projective resolution of $X$.
First (1) $\Rightarrow$ (2) is trivial.
Assume (2) then it implies
\[
0\to\Hom_{\cC_I}(\Phi_I^\wedge X,Q)\to
\Hom_{\cC_I}(\Phi_I^\wedge P_0,Q)\to\dots\to\Hom_{\cC_I}(\Phi_I^\wedge P_k,Q)\]
is exact for any injective $Q\in\cC_I$.
Since $\cC_I$ has enough injectives by the dual of
Lemma~\ref{lem:enough_proj}, the sequence
\[\Phi_I^\wedge P_k\to\dots\to\Phi_I^\wedge P_1\to\Phi_I^\wedge P_0\to\Phi_I^\wedge X\to0\]
must be exact so (3) holds.
Conversely, if (3) is satisfied, the sequence above is
a projective resolution of $\Phi_I^\wedge X$. Thus
\[\Ext^i_{\cC_I}(\Phi_I^\wedge X,Y)
\simeq H_i(\Hom_{\cC_I}(\Phi_I^\wedge P_i,Y))
\simeq H_i(\Hom_\cC(P_i,\Phi_I Y))
\simeq\Ext^i_\cC(X,\Phi_I Y)\]
so (1) holds.
Finally we have $\Phi_I(L_i\Phi_I^\wedge)=L_iT_I$ since $\Phi_I$ is exact.
Hence (3) and (4) are equivalent by that the sequence
\[0\to L_1T_I\to L_0I\to\Id_\cC\to T_I\to0\]
is exact and that $L_iT_I\simeq L_{i-1}I$.
\end{proof}

One can easily check that if every $X\in\cC$ satisfies the above conditions for $k=2$,
then it is also true for $k=\infty$. In this situation,
we can rewrite the conditions as follows.
\begin{corollary*}
The followings are equivalent.
\begin{enumerate}
\item $\Ext^i_{\cC_I}(\Phi_I^\wedge X,Y)\simeq\Ext^i_\cC(X,\Phi_I Y)$
for any $X\in\cC$, $Y\in\cC_I$ and $i\ge0$,
\item $\Phi_I$ sends injectives to injectives,
\item $\Phi_I^\wedge$ is exact,
\item $I$ is an adjunction.
\qedhere
\end{enumerate}
\end{corollary*}

However this condition is too strong for our purpose
because we only need objects of
the form $\Phi_IX\in\cC$ for $X\in\cC_I$.
Now we state a criteria of $\Ext$ preserving property for ideal functors.

\begin{proposition}\label{prop:ext_preserving}
Let $2\le k\le\infty$. Then the followings are equivalent.
\begin{enumerate}
\item $\Ext^i_{\cC_I}(X,Y)\simeq\Ext^i_\cC(\Phi_I X,\Phi_I Y)$ for any $X,Y\in\cC_I$
and $0\le i<k$,
\item $\Ext^i_\cC(\Phi_I P,\Phi_I Q)=0$ for any projective $P\in\cC_I$,
injective $Q\in\cC_I$ and $1\le i<k$,
\item $(L_i\Phi_I^\wedge)\Phi_I=0$ for any $1\le i<k$,
\item $(L_i\Phi_I^\wedge)\Phi_I P=0$ for any projective $P\in\cC_I$ and $1\le i<k$.
\item $(L_iI)\Phi_I=0$ for any $0\le i<k-1$,
\item $(L_iI)\Phi_I P=0$ for any projective $P\in\cC_I$ and $0\le i<k-1$.
\end{enumerate}
When $k\ge3$, it is also equivalent to that:
\begin{enumerate}\setcounter{enumi}{6}
\item $(L_0I)^2\simeq L_0I$, and $(L_iI)IP=0$ for any projective $P\in\cC$
and $1\le i<k-2$.
\end{enumerate}
Since the conditions (1) and (2) are self-dual, we can replace
the rest conditions by their dual statements.
\end{proposition}
\begin{proof}
(1) $\Leftrightarrow$ (3) $\Leftrightarrow$ (5) and
(2) $\Leftrightarrow$ (4) $\Leftrightarrow$ (6) follow from the previous lemma.
(3) $\Rightarrow$ (4) is obvious.
Now suppose (4). Let $X\in\cC$ and take an exact sequence
$0\to K\to P\to X\to0$ with $P$ projective.
By applying $(L_i\Phi_I^\wedge)\Phi_I$ it yields the long exact sequence
\[
\dotsb\to(L_1\Phi_I^\wedge)\Phi_I K\to
(L_1\Phi_I^\wedge)\Phi_I P\to
(L_1\Phi_I^\wedge)\Phi_I X\to
K\to P\to X\to0.
\]
Then the assumption implies $(L_1\Phi_I^\wedge)\Phi_I X=0$.
Moreover we have $(L_i\Phi_I^\wedge)\Phi_I X\simeq(L_{i-1}\Phi_I^\wedge)\Phi_I K$
for each $2\le i<k$ so by induction all of them must be zero.
Hence (4) implies (3).

Now suppose $k\ge3$ and we prove that (6) is also equivalent to (7).
First by Lemma~\ref{lem:enough_proj}, (6) can be replaced by
\begin{enumerate}
\item[(6')] $(L_iI)T_I P=0$ for any projective $P\in\cC$ and $0\le i<k-1$.
\end{enumerate}
Applying $L_i I$'s to the exact sequence $0\to IP\to P\to T_I P\to0$,
we obtain that it is equivalent to that $(L_0I)IP\simeq IP$ and
$(L_iI)IP=0$ for $1\le i<k-2$.
Moreover, the first condition is equivalent to that $(L_0I)^2\simeq L_0I$ since
a right exact functor is determined by the values on projectives
which generate whole $\cC$.
\end{proof}

Hence it also yields a well-known statement as a special case:
for an algebra $A$ and its 2-sided ideal $I\subset A$,
$\Ext^2_{A/I}=\Ext^2_A$ if and only if
the multiplication $I\otimes_A I\to I$ is an isomorphism.

\begin{example}
Suppose $I^2=I$ and $I$ sends projectives to projectives.
Then for $k=\infty$ the condition (7) above is easily verified.
\end{example}

\subsection{Ideal filters}
In this subsection we consider a family of ideal functors
indexed by a partially ordered set $(\Lambda,\le)$.
Such situation occurs mainly in the study of cellular algebras
or quasi-hereditary algebras.
In this subsection we fix an abelian category $\cC$
which is closed under sums and intersections of subobjects with cardinality $\#\Lambda$.

\begin{definition}\label{def:ideal_filter}
An \term{ideal filter} on $\cC$ indexed by $(\Lambda,\le)$
is a family of ideal functors $\{I^{\le\lambda}\}_{\lambda\in\Lambda}$
which satisfies these three conditions:
\begin{align*}
I^{\le\lambda}&\subset I^{\le\mu}\quad\text{if $\lambda\le\mu$},&
\Id_\cC&=\sum_{\lambda}I^{\le\lambda},&
I^{\le\lambda}I^{\le\mu}&\subset\sum_{\nu\le\lambda,\mu}I^{\le\nu}.
\end{align*}
\end{definition}

From now on $\{I^{\le\lambda}\}$ denotes an ideal filter on $\cC$
indexed by $(\Lambda,\le)$.

\begin{notation}
For each $\lambda\in\Lambda$,
we define \[I^{<\lambda}\coloneqq\sum_{\mu<\lambda}I^{\le\mu}.\]
When an ideal filter on $\cC$ is fixed, we write
$\cC^{\le\lambda}\coloneqq\cC^{I^{\le\lambda}}$ and
$\cC_{<\lambda}\coloneqq\cC_{I^{<\lambda}}$ for short.
We also denote $\cC[\lambda]\coloneqq\cC^{\le\lambda}\cap\cC_{<\lambda}$
and $\Irr\cC[\lambda]\coloneqq\set{V\in\Irr\cC}{V\in\cC[\lambda]}$.
\end{notation}

The purpose of introducing an ideal filter
is to divide the category into a small subcategories as follows.

\begin{lemma}
If $\lambda\neq\mu$ we have $\cC[\lambda]\cap\cC[\mu]=\{0\}$,
so that $\Irr\cC[\lambda]\cap\Irr\cC[\mu]=\varnothing$.
\end{lemma}
\begin{proof}
Suppose that $X\in\cC[\lambda]\cap\cC[\mu]$.
If $\lambda<\mu$, we have
\[X=I^{\le\lambda}X\subset I^{<\mu}X=0.\]
Otherwise $\lambda\not\le\mu$.
Then \[X=I^{\le\lambda}I^{\le\mu}X\subset\sum_{\nu\le\lambda,\mu}I^{\le\nu}X
\subset I^{<\lambda}X=0.\]
In either case, we have that $X=0$.
\end{proof}

A partially ordered set $(\Lambda,\le)$ is said to be
\term{well-founded} if its every non-empty subset
has a minimal element.
Under the axiom of choice, this property is equivalent
to the conditions below:
\begin{enumerate}
\item there are no infinite descending
chains $\lambda_1>\lambda_2>\dotsb$,
\item there is a well-ordering extension of $\le$.
\end{enumerate}

\begin{proposition}\label{prop:dividing_simples}
Suppose that $\Lambda$ is well-founded. Then
\[\Irr\cC=\bigsqcup_{\lambda\in\Lambda}\Irr\cC[\lambda].\]
\end{proposition}
\begin{proof}
Suppose that $X\in\cC$ is simple. Since we have
$0\neq X=\sum_\lambda I^{\le\lambda}X$,
the set $\set{\lambda\in\Lambda}{I^{\le\lambda}X\neq0}$ is non-empty.
By the assumption it has a minimal element $\lambda\in\Lambda$,
so that $X\in\cC[\lambda]$.
\end{proof}

In practice we can choose a partially ordered set from various choices
to obtain a same result.
First observe that we can remove redundant indices from $\Lambda$.

\begin{lemma}
Suppose that there is a subset $\Lambda_0\subset\Lambda$
such that $\Id_\cC=\sum_{\lambda\in\Lambda_0}I^{\le\lambda}$.
Let $\Lambda'\coloneqq\set{\mu\in\Lambda}{\exists\lambda\in\Lambda_0\text{ s.t.\ }
\mu\le\lambda}$ be the order ideal generated by $\Lambda_0$.
Then
\begin{enumerate}
\item $\{I^{\le\lambda}\}_{\lambda\in\Lambda'}$ is also an ideal filter
indexed by $(\Lambda',\le)$,
\item $I^{<\lambda}=I^{\le\lambda}$ unless $\lambda\in\Lambda'$, so that $\cC[\lambda]=\{0\}$.
\end{enumerate}
\end{lemma}
\begin{proof}
(1) is obvious.
Suppose $\lambda\notin\Lambda'$. Then
\[I^{\le\lambda}\subset\sum_{\mu\in\Lambda'}I^{\le\mu}I^{\le\lambda}\subset
\sum_{\mu\in\Lambda'}\sum_{\nu\le\lambda,\mu}I^{\le\nu}\subset
I^{<\lambda}\] so (2) holds.
\end{proof}

Next we consider how these subcategories will be affected
when we strengthen the order on $\Lambda$.

\begin{lemma}
Let $\unlhd$ be an extension of $\le$,
that is, another partial ordering on $\Lambda$
such that $\lambda\le\mu$ implies $\lambda\unlhd\mu$.
For each $\lambda\in\Lambda$, define
\[I^{\unlhd\lambda}\coloneqq\sum_{\mu\unlhd\lambda}I^{\le\mu}.\]
Then 
\begin{enumerate}
\item $\{I^{\unlhd\lambda}\}$ is also an ideal filter
indexed by $(\Lambda,\unlhd)$,
\item $\cC^{\le\lambda}\cap\cC_{<\lambda}=\cC^{\unlhd\lambda}\cap\cC_{\lhd\lambda}$,
\item there is a surjective morphism
$I^{\le\lambda}/I^{<\lambda}\twoheadrightarrow I^{\unlhd\lambda}/I^{\lhd\lambda}$.
\end{enumerate}
\end{lemma}
\begin{proof}
Let us check the conditions in the definition for $\{I^{\unlhd\lambda}\}$.
The first two clearly hold.
Since ideal functors commute with summation, we also have the third one
\begin{align*}
I^{\unlhd\lambda}I^{\unlhd\mu}=
\sum_{\nu\unlhd\lambda,\,\pi\unlhd\mu}I^{\le\nu}I^{\le\pi}\subset
\sum_{\nu\unlhd\lambda,\,\pi\unlhd\mu}\sum_{\rho\le\nu,\pi}I^{\le\rho}\subset
\sum_{\nu\unlhd\lambda,\,\pi\unlhd\mu}\sum_{\rho\unlhd\nu,\pi}I^{\unlhd\rho}=
\sum_{\rho\unlhd\lambda,\mu}I^{\unlhd\rho}.
\end{align*}
Thus $\{I^{\unlhd\lambda}\}$ is an ideal filter.

In order to prove (2), first note that
\[I^{\lhd\lambda}=\sum_{\mu\lhd\lambda}I^{\unlhd\mu}
=\sum_{\mu\lhd\lambda}\sum_{\nu\unlhd\mu}I^{\le\mu}=\sum_{\nu\lhd\lambda}I^{\le\nu}.\]
Suppose that $X\in\cC^{\le\lambda}\cap\cC_{<\lambda}$.
Since $I^{\le\lambda}\subset I^{\unlhd\lambda}$,
clearly $X\in\cC^{\unlhd\lambda}$.
Moreover,
\begin{align*}
I^{\lhd\lambda}X
=I^{\lhd\lambda}I^{\le\lambda}X
=\sum_{\mu\lhd\lambda}I^{\le\mu}I^{\le\lambda}X
\subset\sum_{\mu\lhd\lambda}\sum_{\nu\le\lambda,\mu}I^{\le\nu}X
\subset I^{<\lambda}X=0
\end{align*}
so $X\in\cC_{\lhd\lambda}$.
Conversely, suppose that $X\in\cC^{\unlhd\lambda}\cap\cC_{\lhd\lambda}$.
Then $I^{<\lambda}\subset I^{\lhd\lambda}$ immediately
implies that $X\in\cC_{<\lambda}$. We also have
\begin{align*}
X=I^{\unlhd\lambda}X=I^{\le\lambda}X+I^{\lhd\lambda}X=I^{\le\lambda}X,
\end{align*}
that is, $X$ is $I^{\le\lambda}$-accessible.
We can prove that $X$ is $I^{\le\lambda}$-torsion-free in a similar manner
so $X\in\cC^{\le\lambda}$. Thus
$\cC^{\le\lambda}\cap\cC_{<\lambda}=\cC^{\unlhd\lambda}\cap\cC_{\lhd\lambda}$.

Finally (3) follows from $I^{\le\lambda}\subset I^{\unlhd\lambda}$,
$I^{<\lambda}\subset I^{\lhd\lambda}$ and $I^{\unlhd\lambda}=I^{\le\lambda}+I^{\lhd\lambda}$.
\end{proof}

Hence the notation $\cC[\lambda]\coloneqq\cC^{\le\lambda}\cap\cC_{<\lambda}$
does not depend on taking extension of ordering.
Unfortunately, the functor $I^{\le\lambda}/I^{<\lambda}$
\emph{does} change by extension of ordering.
The condition for stability of this functor is described as follows.

\begin{definition}
For each $\lambda$, let \[I^{\not\ge\lambda}\coloneqq\sum_{\mu\not\ge\lambda}I^{\le\mu}.\]
An ideal filter $\{I^{\le\lambda}\}$ is said to be \term{rigid} if it satisfies
$I^{\le\lambda}\cap I^{\not\ge\lambda}=I^{<\lambda}$ for every $\lambda$.
\end{definition}

Note that the condition $I^{\le\lambda}\cap I^{\not\ge\lambda}\supset I^{<\lambda}$
is always satisfied.
Clearly if $\le$ is a total order then every ideal filter is rigid.

\begin{proposition}
Suppose $\{I^{\le\lambda}\}$ is rigid.
Then for any extension $\unlhd$ of $\le$,
\begin{enumerate}
\item $\{I^{\unlhd\lambda}\}$ is also rigid,
\item the canonical morphism $I^{\le\lambda}/I^{<\lambda}\twoheadrightarrow I^{\unlhd\lambda}/I^{\lhd\lambda}$ is an isomorphism.
\end{enumerate}
\end{proposition}
\begin{proof}
Note that $I^{\lhd\lambda}\subset I^{\not\unrhd\lambda}\subset I^{\not\ge\lambda}$.
These inclusions imply
\[I^{\unlhd\lambda}\cap I^{\not\unrhd\lambda}
=(I^{\le\lambda}+I^{\lhd\lambda})\cap I^{\not\unrhd\lambda}
=(I^{\le\lambda}\cap I^{\not\unrhd\lambda})+I^{\lhd\lambda}
\subset(I^{\le\lambda}\cap I^{\not\ge\lambda})+I^{\lhd\lambda}.
\]
By the assumption, the right hand side is equal to $I^{<\lambda}+I^{\lhd\lambda}=I^{\lhd\lambda}$, so $\{I^{\unlhd\lambda}\}$ is rigid.
Moreover
\[I^{\le\lambda}\cap I^{\lhd\lambda}\subset I^{\le\lambda}\cap I^{\not\ge\lambda}=I^{<\lambda}\]
so $I^{\le\lambda}\cap I^{\lhd\lambda}=I^{<\lambda}$.
Thus the morphism
$I^{\le\lambda}/I^{<\lambda}\twoheadrightarrow I^{\unlhd\lambda}/I^{\lhd\lambda}$
is an isomorphism since its kernel is
$(I^{\le\lambda}\cap I^{\lhd\lambda})/I^{<\lambda}=0$.
\end{proof}

When an ideal filter is rigid we have an additional result
on simple objects.

\begin{proposition}\label{prop:composition_factor}
Suppose $\{I^{\le\lambda}\}$ is rigid
and $\Lambda$ is well-founded.
Then in the Grothendieck group of finite length objects in $\cC$,
we have the equation
\[[X]=\sum_{\lambda\in\Lambda}[I^{\le\lambda}X/I^{<\lambda}X].\]
\end{proposition}
\begin{proof}
By the proposition above, by taking its extension
we may assume that $\le$ is a well-ordering.
Suppose $X\in\cC$ is of finite length.
Let us denote by $l(Y)$ the length of $Y\in\cC$.
Since $X=\sum_\lambda I^{\le\lambda}X$, the set
$\set{\lambda\in\Lambda}{l(I^{\le\lambda}X)\ge k}$ is not empty
for each $0\le k\le l(X)$.
Let $\lambda_k$ be its minimum element.
Then we have
\[I^{<\lambda_k}X=\begin{cases*}
0 & if $k=0$,\\
I^{\le\lambda_{k-1}}X & if $\lambda_k\neq\lambda_{k-1}$.
\end{cases*}\]
Hence by taking the composition series
\[0\subset I^{\le\lambda_0}X\subset I^{\le\lambda_1}X\subset\dots
\subset I^{\le\lambda_{l(X)}}X=X\]
we have
\[[X]=[I^{\le\lambda_0}X]+\sum_{1\le k\le l(X)}[I^{\le\lambda_k}X/I^{\le\lambda_{k-1}}X]
=\sum_{\lambda\in\Lambda_0}[I^{\le\lambda}X/I^{<\lambda}X]\]
where $\Lambda_0=\{\lambda_0,\lambda_1,\dots,\lambda_{l(X)}\}$
(overlapping elements are excluded).
It is also clear that $I^{<\lambda}X=I^{\le\lambda}X$ when $\lambda\notin\Lambda_0$,
so the statement holds.
\end{proof}

\section{Morita context between abelian categories}\label{sec:morita}
The classical Morita theory~\cite{Morita58}
treats a category equivalence between respective module categories of
two rings $A$ and $B$. It is performed
as a tensor functor $P\otimes_B{\bullet}\colon\lMod{B}\to\lMod{A}$
and a hom functor $\Hom_A(P,{\bullet})\colon\lMod{A}\to\lMod{B}$
by use of a progenerator $P$,
which is an $(A,B)$-bimodule such that finitely generated and projective
as both left and right modules.
To make this correspondence symmetric,
we can take a $(B,A)$-bimodule $P^\vee\coloneqq\Hom_A(P,A)$ and
rewrite $\Hom_A(P,{\bullet})\simeq P^\vee\otimes_A{\bullet}$.
A Morita context between rings is a weaker notion of Morita equivalence
consists of such pair $(P,P^\vee)$,
which still provides an equivalence between certain
full subcategories of the module categories.
We here introduce a more generalized notion,
a Morita context between two abelian categories.

\subsection{Morita context and its trace ideals}

\begin{definition}
Let $\cC$ and $\cD$ be abelian categories.
A \term{Morita context} between $\cC$ and $\cD$
is a pair of adjunctions $F\colon\cD\to\cC$ and $G\colon\cC\to\cD$
equipped with two degree-zero natural transformations
$\eta\colon F G\to \Id_\cC$ and
$\rho\colon G F\to \Id_\cD$ such that
$F\rho=\eta F$ as morphisms $FGF\rightrightarrows F$ and $\rho G=G\eta$
as $GFG\rightrightarrows G$.
These equations are called the associativity laws.
\end{definition}

When $\cC=\lMod{A}$ and $\cD=\lMod{B}$
are respectively the module categories of algebras $A$ and $B$,
the above definition of Morita context between $\cC$ and $\cD$
coincides with Definition~\ref{def:Morita_between_algebras}
of that between $A$ and $B$ we introduced in the introduction.

\begin{remark}
Iglesias and Torrecillas~\cite{IglesiasTorrecillas95,IglesiasTorrecillas98}
has defined a more general notion called \term{wide (right) Morita context}.
They only required that $F$ and $G$ are right exact.
\end{remark}

For a while we fix a Morita context $(F,G)$
between $\cC$ and $\cD$ as above.
\begin{notation}
We denote by
\begin{align*}
\bar\eta&\colon G\to F^\vee,&
\bar\rho&\colon F\to G^\vee,\\
\eta^\vee&\colon\Id_\cC\to G^\vee F^\vee,&
\rho^\vee&\colon\Id_\cD\to F^\vee G^\vee
\end{align*}
the morphisms induced by adjunctions.
Let $D$ and $D'$ be the
images of $\bar\eta\colon G\to F^\vee$ and
$\bar\rho\colon F\to G^\vee$ respectively.
\end{notation}

The functors $D\colon\cC\to\cD$ and $D'\colon\cD\to\cC$
are called \term{Morita context functors}.
Similarly as ideal functors,
a Morita context functor is not left nor right exact.
However it has a following property again.
\begin{lemma}
Morita context functors preserve all images.
\end{lemma}
\begin{proof}
$D$ is both mono and epi
since it is a subobject of a left exact functor $F^\vee$ as well as
a quotient of a right exact functor $G$.
\end{proof}

Let $I\subset\Id_\cC$ and $J\subset\Id_\cD$ be
the images of $\eta\colon FG\to\Id_\cC$
and $\rho\colon GF\to\Id_\cD$ respectively.
These are ideal functors on the respective categories
by Proposition~\ref{prop:image_of_adjunction_is_ideal},
which we call \term{trace ideals}.
First we study how these functors act on
the subcategories defined by these ideal functors $I$ and $J$.

\begin{lemma}
Suppose $X\in\cC$ and consider three morphisms
$\eta X\colon FGX\to X$, $\bar\eta X\colon GX\to F^\vee X$
and $\eta^\vee X\colon X\to G^\vee F^\vee X$.
\begin{enumerate}
\item $X$ is $I$-accessible if and only if $\eta X$ is epic.
\item $X$ is $I$-torsion-free if and only if $\eta^\vee X$ is monic.
\item $X$ is $I$-annihilated if and only if
$\eta X=0$ (equivalently, $\bar\eta X=0$ or $\eta^\vee X=0$).
In particular, it is also equivalent to that $DX=\Image(\bar\eta X)=0$.
\end{enumerate}
\end{lemma}
\begin{proof}
Obvious by definition.
\end{proof}

\begin{lemma}\label{lem:morita_funct_on_subcats}
Suppose $X\in\cC$.
\begin{enumerate}
\item If $X$ is $I$-accessible
then $GX$ is $J$-accessible.
\item If $X$ is $I$-torsion-free
then $F^\vee X$ is $J$-torsion-free.
\item If $X$ is $I$-annihilated
then both $GX$ and $F^\vee X$ are $J$-annihilated.
\end{enumerate}
In particular, so is $DX$ in each cases.
\end{lemma}
\begin{proof}
Since $\rho GX\colon GFGX\to GX$
is equal to $G\eta X$ and $G$ is right exact,
if $\eta X$ is epic then so is $\rho GX$.
This means that if $X$ is $I$-accessible then $GX$ is $J$-accessible
by the lemma above.
(2) and (3) can be proven in a similar manner.
The last statement follows from that these properties are inherited
to subobjects or quotients.
\end{proof}

\begin{lemma}\label{lem:rest_is_annihilated}
$\Coker(DX\hookrightarrow F^\vee X)$ is $J$-annihilated for any $X\in\cC$.
\end{lemma}
\begin{proof}
Let $C\coloneqq\Coker(D\hookrightarrow F^\vee)=\Coker(G\to F^\vee)$
and consider the commutative diagram
\[\begin{tikzcd}
GFG \arrow{r} \arrow{d} &
GFF^\vee \arrow{r} \arrow{d}{\rho F^\vee} \arrow{ld} &
GFC \arrow{r} \arrow{d}{\rho C} &
0 \\
G \arrow{r}{\bar\eta} &
F^\vee \arrow{r} &
C \arrow{r} &
0\rlap{.}
\end{tikzcd}\]
Its rows are exact since $GF$ is right exact.
Since $\rho F^\vee\colon GFF^\vee\to F^\vee$ factors through $G$
by the associativity on $FGF$,
the induced morphism $\rho C\colon GFC\to C$ is zero.
In other words, $JC=0$.
\end{proof}

\begin{proposition}
If $X\in\cC$ is $I$-accessible, $DX$ is the unique largest
$J$-accessible subobject of $F^\vee X$.
\end{proposition}
\begin{proof}
Let $Y\subset F^\vee X$ be $J$-accessible.
Then by Proposition~\ref{prop:limit_properties_of_subcats},
$Y'\coloneqq DX+Y$ and $Y'/DX$ are also $J$-accessible.
On the other hand, by the lemma above $Y'/DX\subset F^\vee X/DX$
must be $J$-annihilated too.
Hence we conclude that $Y'=DX$, that is, $Y\subset DX$.
\end{proof}

\subsection{Category equivalence}

The first remarkable result which Morita context brings
is the equivalence of categories between respective subcategories
defined by ideal functors.
\begin{theorem}\label{thm:morita_cat_equiv}
$D$ and $D'$ induce
a category equivalence $\cC^I\simeq\cD^J$.
\end{theorem}

To prove this theorem, first we list several endofunctors on $\cC$
into a diagram.

\begin{lemma}
Consider the following epi-mono factorizations
\[\eta\colon FG\twoheadrightarrow I\hookrightarrow\Id_\cC,\qquad
\bar\rho G\colon FG\twoheadrightarrow D'G\hookrightarrow G^\vee G,\qquad
\bar\rho D\colon FD\twoheadrightarrow D'D\hookrightarrow G^\vee D.\]
These epimorphisms factor through
\[FG\twoheadrightarrow FD\twoheadrightarrow I
\twoheadrightarrow D'G\twoheadrightarrow D'D.\]
Dually, monomorphisms $\Id_\cC\hookrightarrow G^\vee F^\vee$,
$D'F\hookrightarrow G^\vee F^\vee$ and $D'D\hookrightarrow G^\vee D$
factor through
\[D'D\hookrightarrow D'F^\vee\hookrightarrow I^\circ
\hookrightarrow G^\vee D\hookrightarrow G^\vee F^\vee.\]
These chains of morphisms fit into the commutative diagram
\[\begin{tikzcd}
FG \arrow[two heads]{r} & FD \arrow[two heads]{r} \arrow[two heads]{ldd} \arrow{d} &
I \arrow[two heads]{r} \arrow[hook]{d} & D'G \arrow[two heads]{r} \arrow[hook]{d} & D'D \arrow[hook]{ldd} \\
& FF^\vee \arrow{r} \arrow[two heads]{d} & \Id_\cC \arrow{r} \arrow[two heads]{d} & G^\vee G \arrow{d} \\
D'D \arrow[hook]{r} & D'F^\vee \arrow[hook]{r} &
I^\circ \arrow[hook]{r} & G^\vee D \arrow[hook]{r} & G^\vee F^\vee\rlap{.}
\end{tikzcd}\]
\end{lemma}
\begin{proof}
First the morphisms $FG\twoheadrightarrow FD$ and $D'G\twoheadrightarrow D'D$
at both ends are induced by
$\bar\eta\colon G\twoheadrightarrow D\hookrightarrow F^\vee$.
Since the functors $F$ and $D'$ are both epi, these morphisms are epic.
Now consider the diagram
\[\begin{tikzcd}
\ & FG \arrow[two heads]{ld} \arrow[two heads]{d} \arrow[two heads]{rd} \\
FD \arrow{d} & I \arrow[hook]{d} & D'G \arrow[hook]{d} \\
FF^\vee \arrow{r} & \Id_\cC \arrow{r} & G^\vee G\rlap{.}
\end{tikzcd}\]
The right pentagon is commutative by the help of the associativity
on $GFG$ while the commutativity of the left one is trivial.
Since $I$ and $D'G$ are the images of the respective pentagons,
there exist unique morphisms $FD\twoheadrightarrow I\twoheadrightarrow D'G$
which make the diagram commutes.
Now it is left us to check the commutativity for $\bar\rho D$.
We have the diagram
\[\begin{tikzcd}
FG \arrow[two heads]{r} \arrow[two heads]{d} &
D'G \arrow[two heads]{d} \\
FD \arrow[two heads]{r} \arrow[two heads]{ru} &
D'D
\end{tikzcd}\]
where the outer square trivially commutes.
Since $FG\twoheadrightarrow FD$ is epic and the upper triangle commutes,
the lower also does.
The dual statement goes similarly and
the last commutativity has been already proven.
\end{proof}

\begin{corollary*}
$D'D$ is equal to
the image of the composite
$I\hookrightarrow\Id_\cC\twoheadrightarrow I^\circ$.
In particular, if $X\in\cC^I$ then canonically $X\simeq D'DX$.
\end{corollary*}

Putting this corollary and Lemma~\ref{lem:morita_funct_on_subcats} together,
we obtain Theorem~\ref{thm:morita_cat_equiv}.

\begin{remark}
Though the categories $\cC^I$ and $\cD^J$ are equivalent,
their exact structures may differ.
For example, let $A$ and $B$ be the upper triangle matrix algebras
\begin{align*}
A&\coloneqq\left\{\begin{pmatrix}*&*&*\\0&*&*\\0&0&*\end{pmatrix}\right\},&
B&\coloneqq\left\{\begin{pmatrix}*&*\\0&*\end{pmatrix}\right\}
\end{align*}
over $\kk$. Let $M$ and $N$ be the bimodules
\begin{align*}
M&\coloneqq\left\{\begin{pmatrix}*&*\\0&*\\0&*\end{pmatrix}\right\},&
N&\coloneqq\left\{\begin{pmatrix}*&*&*\\0&0&*\end{pmatrix}\right\}
\end{align*}
and define
$\eta\colon M\otimes_B N\to A$ and $\rho\colon N\otimes_A M\to B$
by matrix multiplication.
These data define a Morita context between
$A\text{-}\Mod$ and $B\text{-}\Mod$.
$\rho$ is surjective and the image of $\eta$ is
\[I\coloneqq\left\{\begin{pmatrix}*&*&*\\0&0&*\\0&0&*\end{pmatrix}\right\}\subset A,\]
so the Morita context functors induce a category equivalence
$(A\text{-}\Mod)^I\simeq B\text{-}\Mod$.
However it sends a short exact sequence
\[0\longrightarrow
\left\{\begin{pmatrix}*\\0\end{pmatrix}\right\}\longrightarrow
\left\{\begin{pmatrix}{*}\\{*}\end{pmatrix}\right\}\longrightarrow
\left\{\begin{pmatrix}{-}\\{*}\end{pmatrix}\right\}\longrightarrow
0\]
in $B\text{-}\Mod$ to the sequence
\[0\longrightarrow
\left\{\begin{pmatrix}*\\0\\0\end{pmatrix}\right\}\longrightarrow
\left\{\begin{pmatrix}{*}\\{*}\\{*}\end{pmatrix}\right\}\longrightarrow
\left\{\begin{pmatrix}{-}\\{-}\\{*}\end{pmatrix}\right\}\longrightarrow
0\]
in $A\text{-}\Mod$, which is obviously not exact at the middle term.
\end{remark}

As we have seen in this remark, the category equivalence does not
preserve extensions in general.
However, it is true if one of the categories is semisimple.
\begin{lemma}\label{lem:morita_with_semisimple}
If $\cD$ is semisimple, then $\Ext^1_\cC(X,Y)=0$ for any $X,Y\in\cC^I$.
\end{lemma}
\begin{proof}
Let $0\to Y\to E\to X\to 0$
be a short exact sequence in $\cC$.
Since $\cC^I$ is closed under extensions, $E$ is also in $\cC^I$.
Now $\cC^I\simeq\cD^J$ is semisimple, so that this sequence splits.
\end{proof}

On the other hand, the Ext preserving property for the other category
is induced from the following condition.
\begin{lemma}
Suppose that $\rho\colon GF\to\Id_\cD$ is surjective.
Then $\Ext^1_\cC(X,Y)\simeq\Ext^1_{\cC_I}(X,Y)$ for any $X,Y\in\cC_I$.
\end{lemma}
\begin{proof}
By the assumption $FGFG\to FG$ is also surjective.
This implies $I^2=I$, so that we can use Proposition~\ref{prop:ext_1}.
\end{proof}

\subsection{Correspondence on simple objects}

Next we prove the correspondence between the simple objects
in the respective subcategories.

\begin{theorem}\label{thm:morita_simple_1:1}
$D$ and $D'$ induce
a one-to-one correspondence $\Irr\cC^I\onetoone\Irr\cD^J$.
\end{theorem}

Though we have proven the category equivalence $\cC^I\simeq\cD^J$,
we have to prove this theorem independently since
we do not know how to characterize the set $\Irr\cC^I$ from
the category $\cC^I$ itself.
Actually, using Lemma~\ref{lem:morita_funct_on_subcats} again,
this theorem is obtained as an immediate corollary
of the next theorem.

\begin{theorem}
Let $X\in\Irr\cC^I$.
Then $DX$ is the simple socle of $F^\vee X$ as well as the simple top of $GX$.
\end{theorem}
\begin{proof}
By the assumption $X\notin\cC_I$, we have $DX\neq0$.
Take any non-zero subobject $Y\hookrightarrow F^\vee X$.
Then the corresponding morphism $FY\to X$ is also non-zero,
so it must be epic since $X$ is simple.
Now consider the commutative diagram
\[\begin{tikzcd}
GFY \arrow{r} \arrow[bend left=20,two heads]{rr} \arrow[two heads]{d} &
GFF^\vee X \arrow{r} \arrow{d} \arrow{rd}{\rho F^\vee X} &
GX \arrow{d}{\bar\eta X} \\
JY \arrow[hook]{r} & JF^\vee X \arrow[hook]{r} & F^\vee X\rlap{.}
\end{tikzcd}\]
Since $G$ is right exact $GFY\to GX$ is also epic.
So we have
\[JY=\Image(GFY\to F^\vee X)=\Image(GX\to F^\vee X)=DX.\]
This implies that $DX$ is contained in an arbitrary non-zero subobject
$Y\hookrightarrow F^\vee X$,
so it must be a simple socle of $F^\vee X$.
Dually it is also a simple top of $GX$.
\end{proof}

Putting it together with Lemma~\ref{lem:morita_funct_on_subcats}
and Lemma~\ref{lem:rest_is_annihilated}
we obtain the next corollary.
\begin{corollary*}\label{cor:multiplicity_on_Morita_context}
Let $X\in\Irr\cC$ and $Y\in\Irr\cD^J$. Then
\[[GX:Y]=[F^\vee X:Y]=\begin{cases*}
1 & $Y\simeq DX$,\\
0 & otherwise.
\end{cases*}
\qedhere\]
\end{corollary*}
Here $[M:S]$ is the multiplicity of a simple object $S$
in the composition factors of $M$.
If $M$ is not of finite length this symbol does not make sense in general,
but the formula above can be always read
in an appropriate manner.

When we replace $X$ above to
its injective hull or its projective cover,
we obtain similar statements.
\begin{proposition}\label{prop:projective_cover}
Let $X\in\Irr\cC^I$ and
suppose that it has a projective cover $P\twoheadrightarrow X$.
Then $GP\twoheadrightarrow GX\twoheadrightarrow DX$
is the top of $GP$.
\end{proposition}
\begin{proof}
Take $Y\in\Irr\cD$ and let $C\coloneqq\Coker(D'Y\hookrightarrow G^\vee Y)$.
By the projectiveness of $P$, the sequence
\[0\to\Hom_\cC(P,D'Y)\to\Hom_\cC(P,G^\vee Y)\to\Hom_\cC(P,C)\to0\]
is exact. By Lemma~\ref{lem:rest_is_annihilated},
$C$ is $I$-annihilated.
Hence it has no subquotients isomorphic to $X$,
so $\Hom_\cC(P,C)=0$ by a property of projective cover. Thus
\[
\Hom_\cD(GP,Y)\simeq\Hom_\cC(P,G^\vee Y)\simeq\Hom_\cC(P,D'Y).
\]
Now $D'Y$ is simple or zero, so
there is a non-zero morphism $GP\twoheadrightarrow Y$
if and only if $D'Y\simeq X$,
or equivalently, $Y\simeq DX$. Moreover
\[
\Hom_\cD(GP,DX)\simeq\Hom_\cC(P,X)\simeq\End_\cC(X)\simeq\End_\cD(DX).
\]
Thus $GP\twoheadrightarrow DX$ is the unique its simple quotient.
\end{proof}

Remark that $\Ker(GP\twoheadrightarrow DP)$
is $J$-annihilated by the dual of Lemma~\ref{lem:rest_is_annihilated},
but in general $\Ker(DP\twoheadrightarrow DX)$ is not,
so may contains a composition factor in $\cC^J$.

\subsection{Morita context among multiple categories}
We here generalize the notion of Morita context,
from that between two categories to that among more than two categories.
Let us take an index set $\Lambda$ which is not necessarily finite.
We assume that every category appears in this subsection
is closed under sums and intersections with
cardinality $\#\Lambda$.

\begin{definition}
Let $\{\cC_\lambda\}_{\lambda\in\Lambda}$ be a family of abelian categories
indexed by a set $\Lambda$.
A \term{Morita context} among $\{\cC_\lambda\}$
is a family of
adjunctions $F_{\lambda\mu}\colon\cC_\mu\to\cC_\lambda$
indexed by a pair of $\lambda,\mu\in\Lambda$,
equipped with a family of degree-zero natural transformations
$\eta_{\lambda\mu\nu}\colon F_{\lambda\mu}F_{\mu\nu}\to F_{\lambda\nu}$
indexed by a triple of $\lambda,\mu,\nu\in\Lambda$
which satisfies the following conditions.
\begin{enumerate}
\item (The associativity law)
For each $\lambda,\mu,\nu,\pi\in\Lambda$,
the square
\[\begin{tikzcd}
F_{\lambda\mu}F_{\mu\nu}F_{\nu\pi} \arrow{r}{\eta_{\lambda\mu\nu}F_{\nu\pi}} \arrow{d}[swap]{F_{\lambda\mu}\eta_{\mu\nu\pi}} &
F_{\lambda\nu}F_{\nu\pi} \arrow{d}{\eta_{\lambda\nu\pi}} \\
F_{\lambda\mu}F_{\mu\pi} \arrow{r}[swap]{\eta_{\lambda\mu\pi}} &
F_{\lambda\pi}
\end{tikzcd}\]
commutes.
\item (The unit law)
For each $\lambda$, there is a fixed isomorphism $F_{\lambda\lambda}\simeq\Id_{\cC_\lambda}$
such that $\eta_{\lambda\lambda\mu}$ and $\eta_{\lambda\mu\mu}$ are respectively equal to
\begin{align*}
F_{\lambda\lambda}F_{\lambda\mu}&\simeq\Id_{\cC_\lambda}F_{\lambda\mu}\simeq F_{\lambda\mu},&
F_{\lambda\mu}F_{\mu\mu}&\simeq F_{\lambda\mu}\Id_{\cC_\mu}\simeq F_{\lambda\mu}.
\end{align*}
\end{enumerate}
\end{definition}

One can easily verify that when $\#\Lambda=2$
this definition is equivalent to the previous one.

\begin{remark}
Let $\mathbf{A}$ be a 2-category which consists of
abelian categories as 0-cells, adjunctions as 1-cells and
natural transformation as 2-cells.
Consider $\Lambda$ as a codiscrete category,
that is, we regard that there exists a unique morphism
$\mu\to\lambda$ for each $\lambda,\mu\in\Lambda$.
Then a Morita context is just a lax functor
$\mathcal{F}\colon\Lambda\to\mathbf{A}$
(where $\cC_\lambda=\mathcal{F}(\lambda)$, $F_{\lambda\mu}=\mathcal{F}(\mu\to\lambda)$)
such that the unit $\Id_{\mathcal{F}(\lambda)}\to\mathcal{F}(\lambda\to\lambda)$
is an isomorphism for every $\lambda\in\Lambda$.
\end{remark}

\begin{example}\label{ex:category_is_morita}
Let $\cA$ be a category,
and take an object $X_\lambda\in\cA$ for each $\lambda$.
Let $A_\lambda\coloneqq\End_{\cA}(X_\lambda)$ be its endomorphism algebra.
Then for each pair of $\lambda,\mu$,
$\Hom_\cA(X_\mu,X_\lambda)$ is a $(A_\lambda,A_\mu)$-bimodule
so it induces an adjunction $\lMod{A_\mu}\to\lMod{A_\lambda}$.
Moreover the composition of morphisms
\[\Hom_\cA(X_\mu,X_\lambda)\otimes\Hom_\cA(X_\nu,X_\mu)
\to\Hom_\cA(X_\nu,X_\lambda)\]
gives a natural transformations between these adjunctions
which is associative and unital.
Hence these define a Morita context among the categories $\{\lMod{A_\lambda}\}$.
Conversely all Morita contexts among module categories
are obtained in this way.
\end{example}
\begin{example}
More generally, take a small full subcategory
$\cA_\lambda\subset\cA$ for each $\lambda$.
Then a collection of an $(\cA_\lambda,\cA_\mu)$-module
\begin{align*}
\cA_\mu^\op\boxtimes\cA_\lambda&\to\cV\\
X\boxtimes Y&\mapsto\Hom_\cA(X,Y)
\end{align*}
also defines a Morita context among $\{\lMod{\cA_\lambda}\}$.
\end{example}

Suppose that $\{F_{\lambda\mu}\}_{\lambda,\mu\in\Lambda}$ is
a Morita context among categories $\{\cC_\lambda\}_{\lambda\in\Lambda}$.
For each triple of $\alpha,\lambda,\mu\in\Lambda$,
let $I^\alpha_{\lambda\mu}$
be a subfunctor of $F_{\lambda\mu}$ defined by
\[I^\alpha_{\lambda\mu}\coloneqq
\Image(\eta_{\lambda\alpha\mu}\colon F_{\lambda\alpha}F_{\alpha\mu}\to F_{\lambda\mu}).\]
In particular, $I^\alpha_{\lambda\lambda}\subset F_{\lambda\lambda}\simeq\Id_{\cC_\lambda}$
is an ideal functor on $\cC_\lambda$.
The unit law implies that $I^\lambda_{\lambda\mu}=I^\mu_{\lambda\mu}=F_{\lambda\mu}$.
By the associativity law there are natural transformations
\[I^\alpha_{\lambda\mu}F_{\mu\nu}\to I^\alpha_{\lambda\nu}
\quad\text{and}\quad
F_{\lambda\mu}I^\beta_{\mu\nu}\to I^\beta_{\lambda\nu}\]
induced by $\eta_{\lambda\mu\nu}$.
Since $F_{\lambda\mu}$ is right exact, they induce
\[(F_{\lambda\mu}/I^\alpha_{\lambda\mu})(F_{\mu\nu}/I^\beta_{\mu\nu})\to
F_{\lambda\nu}/(I^\alpha_{\lambda\nu}+I^\beta_{\lambda\nu}).\]

Now take a subset $\Lambda'\subset\Lambda$.
Then clearly the restriction $\{F_{\lambda\mu}\}_{\lambda,\mu\in\Lambda'}$
gives a Morita context among the subcollection $\{\cC_\lambda\}_{\lambda\in\Lambda'}$.
In contrast, we can also take a ``quotient'' of this Morita context
with respect to $\Lambda'$ as follows.
\begin{proposition}
For each $\lambda,\mu\in\Lambda$, let
\[I'_{\lambda\mu}\coloneqq\sum_{\alpha\in\Lambda'}I^\alpha_{\lambda\mu}\]
and $\cC'_\lambda\coloneqq(\cC_\lambda)_{I'_{\lambda\lambda}}$.
Then there exists a Morita context $\{F'_{\lambda\mu}\}$
among the abelian categories $\{\cC'_\lambda\}$ defined by
\[F'_{\lambda\mu}\coloneqq
\Phi^\wedge_{I'_{\lambda\lambda}}(F_{\lambda\mu}/I'_{\lambda\mu})
\Phi_{I'_{\mu\mu}}.\]
\end{proposition}
\begin{proof}
Since $F_{\lambda\mu}$ is cocontinuous,
by taking colimits of natural transformation above we obtain
\[(F_{\lambda\mu}/I'_{\lambda\mu})(F_{\mu\nu}/I'_{\mu\nu})\to
F_{\lambda\nu}/I'_{\lambda\nu}.\]
Moreover, by the unit law we have
$T_{I'_{\lambda\lambda}}(F_{\lambda\mu}/I'_{\lambda\mu})T_{I'_{\mu\mu}}
=F_{\lambda\mu}/I'_{\lambda\mu}$.
Thus there are a natural transformation
\[\eta'_{\lambda\mu\nu}\colon F'_{\lambda\mu}F'_{\mu\nu}
=\Phi^\wedge_{I'_{\lambda\lambda}}(F_{\lambda\mu}/I'_{\lambda\mu})
(F_{\mu\nu}/I'_{\mu\nu})\Phi_{I'_{\nu\nu}}\to
\Phi^\wedge_{I'_{\lambda\lambda}}(F_{\lambda\mu}/I'_{\lambda\nu})\Phi_{I'_{\nu\nu}}
=F'_{\lambda\nu}\]
and an isomorphism
\[F'_{\lambda\lambda}\simeq
\Phi^\wedge_{I'_{\lambda\lambda}}T_{I'_{\lambda\lambda}}\Phi_{I'_{\lambda\lambda}}
=\Id_{\cC_\lambda}\]
which form a Morita context.
\end{proof}

Note that $\cC'_\alpha=\{0\}$ for every $\alpha\in\Lambda'$,
so the quotient Morita context above should be considered
as parameterized by the complement set $\Lambda\setminus\Lambda'$.
When $\Lambda'$ has a decomposition $\Lambda'=\Lambda'_1\sqcup\Lambda'_2$,
taking the quotient by $\Lambda'$ is equal to
first taking by $\Lambda'_1$, then by $\Lambda'_2$.

\subsection{Morita context with a partial order}
As a special case of quotient,
let us consider the case that $\Lambda'$ in the previous subsection
consists of a single element $\alpha$.
For each $\lambda\in\Lambda\setminus\{\alpha\}$,
the pair $(F_{\lambda\alpha},F_{\alpha\lambda})$ is a Morita context between
two categories $\cC_\lambda$ and $\cC_\alpha$.
Hence we have a category equivalence
\[(\cC_\lambda)^{I^\alpha_{\lambda\lambda}}
\simeq(\cC_\alpha)^{I^\lambda_{\alpha\alpha}}\]
and a one-to-one correspondence
\[\Irr(\cC_\lambda)^{I^\alpha_{\lambda\lambda}}
\onetoone\Irr(\cC_\alpha)^{I^\lambda_{\alpha\alpha}}.\]
In practice we should choose $\alpha$ such that the structure of $\cC_\alpha$ is very simple
so that we can describe a part of $\cC_\lambda$, which may be hard to study,
by terms of $\cC_\alpha$.
Now on the collection of the rest part
$\cC'_\lambda=(\cC_\lambda)_{I^\alpha_{\lambda\lambda}}$ we have
a new Morita context, so we can recursively continue
this process for $\cC'_\lambda$ by choosing another $\beta\in\Lambda$
to decompose $\cC_\lambda$ into small parts.
In order to perform this strategy at one time,
we introduce a partial order on the set $\Lambda$
as we do before in the previous subsection.
Intuitively it indicates the order of $\alpha,\beta,\dotsc$
we pick up from $\Lambda$.

\begin{definition}
Let $\{F_{\lambda\mu}\}$ be a Morita context among
the categories $\{\cC_{\lambda}\}$.
A partial order $\le$ on the set $\Lambda$
is said to be \term{compatible} with $\{F_{\lambda\mu}\}$ if it satisfies
\[F_{\lambda\mu}=\sum_{\nu\le\lambda,\mu}I^\nu_{\lambda\mu},
\quad\text{where}\quad
I^\nu_{\lambda\mu}\coloneqq
\Image(F_{\lambda\nu}F_{\nu\mu}\to F_{\lambda\mu})\]
for each pair of $\lambda,\mu\in\Lambda$.
\end{definition}
When $\lambda$ and $\mu$ are comparable then the condition above is trivially satisfied.
Hence every total order on $\Lambda$ is compatible.

\begin{lemma}\label{lem:well_founded_Morita}
If $\Lambda$ is well-founded, then the condition above is equivalent
to that
\[F_{\lambda\mu}=\sum_{\nu<\lambda}I^\nu_{\lambda\mu}\]
is satisfied for each pair of $\lambda,\mu\in\Lambda$ such that $\lambda\not\leq\mu$.
\end{lemma}
\begin{proof}
Clearly the first condition implies the second.
Suppose the second one.
We prove the first condition for a fixed $\mu$
by transfinite induction on $\lambda$. So assume that for every $\nu<\lambda$
we have $F_{\nu\mu}=\sum_{\pi\le\nu,\mu}I^\pi_{\nu\mu}$.
If $\lambda\le\mu$ then the condition is trivially satisfied
so assume $\lambda\not\le\mu$. Then by the assumption we have
$F_{\lambda\mu}=\sum_{\nu<\lambda}I^\nu_{\lambda\mu}$.
Each $I^\nu_{\lambda\mu}$ is contained in
\[\sum_{\pi\le\nu,\mu}\Image(F_{\lambda\nu}F_{\nu\pi}F_{\pi\mu}\to F_{\lambda\mu})
\subset\sum_{\pi\le\lambda,\mu}I^\pi_{\lambda\mu}\]
since $F_{\lambda\mu}$ is cocontinuous.
Thus the condition is also satisfied for $\lambda$.
\end{proof}

Now assume that a partial order $\le$ is compatible with $\{F_{\lambda\mu}\}$.
Let us denote
\[I^{\le\lambda}_{\alpha\beta}\coloneqq\sum_{\mu\le\lambda}I^\lambda_{\alpha\beta}
\quad\text{and}\quad
I^{<\lambda}_{\alpha\beta}\coloneqq\sum_{\mu<\lambda}I^\lambda_{\alpha\beta}.\]

\begin{proposition}
For each $\omega\in\Lambda$
the family $\{I^{\le\lambda}_{\omega\omega}\}_{\lambda\in\Lambda}$
is an ideal filter on $\cC_\omega$.
\end{proposition}
\begin{proof}
For simplicity let us write $I^\lambda\coloneqq I^\lambda_{\omega\omega}$
and $I^{\le\lambda}\coloneqq I^{\le\lambda}_{\omega\omega}$.
The first two conditions in Definition~\ref{def:ideal_filter} are obvious.
So we prove
$I^{\le\lambda}I^{\le\mu}\subset\sum_{\rho\le\lambda,\mu}I^{\le\rho}$
for each $\lambda,\mu\in\Lambda$.
Let us take $\nu\le\lambda$ and $\pi\le\mu$.
Since $F_{\omega\nu}$ is cocontinuous,
we have
\[I^\nu I^\pi\subset\Image(F_{\omega\nu}F_{\nu\pi}F_{\pi\omega}\to\Id_{\cC_\omega})
\subset\sum_{\rho\le\nu,\pi}I^\rho.\]
Hence by taking sum we obtain the inclusion as desired.
\end{proof}

Using this ideal filter
the category $\cC_\omega$ is divided into
$\cC_\omega[\lambda]=(\cC_\omega)^{\leq\lambda}\cap(\cC_\omega)_{<\lambda}$.
For each $\lambda$, by taking the quotient with respect to
the subset $\Lambda'=\set{\mu\in\Lambda}{\mu<\lambda}$
we have a Morita context between $(\cC_\omega)_{<\lambda}$ and $(\cC_\lambda)_{<\lambda}$
whose trace ideal in $(\cC_\omega)_{<\lambda}$ is just
$(I^{\le\lambda}_{\omega\omega})_{I^{<\lambda}_{\omega\omega}}$.
The corresponding trace ideal in $(\cC_\lambda)_{<\lambda}$
is $(I^\omega_{\lambda\lambda}+I^{<\lambda}_{\lambda\lambda})_{I^{<\lambda}_{\lambda\lambda}}$.
Thus by letting
\[\cC_\lambda\langle\omega\rangle\coloneqq
(\cC_\lambda)^{I^\omega_{\lambda\lambda}}\cap(\cC_\lambda)_{<\lambda}\]
and $\Irr\cC_\lambda\langle\omega\rangle\coloneqq\set{V\in\Irr\cC}{V\in\cC_\lambda\langle\omega\rangle}$
we obtain the following theorem.
\begin{theorem*}\label{thm:multi_morita_context}
For each $\lambda\le\omega$, there is a Morita context between
$(\cC_\omega)_{<\lambda}$ and $(\cC_\lambda)_{<\lambda}$
which induces a category equivalence
$\cC_\omega[\lambda]\simeq\cC_\lambda\langle\omega\rangle$
and a one-to-one correspondence 
$\Irr\cC_\omega[\lambda]\onetoone\Irr\cC_\lambda\langle\omega\rangle$.
If $\lambda\not\le\omega$, then $\cC_\omega[\lambda]=0$.
\end{theorem*}
\begin{corollary*}
If $\Lambda$ is well-founded, we have
\[\Irr\cC_\omega=\bigsqcup_{\lambda\le\omega}\Irr\cC_\omega[\lambda]
\onetoone\bigsqcup_{\lambda\le\omega}\Irr\cC_\lambda\langle\omega\rangle.
\qedhere\]
\end{corollary*}

\section{Generalized cellular algebras}\label{sec:cellular}
Now we concentrate on representation theory of algebras.
Here continuously the term ``an algebra'' means a $\cV$-algebra.
With the help of the category equivalence $\cAdj(\lMod{B},\lMod{A})\simeq\bMod{A}{B}$,
we can interpret all the notions we have introduced in the previous sections
into the language of modules.
For example, ideal functors on the category are replaced by 2-sided ideals in an algebra.
So an ideal filter is just a collection of 2-sided ideals
which satisfies the similar conditions.

In this section we fix a partially ordered set $(\Lambda,\le)$
and an indexed family $\{B_\lambda\}_{\lambda\in\Lambda}$ of algebras.
We introduce a generalized notion of standardly based algebra
and that of cellular algebra over the family $\{B_\lambda\}$,
not over the single base algebra $\kk$.
We also study its Morita invariance motivated by the work of
K\"onig and Xi~\cite{KonigXi99}.

\subsection{Standard filter}
We start from a very general setting.
In the last of previous section
we decompose a category into small parts in order to study them one by one.
A standardly filtered algebra is defined so that
we can perform similar strategy for its module category.
\begin{definition}\label{def:standardly_filtered_algebra}
Let $A$ be an algebra.
A \term{prestandard filter} of $A$ over $\{B_\lambda\}$
is a datum consisting of:
\begin{itemize}
\item an ideal filter $\{A^{\le\lambda}\}_{\lambda\in\Lambda}$ on $A$,
\item for each $\lambda\in\Lambda$, a 2-sided ideal
$B'_\lambda\subset B_\lambda$,
\item for each $\lambda\in\Lambda$, a Morita context $(M_\lambda,N_\lambda)$
between $A/A^{<\lambda}$ and $B_\lambda/B'_\lambda$ whose trace ideal in $A$
is $A^{\le\lambda}/A^{<\lambda}$.
\end{itemize}
Moreover if it satisfies
$A^{\le\mu}M_\lambda=0$ and $N_\lambda A^{\le\mu}=0$
for each pair of $\lambda,\mu$ such that $\lambda\not\le\mu$,
we call it a \term{standard filter}.
An algebra equipped with a standard filter
is called a \term{standardly filtered algebra}.
If each $B_\lambda$ is just the base ring $\kk$,
we simply say it is a standard filter over $\kk$ instead of over the family $\{\kk\}$.
\end{definition}

Now Lemma~\ref{lem:category_produces_filter} in the introduction is
just a reformulation of Theorem~\ref{thm:multi_morita_context}
on Example~\ref{ex:category_is_morita}.
\begin{remark}\label{rem:well_founded_filtered_algebra}
By Lemma~\ref{lem:well_founded_Morita}, if $\Lambda$ is well-founded
the first assumption of Lemma~\ref{lem:category_produces_filter}
can be weakened to
\[\Hom_\cA(X_\mu,X_\lambda)=\cA^{<\lambda}(X_\mu,X_\lambda).\]
\end{remark}
\begin{remark}
In the settings of the lemma, for $\omega_1,\dots,\omega_n\in\Lambda$,
the algebra
\[\End_\cA\Bigl(\bigoplus_i X_{\omega_i}\Bigr)=\bigoplus_{i,j}\Hom_\cA(X_{\omega_i},X_{\omega_j})\]
is also standardly filtered.
It can be proven by adding a new index $\infty$
which is greater than any element of $\Lambda$ so that
$X_\infty=\bigoplus_i X_{\omega_i}$,
then remove it since it is needless by that $\cA^{\le\infty}=\cA^{<\infty}$.
\end{remark}

Actually the condition for being a standardly filtered algebra
can be weakened as follows.

\begin{lemma}\label{lem:reducing_morita}
Suppose that $(M,N)$ is a Morita context between algebras $A$ and $B$,
and let us write its equipped maps as $\eta\colon M\otimes_B N\to A$ and
$\rho\colon N\otimes_A M\to B$.
Let
\begin{align*}
B'&\coloneqq\set{b\in B}{\eta(mb\otimes n)=0\text{ for all }m\in M,n\in N},\\
M'&\coloneqq\set{m\in M}{\eta(m\otimes n)=0\text{ for all }n\in N},\\
N'&\coloneqq\set{n\in N}{\eta(m\otimes n)=0\text{ for all }m\in M}.
\end{align*}
Then $(M/M',N/N')$ is a Morita context between algebras $A$ and $B/B'$
with the same trace ideal in $A$.
\end{lemma}
\begin{proof}
First by definition $\eta\colon M/M'\otimes_B N/N'\to A$ is well-defined.
In addition we have $MB'\subset M'$ and $B'N\subset N'$
so that $M/M'$ and $N/N'$ can be considered as modules over $B/B'$.
Moreover $\rho(M'\otimes_A N),\rho(M\otimes_A N')\subset B'$ by the associativity,
so that $\rho\colon N/N'\otimes_A M/M'\to B/B'$ is also well-defined.
Now it is clear that these data form a Morita context between $A$ and $B/B'$.
\end{proof}
Note that $B'$ above is the common annihilator
of $M/M'$ and $N/N'$, so that these are faithful modules over $B/B'$.

\begin{proposition}\label{prop:prestandard_filter_induces_standard}
If an algebra $A$ has a prestandard filter, it also has a standard filter.
\end{proposition}
\begin{proof}
Take a prestandard filter of $A$ as above.
For each $\lambda$, let $A^{\not\ge\lambda}\coloneqq\sum_{\mu\not\ge\lambda}A^{\le\mu}$.
Then
\begin{align*}
\eta(A^{\not\ge\lambda}M_\lambda\otimes_{B_\lambda} N_\lambda)
&=A^{\not\ge\lambda}A^{\le\lambda}+A^{<\lambda}=A^{<\lambda},\\
\eta(M_\lambda\otimes_{B_\lambda} N_\lambda A^{\not\ge\lambda})
&=A^{\le\lambda}A^{\not\ge\lambda}+A^{<\lambda}=A^{<\lambda}.
\end{align*}
Hence taking $M'_\lambda\subset M_\lambda$ and $N'_\lambda\subset N_\lambda$
as in the lemma above, we have inclusions $A^{\not\ge\lambda}M_\lambda\subset M'_\lambda$
and $N_\lambda A^{\not\ge\lambda}\subset N'_\lambda$.
Thus by replacing $(M_\lambda,N_\lambda)$ with the quotients
$(M_\lambda/M'_\lambda,N_\lambda/N'_\lambda)$
we obtain a standard filter.
\end{proof}

The notion of standardly filtered algebra is a Morita invariant
and inherited by Peirce decomposition.

\begin{proposition}\label{prop:standardly_filtered_is_morita_invariant}
Let $A$ be a standardly filtered algebra over $\{B_\lambda\}$.
\begin{enumerate}
\item For any idempotent $e\in A$, the algebra $eAe$ is also standardly filtered.
\item If an algebra $A'$ is Morita equivalent to $A$ (i.e.\ $\lMod{A}\simeq\lMod{A'}$),
$A'$ is also standardly filtered.
\end{enumerate}
\end{proposition}
\begin{proof}
(1) follows from that the pair $(e M_\lambda,N_\lambda e)$ forms
a Morita context between $eAe/eA^{<\lambda}e$ and $B_\lambda/B'_\lambda$.
(2) is a consequence of that the definition of standard filter on an algebra
can be translated into the language of its module category.
\end{proof}

For an algebra $A$ and a 2-sided ideal $I\subset A$, let us write
\[\Irr(A)\coloneqq\Irr(\lMod{A})\quad\text{and}\quad
\Irr^I(A)\coloneqq\Irr(\lMod{A}^I)=\Irr(A)\setminus\Irr(A/I)\]
for short. More generally, for $J\subset I\subset A$ let
\[\Irr^I_J(A)\coloneqq\Irr^{I/J}(A/J)=\Irr(A/J)\setminus\Irr(A/I).\]
Then Proposition~\ref{prop:dividing_simples} and Theorem~\ref{thm:morita_simple_1:1}
immediately bring us
the following classification of simple $A$-modules.
This is a generalization of \cite[Theorem~3.4]{GrahamLehrer96}.

\begin{theorem*}
Suppose that $\Lambda$ is well-founded.
Let $A$ be a standardly filtered algebra over $\{B_\lambda\}$ and
take its prestandard filter as above.
For each $\lambda$, let $B''_\lambda/B'_\lambda\subset B_\lambda/B'_\lambda$
be the trace ideal of the Morita context.
Then there is a one-to-one correspondence
\[\Irr(A)=\bigsqcup_{\lambda\in\Lambda}
\Irr^{A^{\le\lambda}}_{A^{<\lambda}}(A)\onetoone
\bigsqcup_{\lambda\in\Lambda}\Irr^{B''_\lambda}_{B'_\lambda}(B_\lambda)\]
induced by Morita contexts.
\end{theorem*}

Let us write $[M:S]$ the multiplicity of a simple module $S$
in the composition factors of $M$.
The analogue of the decomposition matrix of cellular algebra
can be defined as follows.
It also satisfies the unitriangular property.
\begin{lemma}
Take a standard filter of $A$ as above.
Let $\lambda,\mu\in\Lambda$ and
take $S\in\Irr^{A^{\le\mu}}_{A^{<\mu}}(A)$, $T\in\Irr(B_\lambda)$.
Then unless $\lambda\le\mu$
\[[M_\lambda\otimes_{B_\lambda}T:S]=[\Hom_{B_\lambda}(N_\lambda,T):S]=0.\]
Moreover, if $\lambda=\mu$,
\[[M_\lambda\otimes_{B_\lambda}T:S]=[\Hom_{B_\lambda}(N_\lambda,T):S]=\begin{cases*}
1 & if $T\simeq D S$,\\
0 & otherwise.
\end{cases*}\]
Here $D$ is the Morita context functor which induces
$\Irr^{A^{\le\lambda}}_{A^{<\lambda}}(A)\onetoone\Irr^{B''_\lambda}_{B'_\lambda}(B_\lambda)$.
\end{lemma}
\begin{proof}
The first equation follows from that $M_\lambda\otimes_{B_\lambda}T$ and
$\Hom_{B_\lambda}(N_\lambda,T)$ are $A^{\not\ge\lambda}$-annihilated.
The second follows from Corollary~\ref{cor:multiplicity_on_Morita_context}
if $T\in\Irr(B/B'_\lambda)$;
otherwise $M_\lambda\otimes_{B_\lambda}T=\Hom_{B_\lambda}(N_\lambda,T)=0$
so the formula also holds trivially.
\end{proof}

\subsection{Well-based standard filter}
Graham and Lehrer~\cite[Theorem~3.7]{GrahamLehrer96} also proved that
for a cellular algebra we can compute its Cartan matrix
by its decomposition matrix.
A general standardly filtered algebra does not have this property,
so we strengthen its conditions to prove an analogue of the theorem.

\begin{definition}
Let $(M,N)$ be a Morita context between algebras $A$ and $B$.
We say that $(M,N)$ is \term{well-based} over $B$
if $M$ and $N$ are both finitely generated
and projective over $B$ and the map $M\otimes_B N\to A$ is injective.
\end{definition}

\begin{definition}
A prestandard filter of $A$ is said to be \term{well-based} if
\begin{enumerate}
\item the ideal filter $\{A^{\le\lambda}\}$ is rigid,
\item the Morita context $(M_\lambda,N_\lambda)$ is well-based over $B_\lambda/B'_\lambda$
for every $\lambda\in\Lambda$.
\end{enumerate}
An algebra equipped with a well-based standard filter
is called a \term{weakly standardly based algebra}.
\end{definition}

We can prove a statement similar to Proposition~\ref{prop:prestandard_filter_induces_standard}
for weakly standardly based algebras.
The proof is clear from the lemmas below.
\begin{lemma}
Let $B$ be an algebra.
Let $M$ be a finitely generated projective right $B$-module and
$N$ be a left $B$-module.
Then $x\in M$ satisfies $0=x\otimes y\in M\otimes_B N$ for all $y\in N$
if and only if $x\in M\cdot\Ann_B(N)$.
Here $\Ann_B(N)\coloneqq\set{b\in B}{bN=0}$ denotes the (left) annihilator of $N$.
\end{lemma}
\begin{proof}
The ``if'' part is obvious so we prove the ``only if'' part.
We may assume that there is an $m\times m$ idempotent matrix $e=(e_{ij})$
such that $M=e B^m$.
Since $M\subset B^m$ is an direct summand,
we can regard $M\otimes_B N\subset B^m\otimes_B N=N^m$.
Suppose $x={^t}(x_1,\dots,x_m)\in M$ satisfies $x\otimes N=0$.
This means that $0=x\otimes n={}^t(x_1n,\dots,x_m n)\in N^m$ for all $n\in N$, that is,
$x_1,\dots,x_m\in\Ann_B(N)$. Thus
\[x=ex={^t}(e_{11},\dots,e_{m1})x_1+\dots+{^t}(e_{1m},\dots,e_{mm})x_m
\in M\cdot\Ann_B(N).
\qedhere\]
\end{proof}
\begin{lemma}
Let $A$, $B$, $M$ and $N$ as in Lemma~\ref{lem:reducing_morita}.
If $(M,N)$ is well-based over $B$, then so is $(M/M',N/N')$
over $B/B'$.
\end{lemma}
\begin{proof}
By the lemma above, we have $M'=M\cdot\Ann_B(N)\subset MB'$.
We already has the other inclusion so $M'=MB'$,
hence $M/M'\simeq M\otimes_B (B/B')$ is finitely generated and projective over $B/B'$.
The same holds for $N/N'$.
It is clear that $\eta\colon M/M'\otimes_B N/N'\to A$ is also injective.
\end{proof}

\begin{proposition*}\label{prop:prestandard_basis_induces_standard}
If an algebra $A$ has a well-based prestandard filter,
it also has a well-based standard filter.
\end{proposition*}

We also prove the statements similar to
Proposition~\ref{prop:standardly_filtered_is_morita_invariant}.
\begin{proposition}\label{prop:standardly_based_is_morita_invariant}
Let $A$ be a weakly standardly based algebra over $\{B_\lambda\}$.
\begin{enumerate}
\item For any idempotent $e\in A$, the algebra $eAe$ is also weakly standardly based.
\item If an algebra $A'$ is Morita equivalent to $A$,
$A'$ is also weakly standardly based.
\end{enumerate}
\end{proposition}
\begin{proof}
(1) follows from that $eM_\lambda$ and $N_\lambda e$ are also
finitely generated and projective, and that
$eM_\lambda\otimes_{B_\lambda}N_\lambda e\simeq e(M_\lambda\otimes_{B_\lambda}N_\lambda)e$.
(2) follows from the following categorical characterization
of being finitely generated and projective:
a right (resp.\ left) $B$-module $M$ is finitely generated projective
if and only if the functor $M\otimes_B{\bullet}$
(resp. $\Hom_B(M,{\bullet})$) also have its left (resp.\ right) adjoint functor.
\end{proof}

Now \cite[Theorem~3.7]{GrahamLehrer96} can be generalized as follows.
\begin{theorem}
Suppose that $\kk$ is a field, $\Lambda$ is well-founded
and each $B_\lambda$ is finite dimensional and semisimple.
Let $A$ be a weakly standardly based algebra over $\{B_\lambda\}$
and take its well-based prestandard filter.
Let $S_1,S_2\in\Irr(A)$ and suppose that they have projective covers
$P_i\twoheadrightarrow S_i$.
Then
\[[P_2:S_1]=\dim_\kk\End_A(S_2)\sum_\lambda
\sum_{T\in\Irr(B_\lambda)}\frac{[M_\lambda\otimes_{B_\lambda}T:S_1][\Hom_{B_\lambda}(N_\lambda,T):S_2]}{\dim_\kk\End_{B_\lambda}(T)}.\]
\end{theorem}
Note that we can take $\mu\in\Lambda$ such that $S_2\in\Irr_{A^{<\mu}}^{A^{\le\mu}}(A)$
then by the Morita context we have an isomorphism $\End_A(S_2)\simeq\End_{B_\mu}(DS_2)$,
so that its dimension is also easy to compute.
\begin{proof}
Since the ideal filter is rigid,
we have $[P_2:S_1]=\sum_\lambda[A^{\le\lambda}P_2/A^{<\lambda}P_2:S_1]$
by Proposition~\ref{prop:composition_factor}.
Then for each $\lambda$, we have
\[[A^{\le\lambda}P_2/A^{<\lambda}P_2:S_1]=\dim_\kk\Hom_A(P_1,A^{\le\lambda}P_2/A^{<\lambda}P_2)/\dim_\kk\End_A(S_1)\]
and by using $A^{\le\lambda}/A^{<\lambda}\simeq M_\lambda\otimes_{B_\lambda}N_\lambda$
and that $P_2$ is flat,
\begin{align*}
\Hom_A(P_1,A^{\le\lambda}P_2/A^{<\lambda}P_2)&\simeq
\Hom_A(P_1,M_\lambda\otimes_{B_\lambda}N_\lambda\otimes_A P_2)\\
&\simeq\Hom_{B_\lambda}(M^\vee_\lambda\otimes_A P_1,N_\lambda\otimes_A P_2)
\end{align*}
where $M^\vee_\lambda\coloneqq\Hom^\op_{B_\lambda}(M_\lambda,B_\lambda)$.
Since $B_\lambda$ is semisimple,
\[
\dim_\kk\Hom_{B_\lambda}(M^\vee_\lambda\otimes_A P_1,N_\lambda\otimes_A P_2)
=\sum_{T\in\Irr(B_\lambda)}[M^\vee_\lambda\otimes_A P_1:T][N_\lambda\otimes_A P_2:T]\dim_\kk\End_{B_\lambda}(T).
\]
Moreover we have
\[\Hom_{B_\lambda}(M^\vee_\lambda\otimes_A P_1,T)\simeq\Hom_A(P_1,M_\lambda\otimes_{B_\lambda}T)\]
which implies
\[[M^\vee_\lambda\otimes_A P_1:T]\dim_\kk\End_{B_\lambda}(T)=
[M_\lambda\otimes_{B_\lambda}T:S_1]\dim_\kk\End_A(S_1).\]
Similarly we have
\[[N_\lambda\otimes_A P_2:T]\dim_\kk\End_{B_\lambda}(T)=
[\Hom_{B_\lambda}(N_\lambda,T):S_2]\dim_\kk\End_A(S_2).\]
Putting them all together, we obtain the equation.
\end{proof}

Quasi-hereditary algebra is an important class  of algebra
introduced by Cline, Parshall and Scott~\cite{ClineParshallScott88}.
The condition for a non-generalized standardly based algebra to be quasi-hereditary
is given by Graham and Lehrer~\cite{GrahamLehrer96}, and
Du and Rui~\cite{DuRui98}.
We can prove an analogous partial result for our generalized standardly based algebra.

\begin{lemma}
Let $(M,N)$ be a well-based Morita context between $A$ and $B$,
and suppose that the algebra $B$ is semisimple.
If $\rho\colon N\otimes_A M\to B$ is surjective,
then the trace ideal $I\coloneqq\eta(M\otimes_B N)\subset A$
is generated by an idempotent,
and finitely generated and projective as both a left and a right $A$-module.
\end{lemma}
\begin{proof}
By the Artin--Wedderburn theorem and the Morita equivalence,
we may assume that $B$ is a product of finitely many division algebras:
\[B=D_1\times D_2\times\dots\times D_l\]
(here we mean that every non-zero \emph{homogeneous} element $x\in D_i$
is invertible).
Let us write $1_i=(0,\dots,1,\dots,0)\in B$ the identity element of each $D_i$.
Since $\rho$ is surjective, for each $i$ we can find $m_i\in M$ and $n_i\in N$
such that $\rho(n_i\otimes m_i)1_i\neq0$.
By multiplying elements in $D_i$ we may assume that
$\rho(n_i\otimes m_i)=1_i$, $1_i n_i=n_i$ and $m_i 1_i=m_i$.
Thus $\rho(n_i\otimes m_j)=0$ for $i\neq j$.
Let $e\coloneqq\sum_i\eta(m_i\otimes n_i)\in A$.
Then by the associativity $e$ is an idempotent.
Moreover the maps
\begin{align*}
M&\to Ae,&
Ae&\to M,\\
m&\mapsto\sum_i\eta(m\otimes n_i),&
a&\mapsto\sum_i am_i
\end{align*}
are inverses of each other.
Hence $M\simeq Ae$ is a finitely generated and projective left $A$-module.
Similarly $N\simeq eA$ as right $A$-modules
so that $I\simeq M\otimes_B N$
is finitely generated and projective from both sides.
By these isomorphisms we also have $I=AeA$.
\end{proof}
Hence in this case we have $\Ext^i_{A/I}\simeq\Ext^i_A$ for any $i$
by Proposition~\ref{prop:ext_preserving}.
We also have $\Ext^1_A(V,W)=0$ for $V,W\in\Irr^I(A)$
by Lemma~\ref{lem:morita_with_semisimple}.

\subsection{Standard basis}

We here give the definition of class of algebras
which is more closely related to the original one of cellular algebra.

\begin{definition}
A \term{(generalized) standard basis} of an algebra $A$ is
a direct sum decomposition
\[A=\bigoplus_{\lambda\in\Lambda}A^\lambda\]
as a $\kk$-module (not as a left or right $A$-module)
such that for each $\lambda$
\[A^{\le\lambda}\coloneqq\bigoplus_{\mu\le\lambda}A^\mu\quad\text{and}\quad
A^{<\lambda}\coloneqq\bigoplus_{\mu<\lambda}A^\mu\]
are both 2-sided ideals of $A$, equipped with for each $\lambda$ an isomorphism
of $(A,A)$-bimodules
\[M_\lambda\otimes_{B_\lambda}N_\lambda\simeq A^{\le\lambda}/A^{<\lambda}\]
for a pair of an $(A,B_\lambda)$-bimodule $M_\lambda$ and
a $(B_\lambda,A)$-bimodule $N_\lambda$ which are both
finitely generated and free over $B_\lambda$.
An algebra equipped with a standard basis is called a
\term{(generalized) standardly based algebra}.
\end{definition}

When every $B_\lambda$ is the base algebra $\kk$,
this definition coincides with the original we given at the beginning.

\begin{proposition}
A standardly based algebra is a weakly standardly based algebra.
\end{proposition}
\begin{proof}
Since $A^{\le\lambda}A^{\le\mu}\subset A^{\le\lambda}\cap A^{\le\mu}
=\bigoplus_{\nu\le\lambda,\mu}A^\nu$,
the collection $\{A^{\le\lambda}\}_{\lambda\in\Lambda}$ is an ideal filter on $A$.
As proved in~\cite[Proposition~2.4]{GrahamLehrer96},
we can construct a suitable $(B_\lambda,B_\lambda)$-homomorphism
$\rho\colon M_\lambda\otimes_A N_\lambda\to B_\lambda$ for each $\lambda$
which completes a Morita context between $A/A^{<\lambda}$ and $B_\lambda$.
\end{proof}

The converse is also holds when the following assumptions are satisfied.

\begin{proposition}\label{prop:weakly_standardly_based_implies_standardly_based}
Suppose that $\Lambda$ is well-founded and
every $B_\lambda$ is projective over $\kk$.
Then a weakly standardly based algebra
is a standardly based algebra if
$M_\lambda$ and $N_\lambda$ are free over $B_\lambda$ for every $\lambda$.
\end{proposition}
\begin{proof}
Let $A$ be a weakly standardly based algebra.
Since its ideal filter $\{A^{\le\lambda}\}$ is rigid,
by taking a well-ordering extension,
we obtain a well-ordered filtration of $A$ whose successive quotients
are $A^{\le\lambda}/A^{<\lambda}$.
Each of them is isomorphic to $M_\lambda\otimes_{B_\lambda}N_\lambda$
which is projective over $\kk$, so
we can lift $M_\lambda\otimes_{B_\lambda}N_\lambda\hookrightarrow A/A^{<\lambda}$
to some $\kk$-linear map
$\iota_\lambda\colon M_\lambda\otimes_{B_\lambda}N_\lambda\hookrightarrow A$.
Then as a $\kk$-module $A$ decompose into a direct sum of
$\kk$-modules $A^\lambda\coloneqq\iota_\lambda(M_\lambda\otimes_{B_\lambda}N_\lambda)$
as desired.
\end{proof}

It is a natural question to ask whether the freeness condition of the definition of
standardly based algebra can be weakened to the projectiveness.
So suppose that we are given an $(A,B)$-bimodule $M$ and
a $(B,A)$-bimodule $N$ which are both finitely generated and projective over $B$,
equipped with an injective $(A,A)$-homomorphism
$\eta\colon M\otimes_B N\hookrightarrow A$.
By replacing them with their quotients, we may assume that
$B'$, $M'$ and $N'$ taken as in lemma~\ref{lem:reducing_morita} are all zero.
The existence of $\rho\colon N\otimes_B M\to A$
fails in this general situation:
consider the following counterexample that
\[A=\kk,\qquad B=\left\{\begin{pmatrix}*&*\\0&*\end{pmatrix}\right\},\qquad
M=\left\{\begin{pmatrix}*&*\end{pmatrix}\right\},\qquad
N=\left\{\begin{pmatrix}* \\ *\end{pmatrix}\right\}\]
with natural isomorphism $\eta\colon M\otimes_B N\simeq A$.
We state a sufficient condition for its existence as follows.
This is a generalization of \cite[Proposition~2.4]{GrahamLehrer96} we quoted above.

\begin{lemma}\label{lem:completion_of_morita_context}
Let $A$, $B$, $M$ and $N$ as above.
Let $\Tr_{B^\op}(M),\Tr_B(N)\subset B$ be the trace ideals of $M$ and $N$ in $B$, that is,
\[\Tr_{B^\op}(M)\coloneqq\Image(M^\vee\otimes_A M\to B),\qquad
\Tr_B(N)\coloneqq\Image(N\otimes_A N^\vee\to B)\]
where $M^\vee\coloneqq\Hom_{B^\op}(M,B)$ and $N^\vee\coloneqq\Hom_B(N,B)$.
If $M$ is $\Tr_B(N)$-accessible and $N$ is $\Tr_{B^\op}(M)$-accessible,
then there exists a unique $(B,B)$-homomorphism
$\rho\colon N\otimes_A M\to B$
which makes $(M,N)$ into a Morita context between $A$ and $B$.
\end{lemma}
\begin{proof}
The uniqueness of $\rho$ is clear from that $B'=0$, so we prove its existence.
First consider the sequence
\[\begin{tikzcd}[column sep=large]
M\otimes_B N\otimes_A M\otimes_B N
\arrow[yshift=.7mm]{r}{\eta\otimes_A(M\otimes_B N)}
\arrow[yshift=-.7mm,swap]{r}{(M\otimes_B N)\otimes_A\eta} &
M\otimes_B N \arrow[hook]{r}{\eta} &
A\rlap{.}
\end{tikzcd}\]
Since the two parallel homomorphisms above are equalized by $\eta$ which is injective,
these are equal.
This implies that the diagram below is commutative:
\[\begin{tikzcd}[column sep=large]
N\otimes_A M\otimes_B N\otimes_A N^\vee
\arrow{r}{N\otimes_A\eta\otimes_A N^\vee} \arrow[two heads]{d} &
N\otimes_A N^\vee \arrow{r} &
B \arrow[hook]{d} \\
N\otimes_A M \arrow{rr} &&
\End_A(M)^\op\rlap{.}
\end{tikzcd}\]
Here the map at the bottom is given by $n\otimes m\mapsto(m'\mapsto\eta(m'\otimes n)m)$.
By the assumption that $M$ is $\Tr_B(N)$-accessible, the left vertical arrow
is surjective. On the other hand, since $M$ is faithful over $B$,
the right vertical arrow is injective.
Hence the diagram induces a $(B,B)$-homomorphism $\rho\colon N\otimes_A M\to B$
which satisfies $\eta(m'\otimes n)m=m'\rho(n\otimes m)$
for all $m,m'\in M$ and $n\in N$.
Dually we can prove the existence of another $\rho'\colon N\otimes_A M\to B$
such that $n\eta(m\otimes n')=\rho'(n\otimes m)n'$,
but we have $\rho=\rho'$ by the uniqueness.
\end{proof}
When $M$ and $N$ are free over $B$ the accessibility condition is trivially satisfied,
so that this proof is essentially the same as the original one by Graham and Lehrer.
Note that this condition is not necessary:
consider the example above with replacing
$A$ with $\kk\oplus\kk\epsilon$, $\epsilon^2=0$
so that $\epsilon M=0$, $N\epsilon=0$ and $\eta\colon M\otimes_B N\simeq A\epsilon$,
which clearly has $\rho=0$.
One necessary condition is that $\eta(M\otimes_B N)M\subset M\Tr_B(N)$
and $N\eta(M\otimes_B N)\subset \Tr_{B^\op}(M)N$, but
the author does not know it is sufficient for the existence of $\rho$ or not.

\subsection{Involution on algebras}

For an algebra $A$,
we call an algebra homomorphism $A\to A^\op$ whose square is equal
to the identity an \term{anti-involution} on $A$.
That is, it is a degree-zero map ${\bullet}^*\colon A\to A$ satisfying
\[1^*=1,\quad(ab)^*=(-1)^{\abs{a}\abs{b}}b^*a^*\quad\text{and}\quad a^{**}=a\]
(beware the Koszul sign).
If $A$ has an anti-involution, for each left $A$-module $M$
there is a corresponding right $A$-module $M^*$
whose underlying set is equal to $M$ and
action is defined by $x^*\cdot a^*\coloneqq(-1)^{\abs{a}\abs{x}}(ax)^*$,
where we write $x^*\in M^*$ the element corresponds to $x\in M$.
Similarly for a right $A$-module $N$ we denote by $N^*$
the corresponding left $A$-module, so that $M^{**}\simeq M$.

\begin{definition}
Let $A$ and $B$ be algebras with anti-involution.
A Morita context $(M,N)$ between $A$ and $B$ is said to be \term{involutive}
if there is an isomorphism $M\simeq N^*$ of $(A,B)$-bimodules 
(so $M^*\simeq N$) which satisfies
\[\eta(x\otimes y)^*=(-1)^{\abs{x}\abs{y}}\eta(y^*\otimes x^*),\quad
\rho(y\otimes x)^*=(-1)^{\abs{x}\abs{y}}\rho(x^*\otimes y^*)\]
for every $x\in M$, $y\in N$.
\end{definition}

Now we assume that each $B_\lambda$ has a fixed anti-involution.

\begin{definition}
A standardly filter on an algebra $A$ with anti-involution
is said to be \term{involutive}
if for each $\lambda$, $A^{\le\lambda}$ and $B'_\lambda$ are closed under anti-involution
and the Morita context $(M_\lambda,N_\lambda)$ between
$A/A^{<\lambda}$ and $B_\lambda/B'_\lambda$ is involutive.
An algebra equipped with an involutive well-based standard filter
is called a \term{weakly cellular algebra}.
\end{definition}

Note that in the settings of Lemma~\ref{lem:category_produces_filter},
when $\cA$ has an anti-involution $\cA\to\cA^\op$ which fixes
all $X_\lambda$ and $B_\lambda$,
it produces an involutive standard filter.
The statements below are clear from the definition.

\begin{lemma*}
Let $A$, $B$, $M$ and $N$ as in Lemma~\ref{lem:reducing_morita}.
If $A$ and $B$ have their anti-involutions and $(M,N)$ is involutive, then
$(B')^*=B'$ and $(M/M',N/N')$ is also involutive.
\end{lemma*}

\begin{proposition*}
If an algebra $A$ with anti-involution has a involutive (well-based) prestandard filter,
it also has an involutive (well-based) standard filter.
\end{proposition*}

To deal its Morita invariant property we should be careful
with the compatibility between Morita equivalence and anti-involution.
See the hypotheses $(*)$ and $(\dag)$ in \cite{KonigXi99}.
Recall that the Morita equivalence $\lMod{A}\simeq\lMod{A'}$
also induces $\rMod{A}\simeq\rMod{A'}$.
We here say that algebras $A$ and $A'$ with anti-involution
are \term{involutively Morita equivalent} if the category equivalence
makes the diagram
\[\begin{tikzcd}
\lMod{A} \arrow{r}{*} \arrow{d}[swap]{\sim} &
\rMod{A} \arrow{d}{\sim} \\
\lMod{A'} \arrow{r}[swap]{*} &
\rMod{A'}
\end{tikzcd}\]
commutes up to natural isomorphism.
Note that if $A'$ is equivalent to $A$ we can find an idempotent
$m\times m$ matrix $e=(e_{ij})$ over $A$ such that
$A'\simeq e\cdot\mathrm{Mat}_m(A)\cdot e$.
The condition above is equivalent to that we can also take $e$ so that
$e_{ij}^*=e_{ji}$.

The proofs of the statements below
are obvious by Proposition~\ref{prop:standardly_based_is_morita_invariant}.

\begin{proposition*}\label{prop:cellular_is_morita_invariant}
Let $A$ be a weakly cellular algebra over $\{B_\lambda\}$.
\begin{enumerate}
\item For any idempotent $e\in A$ such that $e^*=e$,
the algebra $eAe$ with the same anti-involution is also weakly cellular.
\item If an algebra $A'$ with anti-involution is involutively Morita equivalent to $A$,
$A'$ is also weakly standardly filtered.
\qedhere
\end{enumerate}
\end{proposition*}
In \cite{KonigXi99} it is also proved that even if we are not given
a such anti-involution on $A'$ we can construct it from that on $A$.
Thus their result is stronger than above.

Finally we give the definition of cellular algebra in terms of basis.
Note the next lemma which follows by the uniqueness of $\rho$.
\begin{lemma*}
Suppose that $M$, $N$ and $\eta$
in Lemma~\ref{lem:completion_of_morita_context} satisfies
$M\simeq N^*$ and $\eta(x\otimes y)^*=(-1)^{\abs{x}\abs{y}}\eta(y^*\otimes x^*)$.
Then the induced map $\rho\colon N\otimes_A M\to B$ also satisfies
$\rho(y\otimes x)^*=(-1)^{\abs{x}\abs{y}}\rho(x^*\otimes y^*)$ so that
the Morita context $(M,N)$ between $A$ and $B$ is involutive.
\end{lemma*}

\begin{definition}
A standardly based algebra $A$ over $\{\cB_\lambda\}$
with anti-involution is called a \term{(generalized) cellular algebra}
if each component $A^\lambda$ is closed under anti-involution and the isomorphism
$M_\lambda\otimes_{B_\lambda}N_\lambda\simeq A^{\le\lambda}/A^{<\lambda}$
satisfies the involutive property similar as above.
\end{definition}
Again, this definition is same as the original one when every $B_\lambda$ is
just the base ring $\kk$.
Then the next statement is obvious from the lemma above.

\begin{proposition*}
A cellular algebra is a weakly cellular algebra.
\end{proposition*}

We prove the converse in suitable conditions.

\begin{proposition}
Suppose that the assumptions in Proposition~\ref{prop:weakly_standardly_based_implies_standardly_based}
are satisfied, in addition to that $2\in\kk$ is invertible.
Then a weakly cellular algebra
is a cellular algebra if
$M_\lambda$ and $N_\lambda$ are free over $B_\lambda$ for every $\lambda$.
\end{proposition}
\begin{proof}
The problem is that $\iota_\lambda$
we chose in the proof of Proposition~\ref{prop:weakly_standardly_based_implies_standardly_based}
does not preserve anti-involution.
So we retake a new map
$\iota'_\lambda\colon M_\lambda\otimes_{B_\lambda}N_\lambda\hookrightarrow A$
defined by
\[\iota'_\lambda(x\otimes y)\coloneqq
\frac{\iota_\lambda(x\otimes y)+(-1)^{\abs{x}\abs{y}}\iota_\lambda(y^*\otimes x^*)^*}{2}.\]
Then $\iota'_\lambda$ is also a lift of
$M_\lambda\otimes_{B_\lambda}N_\lambda\hookrightarrow A/A^{<\lambda}$
which satisfies $\iota'_\lambda(x\otimes y)^*=(-1)^{\abs{x}\abs{y}}\iota'_\lambda(y^*\otimes x^*)$.
Thus by putting $A^\lambda\coloneqq\iota'_\lambda(M_\lambda\otimes_{B_\lambda}N_\lambda)$
we obtain a desired direct sum decomposition.
\end{proof}

\section{Cellular structure on the Iwahori--Hecke algebra}\label{sec:cellular_hecke}

In this section we review the definition and the representation
theory of the Iwahori--Hecke algebra.
We give a new proof for that the Iwahori--Hecke algebra and the associated
$q$-Schur algebra are cellular with respect to
Murphy's basis~\cite{Murphy92,Murphy95}
in a more simple and sophisticated way than his original one
or given in \cite{Mathas99}.
We give a generalized theorem that classify its simple modules
on a very few assumptions.

\subsection{The symmetric groups}
We here introduce standard notions on Young tableaux
used in representation theory of the symmetric group and the Iwahori--Hecke algebra,
and we briefly recall some of their basic facts.
We refer the standard textbooks \cite{Humphreys90}, \cite{Fulton97}
and \cite{Mathas99} for details.

We write $\NN=\{0,1,2,\dots\}$ the set of natural numbers.
We denote by $\fS_n$ the symmetric group
of rank $n\in\NN$ acting on the set $\{1,2,\dots,n\}$ from left.
For $1\le i\le n-1$, let $s_i$ be the basic transposition $(i,i+1)$.
As a Coxeter group, $\fS_n$ is generated by the elements $s_1,s_2,\dots,s_{i-1}$.
With respect to this generator set, the \term{length} of $w\in\fS_n$
is equal to its inversion number
\[\ell(w)=\#\set{(i,j)}{1\le i<j\le n\text{ and }w(j)<w(i)}.\]

A \term{composition} of $n\in\NN$ is an infinite sequence
$\lambda=(\lambda_1,\lambda_2,\dots)$
of natural numbers whose sum, written as $\abs\lambda\coloneqq\sum_i\lambda_i$,
is equal to $n$.
Alternatively we often represent a composition $\lambda$ as a finite tuple
$\lambda=(\lambda_1,\lambda_2,\dots,\lambda_r)$
if it satisfies $\lambda_i=0$ for all $r>i$.
For such $\lambda$,
the corresponding \term{parabolic subgroup} (also called the \term{Young subgroup})
\[\fS_\lambda\coloneqq\fS_{\lambda_1}\times\fS_{\lambda_2}\times\dots\times\fS_{\lambda_r}\subset\fS_n\]
is defined.
It is known that the quotient set $\fS_n/\fS_\lambda$ has
the minimal length coset representatives
\[\fD_\lambda\coloneqq\set{w\in\fS_n}{\ell(w s_i)>\ell(w)\text{ for every }s_i\in\fS_\lambda}.\]
With respect to this set,
every $w\in\fS_n$ is uniquely decomposed as $w=uv$ to a pair of
$u\in\fD_\lambda$ and $v\in\fS_\lambda$ which satisfies $\ell(w)=\ell(u)+\ell(v)$.
For another composition $\mu$, the sets
\[\fD_\mu^{-1}=\set{w\in\fS_n}{\ell(s_i w)>\ell(w)\text{ for every }s_i\in\fS_\mu}\]
and $\fD_\lambda\cap\fD_\mu^{-1}$
are the minimal length representatives of the left cosets $\fS_\mu\backslash\fS_n$
and the double cosets $\fS_\mu\backslash\fS_n/\fS_\lambda$ respectively.

The \term{Poincar\'e polynomial} of a subset $S\subset\fS_n$ is defined by
$P_S(q)\coloneqq\sum_{w\in S}q^{\ell(w)}\in\ZZ[q]$.
If $S$ has a decomposition $S=S_1\cdot S_2$ which preserves lengths,
it follows by definition that $P_S(q)=P_{S_1}(q)P_{S_2}(q)$.
We have a $q$-factorial as the Poincar\'e polynomial of whole $\fS_n$,
\[P_{\fS_n}(q)=[n]!\coloneqq[1][2]\dotsm[n],\]
where $[k]$ is a $q$-integer $[k]=1+q+\dots+q^{k-1}$,
which follows inductively from the decomposition
$\fS_n=\bigsqcup_{1\le k\le n}s_k s_{k+1}\dotsm s_{n-1}\fS_{n-1}$.
Then for a composition $\lambda=(\lambda_1,\lambda_2,\dots,\lambda_r)$ of $n$, we obtain
\[P_{\fS_\lambda}(q)=P_{\fS_{\lambda_1}}(q)P_{\fS_{\lambda_2}}(q)\dotsm P_{\fS_{\lambda_r}}(q)=
[\lambda_1]![\lambda_2]!\dotsm[\lambda_r]!\]
and
\[P_{\fD_\lambda}(q)=\frac{P_{\fS_n}(q)}{P_{\fS_\lambda}(q)}
=\frac{[n]!}{[\lambda_1]![\lambda_2]!\dotsm[\lambda_r]!}.\]
This polynomial
\[\qbinom{n}{\lambda}=\qbinom{n}{\lambda_1,\lambda_2,\dots,\lambda_r}
\coloneqq P_{\fD_\lambda}(q)\]
is called the \term{$q$-multinomial coefficient}.
In particular, when $\lambda=(n-k,k)$ is of length $2$,
the Poincar\'e polynomial of $\fD_{(n-k,k)}$ is given by
\term{a $q$-binomial coefficient}
\[\qbinom{n}{k}\coloneqq\qbinom{n}{n-k,k}=\frac{[n][n-1]\dotsm[n-k+1]}{[k]!}.\]

\subsection{Combinatorics on tableaux}
The \term{Young diagram} of a composition $\lambda$ is defined by
\[Y(\lambda)\coloneqq\set{(i,j)}{1\le i,\,1\le j\le\lambda_i}.\]
We represent it by boxes placed in the fourth quadrant
arranged as matrix indices (the English notation):
\[(3,2)=\yng(3,2),\quad(2,4,1)=\yng(2,4,1),\quad
(2,0,3)=\,\vcenter to 30pt{\hbox{\yng(2)}\vfil\hbox{\yng(3)}}\,.\]
A \term{tableau} of shape $\lambda$ is a function
$\sT\colon Y(\lambda)\to\{1,2,\dots\}$.
The \term{weight} of a tableau $\sT$ is a composition $\mu=(\mu_1,\mu_2,\dots)$
whose $i$-th component is $\mu_i\coloneqq\#\sT^{-1}(i)$.
$\sT$ is said to be \term{row-semistandard} if it satisfies
$\sT(i,j)\leq\sT(i,j+1)$ for each pair of adjacent boxes $(i,j),(i,j+1)\in Y(\lambda)$, that is, entries in each row of $\sT$ are weakly increasing.
We denote by $\Tab_{\lambda;\mu}$ the set of row-semistandard tableaux
of shape $\lambda$ and weight $\mu$.
For example,
\[\Tab_{(2,3);(3,1,0,1)}=\left\{
\young(11,124),\young(12,114),\young(14,112),\young(24,111)
\right\}.\]
A row-semistandard tableau is called a \term{row-standard tableau}
if its weight is $(1^n)=(1,1,\dots,1)$.
We denote by $\Tab_\lambda\coloneqq\Tab_{\lambda;(1^n)}$
the set of row-standard tableaux of shape $\lambda$.

The set $\fD_\lambda$ is in bijection with the set $\Tab_\lambda$
by the following correspondence:
for a row-standard tableau $\sT$,
we obtain a permutation $d(\sT)\in\fD_\lambda$
by reading its entries from left to right for each rows from top to bottom.
For example,
\[\sT=\young(1245,378,6)
\quad\text{corresponds to}\quad
d(\sT)=\begin{pmatrix}1&2&3&4&5&6&7&8\\1&2&4&5&3&7&8&6\end{pmatrix}=s_3s_4s_6s_7.\]
Actually any tableau of weight $(1^n)$ provides a permutation
in this manner, and the increasing condition on the rows just say that
this permutation is in $\fD_\lambda$.
We denote by $\varpi_\lambda$ the longest element in $\fD_\lambda$.
Its corresponding tableau is obtained by putting numbers
on $Y(\lambda)$ from bottom to top, conversely as before.
For each $\sT\in\Tab_\lambda$, let us write $\ell(\sT)\coloneqq\ell(d(\sT))$ for short
which we also call the \term{length} of $\sT$.
$\ell(\sT)$ can be also expressed as the inversion number
\[\ell(\sT)=\set[\big]{\bigl((i,j),(k,l)\bigr)}{i<k\text{ and }T(k,l)<T(i,j)}.\]

Next let us take another composition $\mu$ and consider
the action $\fS_\mu\curvearrowright\fS_n/\fS_\lambda$.
For $\sS\in\Tab_{\lambda;\mu}$,
we denote by $\Tab_\sS$ the set $\set{\sT\in\Tab_\lambda}{\sT|_\mu=\sS}$
where $\sT|_\mu$ is a row-semistandard tableau of weight $\mu$
obtained from $\sT$ by replacing its entries
$1,2,\dots,\mu_1$ by $1$, $\mu_1+1,\dots,\mu_1+\mu_2$ by $2$, and so forth.
For example, for
\[\sS=\young(1123,144,3),\]
we have
\[\Tab_\sS=\left\{\young(1245,378,6),\young(1345,278,6),\young(2345,178,6),
\young(1246,378,5),\young(1346,278,5),\young(2346,178,5)\right\}.\]
Then via the one-to-one correspondence $\Tab_\lambda\onetoone\fS_n/\fS_\lambda$,
each subset $\Tab_\sS\subset\Tab_\lambda$ clearly
corresponds to each orbit of the action above.
Hence the set $\Tab_{\lambda;\mu}$ is in bijection with the set
$\fD_\lambda\cap\fD_\mu^{-1}$.
Namely, for each $\sS\in\Tab_{\lambda;\mu}$,
there is a unique tableau $\sS_\downarrow\in\Tab_\sS$
which has the minimal length, so that
$d(\sS_\downarrow)\in\fD_\lambda\cap\fD_\mu^{-1}$.
We can construct $\sS_\downarrow$ from $\sS$ in the following manner:
first we mark subscripts $1,2,\dots,\mu_k$ to all $k$'s in $\sS$
for each number $k$ along with the above reading order.
Then $\sS_\downarrow$ is obtained by
replacing the entries of $\sS$ by $1,2,\dots,n$ with respect to the total order
\[1_1<1_2<\dots<1_{\mu_1}<2_1<2_2<\dots<2_{\mu_2}<\dotsb.\]
For example, the row-semistandard tableau $\sS$ above is marked as
\[\def\Xaa{1_1}\def\Xab{1_2}\def\Xba{2_1}\def\Xca{3_1}
\def\Xac{1_3}\def\Xda{4_1}\def\Xdb{4_2}\def\Xcb{3_2}
\young(\Xaa\Xab\Xba\Xca,\Xac\Xda\Xdb,\Xcb)\]
and gives the corresponding row-standard tableau
$\sS_\downarrow=\sT$ in the previous example; so $d(\sS_\downarrow)=d(\sT)=s_3 s_4 s_6 s_7$.
Other elements in $\Tab_\sS$ can be constructed from $\sS_\downarrow$ as follows:
let $\#_{ij}(\sS)$ be the number of $j$'s in the $i$-th row of $\sS$,
and $\sS[j]$ be the composition of $\mu_j$ defined by $\sS[j]_i\coloneqq\#_{ij}(\sS)$.
We define $\fD_\sS\subset\fS_n$ by
\[\fD_\sS\coloneqq\fD_{\sS[1]}\times\fD_{\sS[2]}\times\dots\times\fD_{\sS[r]}
\subset\fS_\mu\subset\fS_n.\]
Then we have a one-to-one correspondence
$\fD_\sS\to\Tab_\sS$; $w\mapsto w\sS_\downarrow$
which preserves lengths, that is, $\ell(w\sS_\downarrow)=\ell(w)+\ell(\sS_\downarrow)$.
Let $\varpi_\sS\in\fD_\sS$ be its longest element
$\varpi_\sS\coloneqq(\varpi_{\sS[1]},\varpi_{\sS[2]},\dots,\varpi_{\sS[r]})$.
The tableau $\sS^\uparrow\coloneqq\varpi_\sS\sS_\downarrow\in\Tab_\sS$
which has maximal length is obtained by
replacing the entries of $\sS$ from bottom to top, contrary to $\sS_\downarrow$.

The matrix $(\#_{ij}(\sS))_{i,j\ge1}$ uniquely determines a row-semistandard tableau $\sS$,
and its shape $\lambda$ and its weight $\mu$ are recovered from this matrix as
\[\lambda_i=\sum_j\#_{ij}(\sS)\quad\text{and}\quad
\mu_j=\sum_i\#_{ij}(\sS).\]
So for each $\sS$,
there exists a unique tableau $\sS^*$ of shape $\mu$ and of weight $\lambda$,
which satisfies $\#_{ij}(\sS^*)=\#_{ji}(\sS)$.
We call it the \term{dual tableau} of $\sS$.
For example, for $\sS$ above, its dual is
\[\sS^*=\young(112,1,13,22).\]
It easily follows that $d((\sS^*)_\downarrow)=d(\sS_\downarrow)^{-1}$.
So taking dual $\Tab_{\lambda;\mu}\to\Tab_{\mu;\lambda}$; $\sS\mapsto\sS^*$
corresponds to the inversion
$\fD_\lambda\cap\fD_\mu^{-1}\to\fD_\mu\cap\fD_\lambda^{-1}$; $w\mapsto w^{-1}$
via the bijection $d$.

\subsection{The Iwahori--Hecke algebra}
\label{sub:parabolic_modules}

Hereafter we fix a parameter $q\in\kk$.
For each $n\in\NN$, the \term{Iwahori--Hecke algebra} $H_n=H_n(q)$ of rank $n$
(or of type $\mathsf{A}_{n-1}$)
is an algebra generated by elements $T_1,T_2,\dots,T_{n-1}$
with defining relations
\[T_i T_j = T_j T_i \text{\quad if } |i-j|\ge2,\qquad
T_i T_{i+1} T_i = T_{i+1} T_i T_{i+1},\qquad
(T_i-q)(T_i+1) = 0.\]
Here for $n=0$ or $1$, it is defined as $H_0=H_1=\kk$.
For each $w\in\fS_n$, we take a reduced expression $w=s_{i_1}s_{i_2}\dotsm s_{i_r}$
and define an element $T_w\coloneqq T_{i_1}T_{i_2}\dotsm T_{i_r}$ of $H_n$.
Then it is known that it does not depend on choice of expression,
and that $H_n$ is a free $\kk$-module with basis $\set{T_w}{w\in\fS_n}$.
Thus $H_n$ can be considered as a $q$-deformation of $\kk\fS_n$,
the group ring of the symmetric group.
The element $T_w$ is invertible if and only if $q\in\kk^\times$;
in such a case, we have $T_i^{-1}=q^{-1}(T_i-q+1)$
and $T_w^{-1}=T_{i_r}^{-1}\dotsm T_{i_2}^{-1}T_{i_1}^{-1}$.
By definition, if $u,v\in\fS_n$ satisfy $\ell(uv)=\ell(u)+\ell(v)$
then $T_{uv}=T_u T_v$.
The algebra $H_n$ has an anti-involution defined by
$(T_w)^*\coloneqq T_{w^{-1}}$.
Thus the category of left $H_n$-modules is equivalent
to that of right modules.

For a composition $\lambda=(\lambda_1,\lambda_2,\dots,\lambda_r)$ of $n$,
let $H_\lambda$ be a subalgebra of $H_n$ spanned by $\set{T_w}{w\in\fS_\lambda}$.
Then $H_n$ is free as a right $H_\lambda$-module with basis $\set{T_w}{w\in\fD_\lambda}$
by the decomposition $\fS_n=\fD_\lambda\fS_\lambda$.
As an abstract algebra, we have an isomorphism
\[H_\lambda\simeq H_{\lambda_1}\otimes H_{\lambda_2}\otimes\dots\otimes H_{\lambda_r}.\]
It is called a \term{parabolic subalgebra} of $H_n$.

In representation theory of the symmetric groups and the Iwahori--Hecke algebra,
it is important to treat modules over these algebras for all ranks at once.
So it is better to consider the direct sum of all their module categories.
Convolution product of modules is defined as a binary operation on this category.

\begin{definition}
Let $\lambda=(\lambda_1,\lambda_2,\dots,\lambda_r)$ be a composition of $n$.
For each $i=1,2,\dots,r$, let $V_i$ be an $H_{\lambda_i}$-module.
We define an $H_n$-module
\[V_1*V_2*\dots*V_r\coloneqq H_n\otimes_{H_\lambda}(V_1\boxtimes V_2\boxtimes\dotsb\boxtimes V_r)\]
where $\boxtimes$ denotes the outer tensor product of modules.
It is called the \term{convolution product} of $V_1,V_2,\dots,V_r$.
\end{definition}
Obviously this product is associative up to natural isomorphism.
By the basis theorem, we have a direct sum decomposition
\[V_1*V_2*\dots*V_r=\bigoplus_{w\in\fD_\lambda}T_w(V_1\boxtimes V_2\boxtimes\dotsb\boxtimes V_r)\]
as a $\kk$-module.
The convolution product $*$ defines a structure of
tensor category on the direct sum of the module categories $\bigoplus_n(\lMod{H_n})$.
This tensor category also admits a braiding
\begin{align*}
\sigma(V,W)\colon V*W&\to W*V\\
x\boxtimes y&\mapsto T_{\varpi_{(n,m)}}(y\boxtimes x)
\end{align*}
in a weak sense;
it satisfies the hexagon axioms of braiding
but is not invertible unless $q\in\kk^\times$.
Here $\varpi_{(n,m)}$ is the longest element in $\fD_{(n,m)}$ defined by
\[\varpi_{(n,m)}(i)\coloneqq\begin{cases*}
i+m & if $1\le i\le n$,\\
i-n & if $n+1\le i\le m+n$.
\end{cases*}\]
The hexagon axioms follow from the decompositions
\begin{align*}
\varpi_{(n+p,m)}&=(\varpi_{(n,m)},1_p)\cdot(1_n,\varpi_{(p,m)}),&
\varpi_{(p,m+n)}&=(1_m,\varpi_{(p,n)})\cdot(\varpi_{(p,m)},1_n)
\end{align*}
which preserve lengths.
Here we denote by $1_n$ the unit element of $\fS_n$.

\subsection{Parabolic modules and the \texorpdfstring{$q$-Schur}{q-Schur} algebra}
Let $\lambda$ be a composition.
We define an element $m_\lambda\in H_\lambda$ by
\[m_\lambda\coloneqq\sum_{w\in\fS_\lambda}T_w.\]
Note that $T_i(1+T_i)=(1+T_i)T_i=q(1+T_i)$.
Hence $m_\lambda$ satisfies
$T_w m_\lambda=m_\lambda T_w=q^{\ell(w)}m_\lambda$ for all $w\in\fS_\lambda$
since it can be also written as
\[m_\lambda=\sum_{\substack{w\in\fS_\lambda,\\\ell(s_i w)>\ell(w)}}(1+T_i)T_w
=\sum_{\substack{w\in\fS_\lambda,\\\ell(w s_i)>\ell(w)}}T_w(1+T_i)\]
for each $s_i\in\fS_\lambda$.
In particular, $\kk m_\lambda$ is a 2-sided ideal of $H_\lambda$.

Let $M_\lambda\coloneqq H_n m_\lambda$ be a left ideal of $H_n$ generated by $m_\lambda$,
which we call a \term{parabolic module}.
In particular, the \term{trivial module} $\unit_n\coloneqq M_{(n)}$
is a free $\kk$-module of rank one spanned by $m_n\coloneqq m_{(n)}$,
on which every $T_w$ acts by a scalar $q^{\ell(w)}$.
Since the action $H_n\curvearrowleft H_\lambda$ is free,
$M_\lambda$ is isomorphic to $H_n\otimes_{H_\lambda}\kk m_\lambda$ as an $H_n$-module;
so it has a basis $\set{T_w m_\lambda}{w\in\fD_\lambda}$ over $\kk$.
Or equivalently, by using convolution product, we can also represent it as
$M_\lambda\simeq\unit_{\lambda_1}*\unit_{\lambda_2}*\dots*\unit_{\lambda_r}$.
Elements of $M_\lambda\subset H_n$ are characterized as
\[M_\lambda=\set{x\in H_n}{x T_w=q^{\ell(w)}x\text{ for all }w\in\fS_\lambda}\]
because for $x=\sum_{w\in\fS_n} x_w T_w$ ($x_w\in\kk$), $x T_i=qx$ is equivalent
to that $x_w=x_{w s_i}$ for all $w\in\fS_n$.
For each $w\in\fD_\lambda$, we take the corresponding row-standard tableau $\sT$
such that $w=d(\sT)$ and write $m_\sT\coloneqq T_w m_\lambda$.
The action of $H_n$ on it is described as follows:
suppose each number $i$ is contained in the $r(i)$-th row of $\sT$.
Then
\[T_i\cdot m_\sT=\begin{cases*}
q m_\sT & if $r(i)=r(i+1)$,\\
m_{s_i\sT} & if $r(i)<r(i+1)$,\\
q m_\sT + (q-1)m_{s_i\sT} & if $r(i)>r(i+1)$.
\end{cases*}\]
We similarly define right ideals $M^*_\lambda\coloneqq m_\lambda H_n$
and $\unit^*_n\coloneqq M^*_{(n)}$. Then we have
\[M^*_\lambda=\set{x\in H_n}{T_w x=q^{\ell(w)}x\text{ for all }w\in\fS_\lambda}.\]

Now take two compositions $\lambda,\mu$ of $n$.
Since $M_\mu$ is a cyclic module generated by $m_\mu$
with the relations $T_w m_\mu=q^{\ell(w)}m_\mu$
for every $w\in\fS_\mu$, by taking the image of the generator $m_\mu$
we have an isomorphism
\[\Hom_{H_n}(M_\mu,M_\lambda)\simeq\set{x\in M_\lambda}
{T_w x=q^{\ell(w)}x \text{ for all } w\in\fS_\mu}
= M_\lambda\cap M^*_\mu.\]
Let us write $M_{\lambda;\mu}\coloneqq M_\lambda\cap M^*_\mu$.
The collection of these $\kk$-modules has a natural product
\begin{align*}
\circ_\mu\colon M_{\mu;\nu}\otimes M_{\lambda;\mu}&\to M_{\lambda;\nu},\\
x m_\mu \otimes m_\mu y &\mapsto x m_\mu y.
\end{align*}
According to the isomorphism above, this product corresponds to
the opposite of the composition of homomorphisms.
Note that $M_{\lambda;\mu}$ naturally acts on the parabolic module $M_\mu$ from right,
so that the composition is given by the reversed product.
On the other hand, $M_{\lambda;\mu}$ is also isomorphic to
$\Hom_{H_n^\op}(M^*_\lambda,M^*_\mu)$, the set of homomorphisms between right modules.
In this view, the product is just same as the composition of such homomorphisms.
Anyway, the algebra with this product
\[\rS_{r,n}\coloneqq\bigoplus_{\lambda,\mu}M_{\lambda;\mu}\simeq
\End_{H_n}\Bigl(\bigoplus_\lambda M_\lambda\Bigr)^\op\simeq
\End_{H_n^\op}\Bigl(\bigoplus_\lambda M^*_\lambda\Bigr)\]
is called the \term{$q$-Schur algebra} introduced
by Dipper and James~\cite{DipperJames89}.
Here $\lambda=(\lambda_1,\dots,\lambda_r)$
and $\mu=(\mu_1,\dots,\mu_r)$ runs over all compositions of $n$
whose components are zero except for the first $r$ ones.
Note that the Iwahori--Hecke algebra itself can be obtained similarly:
\[H_n=M_{(1^n);(1^n)}\simeq\End_{H_n}(M_{(1^n)})^\op\simeq\End_{H_n^\op}(M^*_{(1^n)}).\]

Since we can write
\[M_{\lambda;\mu}=\set{x\in H_n}{T_v x T_w=q^{\ell(v)+\ell(w)}x\text{ for all }
v\in\fS_\mu,w\in\fS_\lambda},\]
it has a basis
$\set{\sum_{v\in\fS_\mu w\fS_\lambda}T_v}{w\in\fD_\lambda\cap\fD_\mu^{-1}}$
which corresponds to the double cosets $\fS_\mu\backslash\fS_n/\fS_\lambda$.
Similarly as before, for $w\in\fD_\lambda\cap\fD_\mu^{-1}$ we take the corresponding
row-semistandard tableau $\sS\in\Tab_{\lambda;\mu}$ such that $w=d(\sS_\downarrow)$
and write $m_\sS\coloneqq\sum_{v\in\fS_\mu w\fS_\lambda}T_v$.
As an element of $M_\lambda$, we can decompose it as
$m_\sS=\sum_{\sT\in\Tab_\sS}m_\sT$.
The anti-involution on $H_n$ induces a map
\[{\bullet}^*\colon M_{\lambda;\mu}\to M_{\mu;\lambda}\]
which induces that on $\rS_{r,n}$.
By definition we have $(m_\sS)^*=m_{\sS^*}$.

\subsection{Decomposing a tableau}
In this subsection we observe that for each $\sS\in\Tab_{\lambda;\mu}$,
$m_\sS\in M_{\lambda;\mu}$ has a canonical decomposition
\[m_\sS=m_\mu\circ_\nu m_{P_{w,\nu}}\circ_{w\nu} m_\lambda\]
into three tableaux.
We first explain each of these terms.

Let $\mu$ and $\nu$ be compositions of $n$.
We say that $\nu$ is a \term{refinement} of $\mu$
when there is an increasing sequence of indices $1\le a_1\le a_2\le\dotsb$
such that $\mu_i=\sum_{a_i\le j<a_{i+1}}\nu_j$.
Clearly it is equivalent to that $\fS_\nu\subset\fS_\mu$.
Hence $m_\mu$ is contained in both $M_{\mu;\nu}$ and $M_{\nu;\mu}$.
As elements of these sets, $m_\mu$ is respectively represented by tableaux $\sS$
and its dual $\sS^*$ defined by $\sS^*(j,k)\coloneqq i$ for $a_i\le j<a_{i+1}$,
such as
\[\sS=\young(1223,44455)\qquad\text{and}\qquad\sS^*=\young(1,11,1,222,22)\]
for $\mu=(4,5)$ and $\nu=(1,2,1,3,2)$.
For $\sT\in\Tab_{\lambda;\nu}$ of weight $\nu$,
let $\sT|_\mu\in\Tab_{\lambda;\mu}$ be a row-standard tableau of weight $\mu$
obtained by replacing each entry $j$ in $\sT$ such that $a_i\le j<a_{i+1}$ with $i$,
similarly as before.
Since $(1^n)$ is a refinement of every composition, this notation coincides
with the previous one.

\begin{lemma}\label{lem:multiply_refinement_tableau}
Let $\nu$ be a refinement of $\mu$, and take $a_1\le a_2\le\dotsb$ as above.
\begin{enumerate}
\item For $\sS\in\Tab_{\lambda;\mu}$, we have
\[m_\mu\circ_\mu m_\sS=\sum_{\sT\in\Tab_{\lambda;\nu},\sT|_\mu=\sS}m_\sT
\in M_{\lambda;\nu}\]
where $m_\mu$ is regarded as an element of $M_{\mu;\nu}$.
\item For $\sT\in\Tab_{\lambda;\nu}$, we have
\[m_\mu\circ_\nu m_\sT=\left(\prod_i q^{\ell_i}
\prod_k\qbinom{\#_{ki}(\sT|_\mu)}{\#_{k a_i}(\sT),\#_{k,a_i+1}(\sT),\dots,\#_{k,a_{i+1}-1}(\sT)}
\right)\,m_{\sT|_\mu}.\]
Here $\ell_i\coloneqq\#\set[\big]{\bigl((k,l),(k',l')\bigr)}
{k<k'\text{ and }a_i\le\sT(k',l')<\sT(k,l)<a_{i+1}}$
is the inversion number of $\sT$ for entries $j$ such that $a_i\le j<a_{i+1}$.
\end{enumerate}
\end{lemma}
\begin{proof}
(1). By definition, as an element of $M_\lambda$, $m_\mu\circ_\mu m_\sS$
is just $m_\sS=\sum_{\sR\in\Tab_\sS}m_\sR$.
Hence the formula is clear from that $(\sR|_\nu)|_\mu=\sR|_\mu$.

(2). Let us write $\sS\coloneqq\sT|_\mu$.
Since $M_{\lambda;\mu}$ is free over $\kk$,
it suffices to prove for the universal case $\kk=\ZZ[q]$ where $q$ is an indeterminate.
First we compute an ordinary product $m_\mu\cdot m_\sT$ in $M_\lambda$.
We can take $w\in\fS_\mu$ such that $\sT_\downarrow=w\cdot\sS_\downarrow$,
then $\ell(w)=\sum_i\ell_i$.
Since $m_\sT=\sum_{v\in\fD_\sT}m_{v\cdot\sT_\downarrow}$
and $\fD_\sT\subset\fS_\mu$,
\[m_\mu\cdot m_\sT
=q^{\ell(w)}P_{\fD_\sT}(q)m_\mu\cdot m_{\sS_\downarrow}
=q^{\ell(w)}P_{\fD_\sT}(q)\Bigl(\prod_{k,i}[\#_{ki}(\sS)]!\Bigr)m_\sS.\]
On the other hand, we have $m_\mu\cdot m_\sT=(\prod_j[\nu_j]!)m_\mu\circ_\nu m_\sT$.
Since $M_\lambda$ is a free module over an integral domain $\ZZ[q]$,
we can cancel this coefficient.
Thus the formula follows from
\[\frac{P_{\fD_\sT}(q)\prod_{k,i}[\#_{ki}(\sS)]!}{\prod_j[\nu_j]!}
=\frac{\prod_{k,i}[\#_{ki}(\sS)]!}{\prod_{k,j}[\#_{kj}(\sT)]!}
=\prod_k\qbinom{\#_{ki}(\sS)}{\#_{k a_i}(\sT),\dots,\#_{k,a_{i+1}-1}(\sT)}.
\qedhere\]
\end{proof}

Next we introduce the middle term of the decomposition.
\begin{definition}
Let $\nu=(\nu_1,\nu_2,\dots,\nu_r)$ be a composition of $n$ and
$w\in\fS_r$. Let us write $w\nu\coloneqq(\nu_{w(1)},\nu_{w(2)},\dots,\nu_{w(r)})$.
We define $\sP_{w,\nu}\in\Tab_{w\nu;\nu}$ by
\[\sP_{w,\nu}(i,j)\coloneqq w(i)\]
and call it the \term{permutation tableau} with respect to $w$.
\end{definition}

The composition with a permutation tableau is complicated in general,
so we prove a multiplication formula only for a special case.

\begin{lemma}\label{lem:multiply_permutation_tableau}
Let $\nu$ and $w$ as above.
Suppose $\sT\in\Tab_{\lambda;w\nu}$ satisfies that
for each pair of boxes $(i,j),(k,l)\in Y(\lambda)$,
$i\le k$ and $\sT(k,l)<\sT(i,j)$ implies $w(\sT(k,l))<w(\sT(i,j))$.
Then we have $m_{\sP_{w,\nu}}\circ_{w\nu}m_\sT=m_{w\sT}$.
\end{lemma}
\begin{proof}
The tableau $w\sT$ is also row-standard by the assumption.
By the definition of permutation tableau,
there is a permutation $v\in\fS_n$ such that $m_{\sP_{w,\nu}}=T_v m_{w\nu}$.
The formula follows from that every $\sR\in\Tab_\sT$ satisfies
$v\sR\in\Tab_{w\sT}$ and $\ell(v\sR)=\ell(v)+\ell(\sR)$.
\end{proof}

\begin{proposition}\label{prop:canonical_tableau_decomposition}
For each $\sS\in\Tab_{\lambda;\mu}$,
there exists a unique pair $(\nu,w)$
of a composition $\nu=(\nu_1,\nu_2,\dots,\nu_r)$ with $\nu_1,\nu_2,\dots,\nu_r>0$
and a permutation $w\in\fS_r$ such that
$w\nu$ and $\nu$ are respectively refinements of $\lambda$ and $\mu$, and
\[m_\sS=m_\mu\circ_\nu m_{P_{w,\nu}}\circ_{w\nu} m_\lambda.\]
\end{proposition}
\begin{proof}
For such $\sS$, it suffices to put
\begin{align*}
\nu&\coloneqq(\#_{11}(\sS),\#_{21}(\sS),\dots,\#_{12}(\sS),\#_{22}(\sS),\dots,\#_{13}(\sS),\#_{23}(\sS),\dotsc),\\
w\nu&\coloneqq(\#_{11}(\sS),\#_{12}(\sS),\dots,\#_{21}(\sS),\#_{22}(\sS),\dots,\#_{31}(\sS),\#_{32}(\sS),\dotsc)
\end{align*}
with removing zero entries $\#_{ij}(\sS)=0$, and take the corresponding permutation $w$.
Then by the two lemmas above we have a desired decomposition. For example,
\[\young(22333,1111,133)=\young(1111,1,22,333,33)\circ_{(4,1,2,3,2)}
\young(33,444,1111,2,55)\circ_{(2,3,4,1,2)}\young(11222,3333,455)\]
where we represent an element $m_\sT$ by the tableau $\sT$ itself for short.
The uniqueness is obvious from this construction.
\end{proof}

\subsection{Good tableaux}

We introduce a partial order $\le$ on the set of compositions of $n\in\NN$
called the \term{dominance order}.
Here for two compositions $\lambda$ and $\mu$,
they are defined to be $\lambda\le\mu$ if and only if
\[\lambda_1+\lambda_2+\dots+\lambda_k\le\mu_1+\mu_2+\dots+\mu_k\]
is satisfied for each $k\in\NN$.
It is not a total order; for example, the compositions
$(3,3)$ and $(4,1,1)$ are incomparable.
According to the \emph{reversed} dominance order,
we make a filtration on the module category as we did in the previous part.
For each composition $\lambda$, the set $\set{\mu}{\mu>\lambda}$ is finite.
Hence the set of all compositions with the reversed dominance order
is a well-founded partially ordered set.

\begin{notation}
Let $X,Y\in\lMod{H_n}$. For a composition $\lambda$, let
\[\cH^\lambda(X,Y)\coloneqq\Hom_{H_n}(M_\lambda,Y)\circ\Hom_{H_n}(X,M_\lambda)\]
be the set of homomorphisms which factor through $M_\lambda$.
In other words, $\cH^\lambda$ is a 2-sided ideal of $\lMod{H_n}$ generated by $M_\lambda$.
By using the dominance order we define
\[\cH^{\ge\lambda}(X,Y)\coloneqq\sum_{\mu\ge\lambda}\cH^\mu(X,Y),\quad
\cH^{>\lambda}(X,Y)\coloneqq\sum_{\mu>\lambda}\cH^\mu(X,Y)\]
and
\[\Hom_{H_n}^{(\lambda)}(X,Y)\coloneqq\Hom_{H_n}(X,Y)\big/\cH^{>\lambda}(X,Y).\]
The last one is a hom set in the quotient category $(\lMod{H_n})/\cH^{>\lambda}$.

When $X$ and $Y$ above are parabolic modules, we write these
submodules or quotient modules of $M_{\lambda;\mu}$
as $M^\nu_{\lambda;\mu}$, $M^{\ge\nu}_{\lambda;\mu}$, $M^{>\nu}_{\lambda;\mu}$
and $M^{(\nu)}_{\lambda;\mu}$ respectively.
In particular,
\[M_{\lambda;\mu}^{(\nu)}\simeq\Hom_{H_n}^{(\nu)}(M_\mu,M_\lambda)\]
is the $\kk$-module equipped with the reversed composition as product.
As its special case we let
$S_{\lambda;\mu}\coloneqq M_{\lambda;\mu}^{(\lambda)}$.
Then $S_{\lambda;\lambda}$ is a quotient algebra of $M_{\lambda;\lambda}$
and $S_{\lambda;\mu}$ is a \emph{right} module over this algebra.
When $\mu=(1^n)$ we simply write $S_\lambda\coloneqq S_{\lambda;(1^n)}$.
$S_\lambda$ is also a \emph{left} module over $H_n\simeq M_{(1^n);(1^n)}$
and called the \term{Specht module}.
We denote equalities in the quotient set $S_{\lambda;\mu}$
by the symbol $\equiv$.
\end{notation}

Note that if a composition $\lambda=(\lambda_1,\lambda_2,\dotsc)$
has $\lambda_i=0$ such that $\lambda_{i+1}\neq0$,
letting $\tilde\lambda\coloneqq(\lambda_1,\dots,\lambda_{i-1},\lambda_{i+1},\dotsc)$
we have $\lambda<\tilde\lambda$ and $M_\lambda\simeq M_{\tilde\lambda}$.
Hence for such $\lambda$,
$M_\lambda$ is zero in the quotient category $(\lMod{H_n})/\cH^{>\lambda}$;
in particular we have $S_{\lambda;\mu}=0$ for all $\mu$.
We can remove such needless
compositions from the index set. Then the rest is now a finite set.

For a while we fix $n\in\NN$ and $\lambda,\mu$ denote compositions of $n$.
In order to study this quotient category,
we introduce a combinatorial notion on tableaux as follows.

\begin{definition}
Let $\sT\in\Tab_{\lambda;\mu}$ be a row-semistandard tableau.
We say that a box $(i,j)\in Y(\lambda)$ in the Young diagram
is \term{good} if it satisfies $\sT(i,j)\ge i$,
and $\sT$ is said to be \term{good}
if all boxes in $Y(\lambda)$ are good.
\end{definition}

\begin{lemma}\label{lem:good_tableaux}
$S_{\lambda;\mu}$
is spanned by $\set{m_\sT}{\sT\in\Tab_{\lambda;\mu}\textup{ which is good}}$.
\end{lemma}
\begin{proof}
Suppose that $\sT$ is not good.
For such $\sT$, let us define a tableau $\sT_1$ of shape $\lambda$ by
\[\sT_1(i,j)\coloneqq\min\{i,\sT(i,j)\}.\]
Next let $\sT_2$ be a tableau
obtained by moving up all ungood boxes $\young(i)$ of $\sT$
to its $i$-th row, so that $\sT_2$ is good. For example, when
\[\sT=\young(111123,2233,123)\]
which has ungood $1$ and $2$ in the third row, we let
\[\sT_1=\young(111111,2222,123)\quad\text{and}\quad
\sT_2=\young(1111123,22233,3).\]
Let $\nu$ be the weight of $\sT_1$, which is equal to the shape of $\sT_2$.
For each $k$ we have
\[\nu_1+\nu_2+\dots+\nu_k=\lambda_1+\lambda_2+\dots+\lambda_k+
\#\set{(i,j)\in Y(\lambda)}{i>k\text{ and }T(i,j)\le k}.\]
Since $\sT$ is not good, we have $\nu>\lambda$
so that $m_{\sT_2}\circ_\nu m_{\sT_1}\equiv0$ in $S_{\lambda;\mu}$.

On the other hand, 
observe that the $i$-th row of $\sT_2$ is obtained by
reading entries of $\sT$ at boxes $(k,l)$
such that $\sT_1(k,l)=i$ from bottom to top.
So taking $w\coloneqq d(\sT_2^\uparrow)\in\fD_\nu$ we have
$\sT^\uparrow=w\sT_1^\uparrow$ and
$\ell(\sT^\uparrow)=\ell(w)+\ell(\sT_1^\uparrow)$.
This induces the following decomposition in $M_{\lambda;\mu}$:
\[m_{\sT_2}\circ_\nu m_{\sT_1}=m_\sT+
\sum_{\sS\in\Tab_{\lambda;\mu},\,\ell(\sS^\uparrow)<\ell(\sT^\uparrow)}c_\sS m_\sS
\qquad(c_\sS\in\kk).\]
Hence in $S_{\lambda;\mu}$
we can replace ungood $m_\sT$ by a linear combination of elements $m_\sS$
which has smaller lengths. Consequently it inductively follows that
any tableau can be written as a linear combination of good ones.
\end{proof}

\begin{lemma}\label{lem:good_tableaux_on_dominance_order}
\begin{enumerate}
\item $\Tab_{\lambda;\lambda}$ has only one good tableau.
\item There are no good tableau in $\Tab_{\lambda;\mu}$
unless $\lambda\ge\mu$.
\end{enumerate}
\end{lemma}
\begin{proof}
If $\sT\in\Tab_{\lambda;\mu}$ is good, then for each $k$,
all $i$'s in $\sT$ less than or equal to $k$ are placed in its $k$-th row or upper.
The number of such numbers ($=\mu_1+\dots+\mu_k$) must be
equal to or less than that of such boxes ($=\lambda_1+\dots+\lambda_k$)
so we have $\lambda\ge\mu$.
Moreover if $\lambda=\mu$, all $i$'s in $\sT$ must be in its $i$-th row.
\end{proof}

By these two lemmas, the statements below are obvious.

\begin{corollary*}\label{cor:hom_on_dominance_order}
\begin{enumerate}
\item $S_{\lambda;\lambda}$ is spanned by $m_\lambda$.
Hence it is isomorphic to a quotient ring of $\kk$.
\item $S_{\lambda;\mu}=0$ unless $\lambda\ge\mu$.
\qedhere
\end{enumerate}
\end{corollary*}

Hence the category $\lMod{H_n}$ with objects $\{M_\lambda\}$ and algebras
$\kk m_\lambda\subset M_{\lambda;\lambda}$
satisfies the assumptions in Lemma~\ref{lem:category_produces_filter}
(see Remark~\ref{rem:well_founded_filtered_algebra}),
so it produces several standardly filtered algebras.
\begin{theorem*}
The Iwahori--Hecke algebra $H_n=M_{(1^n;1^n)}$ and the $q$-Schur algebra
$\rS_{r,n}=\bigoplus_{\lambda,\mu}M_{\lambda;\mu}$
are standardly filtered algebras over $\kk$ on the set of compositions.
Here for each composition $\nu$,
their ideal filter and the attached Morita contexts is given by
\[H_n^{\ge\nu}\coloneqq M_{(1^n;1^n)}^{\ge\nu}\quad\text{with}\quad
(S_\nu,S^*_\nu)\]
and
\[\rS_{r,n}^{\ge\nu}\coloneqq\bigoplus_{\lambda,\mu}M_{\lambda;\mu}^{\ge\nu}\quad\text{with}\quad
\Bigl(\bigoplus_\lambda S_{\nu;\lambda},\bigoplus_\lambda S^*_{\nu;\lambda}\Bigr)\]
where $S^*_{\nu;\lambda}\coloneqq M^{(\nu)}_{\lambda;\nu}$
and $S^*_\nu\coloneqq S^*_{\nu;(1^n)}$.
These standard filters are involutive.
\end{theorem*}

It seems to be an interesting problem to determine
the $\kk$-module structure of $S_{\lambda;\mu}$ (or more general $M^{(\nu)}_{\lambda;\mu}$) in detail.
For the case that $q$ is invertible we can completely determine
its structure by taking its free basis as we will study in later subsections.
In the other case the situation is more complicated
so that these modules even need not to be free.
The author conjectures that each $S_{\lambda;\lambda}$
is isomorphic to $\kk$ or $\kk/q^{h(\lambda)}\kk$
for some $h(\lambda)\in\NN$ determined by the shape of $\lambda$,
but a general one is still unable to describe.

\subsection{Local transformations in Specht modules}
In this subsection we prove useful formulas for computation on Specht modules.

\begin{lemma}
Suppose we have an equation
$\sum_\sT c_\sT m_\sT\equiv0$ in $S_{\lambda;\mu}$
for some $c_\sT\in\kk$.
Take an arbitrary sequence $a_1\le a_2\le\dots\le a_k$.
For each $\sT\in\Tab_{\lambda;\mu}$ let $\sT^+\in\Tab_{(k,\lambda);\mu^+}$
be the tableau obtained by adding a new row
$\def\Xa{a_1}\def\Xb{a_2}\def\Xc{a_k}
\young(\Xa\Xb\scdots\Xc)$
at the top of $\sT$; here $(k,\lambda)\coloneqq(k,\lambda_1,\lambda_2,\dotsc)$
and $\mu^+_j=\mu_j+\#\set{i}{a_i=j}$.
Then we have an equation $\sum_\sT c_\sT m_{\sT^+}\equiv0$
in $S_{(k,\lambda);\mu^+}$.
\end{lemma}
\begin{proof}
First note that the convolution functor with trivial module
\[\unit_k*{\bullet}\colon(\lMod{H_n})/\cH^{>\lambda}\to
(\lMod{H_{k+n}})/\cH^{>(k,\lambda)}\]
is still well-defined,
because for any $V\to W$ which factors through
some $M_\nu$ for $\nu>\lambda$,
corresponding $\unit_k*V\to\unit_k*W$
factors through $M_{(k,\nu)}$
with $(k,\nu)>(k,\lambda)$.
For each $\sT$, let us define $\sT^\#\in\Tab_{(k,\lambda);(k,\mu)}$ by
\[\sT^\#(i,j)=\begin{cases*}
1 & if $i=1$,\\
\sT(i-1,j)+1 & otherwise,
\end{cases*}\]
so that $m_{\sT^\#}=\unit_k*m_{\sT}$.
On the other hand, let $\sR\in\Tab_{(k,\mu);\mu^+}$ be the tableau defined by
\[\sR(i,j)=\begin{cases*}
a_j & if $i=1$,\\
i-1 & otherwise.
\end{cases*}\]
Then we have $m_\sR\circ_{(k,\mu)}m_{\sT^\#}=m_{\sT^+}$ by the decomposition of $m_\sR$
according to Proposition~\ref{prop:canonical_tableau_decomposition}
and the formulas in Lemma~\ref{lem:multiply_refinement_tableau}
and Lemma~\ref{lem:multiply_permutation_tableau}.
Hence
\[\sum_\sT c_\sT m_{\sT^+}=m_\sR\circ_{(k,\mu)}\sum_\sT c_\sT m_{\sT^\#}
=m_\sR\circ_{(k,\mu)}\Bigl(\unit_k*\sum_\sT c_\sT m_\sT\Bigr)\equiv0.
\qedhere\]
\end{proof}

By the same argument, we can also add a new row to the bottom of tableaux.
For the bottom row of a tableau we have another kind of formula.

\begin{lemma}
Let $\sum_\sT c_\sT m_\sT\equiv0\in S_{\lambda;\mu}$ as above.
Take a number $a$ which is greater than or equal to
any entries of $\sT$ (so $\mu_i=0$ for $i>a$).
For each $\sT\in\Tab_{\lambda;\mu}$, let $\sT^+\in\Tab_{\lambda^+;\mu^+}$ be
the tableau obtained by joining
a bar $\young(aa\scdots a)$ of length $l$ at the right of the bottom row of $\sT$;
here $\lambda^+\coloneqq(\lambda_1,\dots,\lambda_{r-1},\lambda_r+l)$
and $\mu^+\coloneqq(\mu_1,\dots,\mu_{a-1},\mu_a+l)$.
Then we also have $\sum_\sT c_\sT \qbinom{\#_{ra}(\sT)+l}{l}m_{\sT^+}\equiv0$
in $S_{\lambda^+;\mu^+}$.
\end{lemma}
\begin{proof}
For a composition $\nu=(\nu_1,\dots,\nu_r)$,
we write $(\nu,l)\coloneqq(\nu_1,\dots,\nu_r,l)$.
We define $\sT^\#\in\Tab_{(\lambda,l);\mu^+}$
for each $\sT\in\Tab_{\lambda;\mu}$ by
\[\sT^\#(i,j)\coloneqq\begin{cases*}
\sT(i,j) & if $i\le r$,\\
a & if $i=r+1$
\end{cases*}\]
and $\sR\in\Tab_{\lambda^+;(\lambda,l)}$ by
\[\sR(i,j)\coloneqq\begin{cases*}
i & if $i<r$ or ($i=r$, $j\le\lambda_r$),\\
r+1 & if $i=r$, $j>\lambda_r$,
\end{cases*}\]
so that $m_{\sT^\#}\circ_{(\lambda,l)}m_\sR=\qbinom{\#_{ra}(\sT)+l}{l} m_{\sT^+}$
similarly to the previous proof.
By the similar argument we can prove
$\sum_\sT c_\sT m_{\sT^\#}\equiv0$, and more strongly,
this element can be written as a linear combination of elements
which factor through $M_{(\nu,l)}$ for $\nu>\lambda$.
This implies $(\nu,l)\not\le\lambda^+$;
thus by Corollary~\ref{cor:hom_on_dominance_order},
in $S_{\lambda^+;\mu^+}$ we have
\[\sum_\sT c_\sT \qbinom{\#_{ra}(\sT)+l}{l} m_{\sT^+}
=\sum_\sT c_\sT m_{\sT^\#}\circ_\nu m_\sR\equiv0.
\qedhere\]
\end{proof}

The formula below will be needed for a later computation.

\begin{lemma}\label{lem:scalar_lemma}
Let $k,l,n\in\NN$ such that $k\le l\le n$ and
let $\lambda\coloneqq(n-k,k)$ and $\mu\coloneqq(n-l,l)$.
For each $i$, let $\sT_i\in\Tab_{\lambda;\mu}$ be the tableau determined by
$\#_{21}(\sT_i)=i$, that is, it is in the form
\[\sT_i=\rlap{$\,\underbrace{\phantom{\textstyle\young(111,111)}}_i$}
\young(11\scdots122\scdots2,1\scdots12\scdots2).\]
Then we have $m_{\sT_i}\equiv(-1)^{i} q^{\binom{i}{2}}\qbinom{k}{i}m_{\sT_0}$
in $S_{\lambda;\mu}$.
\end{lemma}
\begin{proof}
We prove it by an induction on $k$.
The case $i=0$ is trivial so assume that $0<i\le k$.
For $i<k$, using the assumption of induction, the formula is implied by the lemma above.
On the other hand, by Lemma~\ref{lem:multiply_refinement_tableau}~(1) we have
\[0\equiv\young(111\scdots12\scdots2)\circ_{(n)}\young(111\scdots\scdots\scdots\scdots1,11\scdots\scdots\scdots1)
=\sum_{0\le i\le k}m_{\sT_i},\]
so that the statement also holds for $i=k$ by the formula
\[\sum_{0\le i\le k}(-1)^i q^{\binom{i}{2}}\qbinom{k}{i}=0
\qquad\text{implied by}\qquad
\prod_{0\le i<k}(1+q^k t)=\sum_{0\le i\le k} q^{\binom{i}{2}}\qbinom{k}{i}t^i.
\qedhere\]
\end{proof}

Multiplying an element to the both-hand sides of this formula for $i=k=l$,
we obtain the following corollary by Lemma~\ref{lem:multiply_refinement_tableau}~(1).

\begin{corollary*}\label{cor:scalar_corollary}
Let $\lambda=(n-k,k)$ as above.
For arbitrary entries $a_1\le\dots\le a_k$, we have
\[\def\Xa{a_1}\def\Xb{a_k}
\young(11\scdots1\Xa\scdots\Xb,1\scdots1)
\equiv(-1)^k q^{\binom{k}{2}}
\young(11\scdots\scdots\scdots\scdots1,\Xa\scdots\Xb).
\qedhere\]
\end{corollary*}

\subsection{Semistandard tableaux}
Hereafter in this section we assume $q\in\kk^\times$.
Then the braiding $\sigma$ of the convolution $*$ is now invertible
so we have $M_\lambda\simeq M_{w\lambda}$ for any
$\lambda=(\lambda_1,\lambda_2,\dots,\lambda_r)$ and $w\in\fS_r$.
Recall that a composition $\lambda$
is called a \term{partition} if it is a descending sequence:
$\lambda_1\ge\lambda_2\ge\dots$.
So in this case, unless $\lambda$ is a partition,
we can take some $w$ such that $\lambda<w\lambda$, so that
$M_\lambda$ is zero in the quotient category $(\lMod{H_n})/\cH^{>\lambda}$ again.

A row-semistandard tableau $\sT\in\Tab_{\lambda;\mu}$ is called
a \term{semistandard tableau} if its shape $\lambda$ is a partition
and for all vertically adjacent boxes $(i,j),(i+1,j)\in Y(\lambda)$
it satisfies $T(i,j)<T(i+1,j)$; or equivalently,
all its columns are strictly increasing.
We denote by $\STab_{\lambda;\mu}$ the set of
all semistandard tableaux of shape $\lambda$ of weight $\mu$.
Note that the strictly increasing condition clearly implies that
every semistandard tableau is good.
Now we can improve a lemma in the previous subsection.

\begin{lemma}\label{lem:semistandard}
$S_{\lambda;\mu}$ is spanned by $\set{m_\sT}{\sT\in\STab_{\lambda;\mu}}$.
\end{lemma}
\begin{proof}
The statement is clear if $\lambda$ is not a partition, so we may assume so.
Suppose $\sT$ is not semistandard and take its box $(k,l)\in Y(\lambda)$
such that $\sT(k,l)\ge\sT(k+1,l)$.
Let $\nu$ be a composition
\[\nu\coloneqq(\lambda_1,\dots,\lambda_{k-1},l-1,\lambda_k+1,\lambda_{k+1}-l,
\lambda_{k+2},\lambda_{k+3},\dotsc).\]
We define tableaux $\sT_1\in\Tab_{\lambda;\nu}$ and $\sT_2\in\Tab_{\nu;\mu}$ by
\begin{align*}
\sT_1(i,j)&=\begin{cases*}
i & if $i<k$ or ($i=k$, $j<l$) or ($i=k+1$, $j\le l$),\\
i+1 & otherwise,
\end{cases*}\\
\sT_2(i,j)&=\begin{cases*}
\sT(i,j) & if $i\le k$ or ($i=k+1$, $j\le l$),\\
\sT(k,j-1) & if $i=k+1$, $j>l$,\\
\sT(k+1,j+l) & if $i=k+2$,\\
\sT(i-1,j) & if $i>k+2$.
\end{cases*}
\end{align*}
For example, when
\[\sT=\young(1112234,12223,25)\]
and $(k,l)=(1,4)$, the corresponding tableaux are
\[\sT_1=\young(1112222,22223,44)\quad\text{and}\quad
\sT_2=\young(111,12222234,3,25).\]
So intuitively $\sT_2$ is obtained by picking up
entries of $\sT$ in the polygonal chain
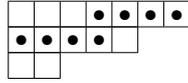
\[\young(\ \ \ \bullet\bullet\bullet\bullet,\bullet\bullet\bullet\bullet\ ,\ \ )\]
which turns at $(k,l)$ and $(k+1,l)$ as a new row.
Then by the same argument in the proof of Lemma~\ref{lem:good_tableaux},
$m_{\sT_2}\circ_\nu m_{\sT_1}\in M_{\lambda;\mu}$ has
the leading term $m_\sT$.
Now let $\nu^+$ be another composition
\[\nu^+\coloneqq(\lambda_1,\dots,\lambda_{k-1},\lambda_k+1,l-1,\lambda_{k+1}-l,
\lambda_{k+2},\dots,\lambda_r)\]
which is obtained by swapping middle two entries of $\nu$.
By the assumption $q\in\kk^\times$, we have $M_\nu\simeq M_{\nu^+}$.
Hence $m_{\sT_2}\circ_\nu m_{\sT_1}$ also factors through $M_{\nu^+}$.
On the other hand, we have clearly $\nu^+\not\le\lambda$.
Hence by Corollary~\ref{cor:hom_on_dominance_order}~(2),
$m_{\sT_2}\circ_\nu m_{\sT_1}\equiv0$ in $S_{\lambda;\mu}$.
By induction on length as before we obtain the statement.
\end{proof}

\begin{theorem}\label{thm:cellular_basis}
Recall the assumption $q\in\kk^\times$. Then
$M_{\lambda;\mu}$ has a basis
\[\bigsqcup_{\nu\colon\textup{partition}}
\set{m_\sS\circ_\nu m_{\sT^*}}{\sS\in\STab_{\nu;\mu},\sT\in\STab_{\nu;\lambda}}.\]
\end{theorem}
\begin{proof}
First we prove that the set above spans the hom space.
Take an appropriate total order on the set of all compositions
$\{\nu_1,\nu_2,\dots,\nu_p=\lambda,\dotsc\}$
which is stronger than the reversed dominance order,
so that $i\le j$ whenever $\nu_i\ge\nu_j$.
We take a filtration on $M_{\lambda;\mu}$ by letting
$M^{\le k}_{\lambda;\mu}\coloneqq\sum_{i\le k}M^{\nu_i}_{\lambda;\mu}$ for each $k$
so that $M_{\lambda;\mu}=M^{\le p}_{\lambda;\mu}$.
Then, on each composition factor,
by inclusion
$M^{>\nu_k}_{\lambda;\mu}\subset M^{\le k-1}_{\lambda;\mu}$
there is a natural surjective map
\[\circ_{\nu_k}\colon S_{\nu_k;\mu}\otimes S^*_{\nu_k;\lambda}\twoheadrightarrow
M^{\ge\nu_k}_{\lambda;\mu}\big/M^{>\nu_k}_{\lambda;\mu}\twoheadrightarrow
M^{\le k}_{\lambda;\mu}\big/M^{\le k-1}_{\lambda;\mu}\]
here recall that $S_{\nu;\mu}=M^{(\nu)}_{\nu;\mu}$ and we define
$S^*_{\nu;\lambda}\coloneqq M^{(\nu)}_{\lambda;\nu}$.
Hence by Lemma~\ref{lem:semistandard}, the right-hand side
is spanned by $\{m_\sS\circ_{\nu_k} m_{\sT^*}\}$ above.

Now remember the Robinson--Schensted--Knuth correspondence~\cite{Knuth70}
\[\Tab_{\lambda;\mu}\onetoone\bigsqcup_{\nu\colon\textup{partition}}
\STab_{\nu;\lambda}\times\STab_{\nu;\mu}.\]
Hence the rank of the free $\kk$-module $M_{\lambda;\mu}$
is equal to the number of elements in the generating set above.
Consequently this set is also linearly independent, so that it forms a basis.
\end{proof}

\begin{corollary*}
\begin{enumerate}
\item $S_{\lambda;\mu}$
has a basis $\set{m_\sT}{\sT\in\STab_{\lambda;\mu}}$.
In particular,
\[S_{\lambda;\lambda}\simeq\begin{cases*}
\kk & if $\lambda$ is a partition,\\
0 & otherwise.
\end{cases*}\]
\item The product
\[\circ_{\nu}\colon S_{\nu;\mu}\otimes S^*_{\nu;\lambda}
\to M^{(\nu)}_{\lambda;\mu}\]
is injective.
\item $H_n$ and $\rS_{r,n}$ are cellular algebras.
\qedhere
\end{enumerate}
\end{corollary*}

Now for the $q$-Schur algebra $\rS_{r,n}=\bigoplus_{\lambda,\mu}M_{\lambda;\mu}$
its simple modules are easily classified.
Let $\nu$ be a partition of $n$.
If $\nu$ is of at most length $r$ then
the trace ideal of the Morita context
$(\bigoplus_\lambda S_{\nu;\lambda},\bigoplus_\lambda S^*_{\nu;\lambda})$
in $S_{\nu;\nu}\simeq\kk$ is clearly $\kk$.
Otherwise the Morita context is zero since $\lambda,\mu\not\le\nu$
for all such $\lambda,\mu$.
Hence we obtain the following classification.
\begin{theorem}
When $q\in\kk^\times$, there is a one-to-one correspondence
\[\Irr(\rS_{r,n})\onetoone\{\nu=(\nu_1,\dots,\nu_r)\colon\textup{partition of $n$}\}
\times\Irr(\kk)\]
induced by the Morita context functors.
Here for a pair of $\nu$ and $V\in\Irr(\kk)$, the corresponding simple module
is given by
\[\Image\biggl(\bigoplus_\lambda S_{\nu;\lambda}\otimes V\to
\Hom_\kk\Bigl(\bigoplus_\lambda S^*_{\nu;\lambda},V\Bigr)\biggr).\]
\end{theorem}

\subsection{Identification of the ideals}
Recall the assumption $q\in\kk^\times$.
We then proceed to the classification of simple modules
of the Iwahori--Hecke algebra $H_n\simeq M_{(1^n);(1^n)}$.
For each partition $\lambda$,
let $J_\lambda=S^*_\lambda\cdot S_\lambda$ be the trace ideal
of the Morita context $(S_\lambda,S^*_\lambda)$
in $S_{\lambda;\lambda}\simeq\kk$; here note that
the product $\circ_{(1^n)}$ is just the ordinary multiplication.
Since $S^*_\lambda$ is generated by $m_\lambda$, we have $J_\lambda=m_\lambda\cdot S_\lambda$.
In order to classify simple modules, we have to determine it.

\begin{lemma}\label{lem:ideal_identification}
Let $\lambda=(\lambda_1,\lambda_2,\dots,\lambda_r)$ be a partition.
For such $\lambda$, let
\[f_\lambda\coloneqq[\lambda_1-\lambda_2]![\lambda_2-\lambda_3]!\dots[\lambda_r]!.\]
Then we have inclusions
$\kk f_\lambda^r\subset J_\lambda\subset \kk f_\lambda$.
In particular, $\Irr^{J_\lambda}(\kk)=\Irr^{\kk f_\lambda}(\kk)$.
\end{lemma}
\begin{proof}
First we prove $J_\lambda\subset\kk f_\lambda$.
So it suffices to prove that for an arbitrary $\sT\in\Tab_\lambda$
we have $m_\lambda\cdot m_\sT\in\kk f_\lambda m_\lambda$ as an element of $S_{\lambda;\lambda}$.
Note that taking a refinement $\mu\coloneqq(\lambda_1,1^{n-\lambda_1})$ of $\lambda$
we can decompose $m_\lambda\in M^*_\lambda$ as
$m_\lambda\circ_\mu m_\mu$.
So let $\sS\coloneqq\sT|_\mu$. Explicitly,
$\sS$ is a row-semistandard tableau of shape $\lambda$
of weight $\mu$ defined by
\[\sS(i,j)\coloneqq\begin{cases*}
1 & if $1\le\sT(i,j)\le\lambda_1$,\\
\sT(i,j)-\lambda_1+1 & otherwise.
\end{cases*}\]
Let $\nu\coloneqq\sS[1]$ be the composition of $\lambda_1$
where $\nu_i$ is the number of entries $1,2,\dots,\lambda_1$ in the $i$-th row of $\sT$.
Then by Lemma~\ref{lem:multiply_refinement_tableau}~(2) we obtain that
\[m_\mu\cdot m_\sT=q^\ell[\nu_1]![\nu_2]!\dotsm[\nu_r]!\,m_\sS\]
for some $\ell\in\NN$.
In particular, the coefficient can be divided by $[\nu_1]!$.
Let $\lambda\setminus\nu$ be the composition of $n-\lambda_1$
defined by $(\lambda\setminus\nu)_i\coloneqq\lambda_i-\nu_i$.
Since $m_\sS$ factors through $M_{(\lambda_1,\lambda\setminus\nu)}$ as before,
if $\nu_1<\lambda_1-\lambda_2$ then $\lambda\not\ge(\lambda_1,\lambda\setminus\nu)$,
which implies $m_\sS\equiv0$ in $S_{\lambda;\mu}$.
Thus the statement trivially holds in this case.
Otherwise $[\nu_1]!$ can be divided by $[\lambda_1-\lambda_2]!$.
By induction, for $\lambda'=(\lambda_2,\dots,\lambda_r)$
we may assume that $m_{\lambda'}\cdot S_{\lambda'}\subset\kk f_{\lambda'} m_{\lambda'}$.
Note that $S_{\lambda;\mu}=\unit_{\lambda_1}*S_{\lambda'}$. Therefore
\[m_\lambda\cdot m_\sT\in[\lambda_1-\lambda_2]!\,m_\lambda\circ_\mu S_{\lambda;\mu}
\subset[\lambda_1-\lambda_2]!(\unit_{\lambda_1}*\kk f_{\lambda'}m_{\lambda'})=\kk f_\lambda m_\lambda.\]

Next we prove the other inclusion $\kk f_\lambda^r\subset J_\lambda$.
Let $\sR\in\Tab_{\lambda;\lambda}$ be the row-semistandard tableau
determined by $\#_{ij}(\sR)=\lambda_{i+j-1}-\lambda_{i+j}$.
For example, when $\lambda=(6,4,1)$,
\[\sR=\young(112223,1112,1).\]
Then by taking its underlying row-standard tableau $\sR_\downarrow\in\Tab_\lambda$,
by Lemma~\ref{lem:multiply_refinement_tableau}~(2) again we obtain
\[m_\lambda\cdot m_{\sR_\downarrow}=[\lambda_1-\lambda_2]![\lambda_2-\lambda_3]!^2\dotsm
[\lambda_r]!^rm_\sR.\]
On the other hand, by using Corollary~\ref{cor:scalar_corollary} repeatedly,
we also obtain that $m_\sR\in\kk^\times m_\lambda$ in $S_{\lambda;\lambda}$.
For example,
\[\young(112223,1112,1)\equiv-\young(112223,1111,2)
\equiv-q^6\young(111111,2223,2)\equiv q^6\young(111111,2222,3).\]
This implies $J_\lambda m_\lambda\supset\kk m_\lambda\cdot m_{\sR_\downarrow}\supset\kk f_\lambda^r m_\lambda$
as desired.
\end{proof}

This completes the classification we noted in the introduction.
\begin{theorem*}
When $q\in\kk^\times$, there is a one-to-one correspondence
\[\Irr(H_n)\onetoone\bigsqcup_{\lambda\colon\textup{partition of $n$}}\Irr^{\kk f_\lambda}(\kk)\]
induced by the Morita context functors.
For a partition $\lambda$ and $V\in\Irr^{\kk f_\lambda}(\kk)$,
 the corresponding simple module is
\[\Image(S_\lambda\otimes V\to\Hom_\kk(S^*_\lambda,V)).
\qedhere\]
\end{theorem*}

Finally let us consider the case that $\kk$ is a field.
Let $e\in\NN\cup\{\infty\}$ be the \term{$q$-characteristic} of $\kk$,
namely $e\coloneqq\min\set{k}{[k]=0}$ (the case $e=\infty$
is usually written as $e=0$, but we use this definition for simplicity). 
A partition $\lambda$ is called $e$-restricted
if $\lambda_i-\lambda_{i+1}<e$ holds for every $i$.
Then clearly we have that $\lambda$ is $e$-restricted if and only if
$f_\lambda=0$.
Thus as a corollary of the theorem we obtain the well-known classification.
\begin{corollary}
If $\kk$ is a field whose $q$-characteristic is $e$
(we still assume that $q\in\kk^\times$),
there is a one-to-one correspondence
\[\Irr(H_n)\onetoone\{\textup{$e$-restricted partition of $n$}\}.\]
\end{corollary}
The right-hand side set is actually the crystal $B(\Lambda_0)$ of type $\mathsf{A}^{(1)}_{e-1}$
under the description of Misra and Miwa~\cite{MisraMiwa90}.

\section{Cellular structure on the Hecke--Clifford superalgebra, I}\label{sec:cellular_hecke_clifford}

We are now ready to introduce the main topic of this paper,
the Hecke--Clifford superalgebra.
In this section we introduce analogues of
the Murphy basis, the $q$-Schur algebra and the Specht modules
for this superalgebra and develop the cellular
representation theory parallel to the Iwahori--Hecke algebra.

\subsection{The Clifford superalgebra}
First we define the most basic superalgebra,
the Clifford superalgebra.

\begin{definition}
Let $n\in\NN$ and take $a_1,a_2,\dots,a_n\in\kk$.
The \term{Clifford superalgebra} (or the \term{Clifford--Grassman superalgebra})
$C_n(a_1,a_2,\dots,a_n)$
is generated by the odd elements $c_1,c_2,\dots,c_n$ with relations
\[c_i^2 = a_i,\qquad
c_i c_j = -c_j c_i \quad\text{for } i\neq j.\]
\end{definition}
We have a canonical isomorphism of superalgebras
\[C_n(a_1,a_2,\dots,a_n)\simeq C_1(a_1)\otimes C_1(a_2)\otimes\dots\otimes C_1(a_n)\]
(note that by the help of Koszul sign $c_i$ and $c_j$ for $i\neq j$ (anti-)commutes).
Clearly $C_1(a)=\kk\oplus\kk c_1$ so
$\{c_1^{p_1}c_2^{p_2}\dotsm c_n^{p_n}\mid p_k\in\{0,1\}\}$
is a basis of $C_n(a_1,a_2,\dots,a_n)$.

\begin{remark}
More generally,
for a free $\kk$-module $V$ equipped with a quadratic form
$Q\colon V\to\kk$,
we have the corresponding Clifford superalgebra $C_Q$
generated by $V$ with the relation $v^2=Q(v)$.
When $\kk$ is a field whose characteristic is different from $2$,
we can always take an orthogonal basis with respect to $Q$, so that
$C_Q$ is isomorphic to the above form.
\end{remark}

The classification of simple modules of $C_n(a_1,a_2,\dots,a_n)$ is well-known
for special cases (see \cite[\S12]{Kleshchev05}).
We here state a more general result.

\begin{proposition}
Suppose $\kk$ is a field.
Then $C_n(a_1,a_2,\dots,a_n)$ has a unique maximal 2-sided ideal.
In particular, it has a unique simple module
up to isomorphism and parity change $\Pi$.
\end{proposition}
\begin{proof}
First we prove the case that $\kk$ is algebraically closed.
By replacing $c_i$ with $c_i/\!\sqrt{a_i}$ for $a_i\neq0$
and permuting the generators, we may assume that it is
in the form $C_n(1,\dots,1,0,\dots,0)$.
If the characteristic of $\kk$ is 2,
$C_2(1,1)$ is isomorphic to $C_2(1,0)$
since $c_1+c_2$ (anti-)commutes with $c_1$ and its square is zero.
Otherwise $C_2(1,1)$ is isomorphic to the matrix algebra
$\mathrm{Mat}_{1|1}(\kk)\coloneqq\End_\kk(\kk\oplus\Pi\kk)$
via the isomorphism using the Pauli matrices below:
\begin{align*}
1&\mapsto \begin{pmatrix}1&0\\0&1\end{pmatrix},&
c_2&\mapsto \begin{pmatrix}0&-\sqrt{-1}\\\sqrt{-1}&0\end{pmatrix},\\
c_1&\mapsto \begin{pmatrix}0&1\\1&0\end{pmatrix},&
c_1c_2&\mapsto \begin{pmatrix}\sqrt{-1}&0\\0&-\sqrt{-1}\end{pmatrix}.
\end{align*}
So in both cases, $C_n(a_1,a_2,\dots,a_n)$ is isomorphic to
$A\otimes\mathrm{Mat}_{p|p}(\kk)\otimes C_q(0,\dots,0)$
for some $p,q\in\NN$ and $A=\kk$ or $C_1(1)$.
Then central idempotent elements $c_1,\dots,c_q\in C_q(0,\dots,0)$
are contained in its Jacobson radical,
and the quotient superalgebra $A\otimes\mathrm{Mat}_{p|p}(\kk)$
with respect to these elements is Morita equivalent to $A$,
which is clearly simple (note that $(1\pm c_1)C_1(1)$ is not
considered as an ideal because it is not homogeneous).

Now let $\kk$ be an arbitrary field and
$\bar\kk$ its algebraic closure.
Let us write $C_n=C_n(a_1,a_2,\dots,a_n)$ for short.
Take a proper 2-sided ideal $I\subsetneq C_n$.
Then $I\otimes\bar\kk\subsetneq C_n\otimes\bar\kk$
is also a proper 2-sided ideal so contained in the
Jacobson radical of $C_n\otimes\bar\kk$ by the previous case.
Since the Jacobson radical is nilpotent, so is $I$.
Hence $I$ is contained in the Jacobson radical of $C_n$.
\end{proof}

\subsection{The Hecke--Clifford superalgebra}
Henceforth we fix elements $a,q\in\kk$.
Let us write $C_n=C_n(a)\coloneqq C_n(a,\dots,a)$ for short.

\begin{definition}
The \term{Hecke-Clifford superalgebra} $H^c_n=H^c_n(a;q)$ is
generated by $C_n(a)$ and $H_n(q)$ with relations
\[T_i c_j = c_j T_i\quad\text{for } j\neq i,i+1,\qquad
T_i c_i = c_{i+1} T_i,\qquad
(T_i-q+1) c_{i+1} = c_i (T_i-q+1).\]
Note that if $q\in\kk^\times$, the second relation
implies the third.
\end{definition}

Here in order to make it compatible with the notions in the previous part
we slightly modified the original definition by Olshanski~\cite{Olshanski92}
so that our $q$ is Olshanski's $q^2$.
When $q=1$, $H^c_n$ is isomorphic to the wreath product of the Clifford superalgebra
\[W_n(a)\coloneqq C_1(a)\wr\fS_n=C_n(a)\rtimes\fS_n\]
which is called the \term{Sergeev superalgebra}.
In this case there is a natural anti-homomorphism
$*\colon W_n(a)^\op\to W_n(-a)$ between superalgebras
defined by $s_i^*\coloneqq s_i$ and $c_i^*\coloneqq c_i$
(note that due to the Koszul sign we have $(c_i^\op)^2=-(c_i^2)^\op=-a$),
but unfortunately this involution does not have its $q$-analogue.

The next basis theorem is well-known, but we make its proof by ourself
since we modified the definition.

\begin{proposition}
The multiplication maps $C_n\otimes H_n\to H^c_n$
and $H_n\otimes C_n\to H^c_n$
are isomorphisms of supermodules.
\end{proposition}
\begin{proof}
We prove the first isomorphism.
By the defining relations this map is surjective.
In order to show that it is also injective,
we construct an action of $H^c_n$ on $C_n\otimes H_n$ by
\begin{align*}
T_i(x\otimes y)&\coloneqq s_i(x)\otimes T_i y + (q-1)t_i(x)\otimes y\\
c_i(x\otimes y)&\coloneqq c_i x\otimes y
\end{align*}
for $x\in C_n$ and $y\in H_n$.
Here $s_i$ is the automorphism of superalgebra $C_n$
which exchanges $c_i$ and $c_{i+1}$,
and $t_i$ is the $\kk$-linear map $C_n\to C_n$ defined by
\begin{align*}
t_i(1) &\coloneqq 0,&
t_i(c_{i+1}) &\coloneqq -c_i+c_{i+1},\\
t_i(c_i) &\coloneqq 0,&
t_i(c_i c_{i+1}) &\coloneqq a+c_i c_{i+1}
\end{align*}
and $t_i(z c_j)=t_i(z)c_j$ for $j\neq i,i+1$.
It is a routine work to verify that the action is well-defined.
This action satisfies $xy\cdot(1\otimes 1)=x\otimes y$
for $x\in C_n$ and $y\in H_n$ so
it defines the inverse map $H^c_n\to C_n\otimes H_n$.

Now we have $H^c_n\simeq C_n\otimes H_n$ so that
$H^c_n$ is a free supermodule over $\kk$ of rank $2^n n!$
with a basis $\{c_1^{p_1}\dotsm c_n^{p_n}T_w\}$.
By the commutation relation $\{T_wc_1^{p_1}\dotsm c_n^{p_n}\}$
also forms a basis of $H^c_n$.
This implies the second isomorphism.
\end{proof}

In particular, $C_n$ and $H_n$ can be identified with subsuperalgebras of $H^c_n$.
For each left $H_n$-module $V$,
$C_n\otimes V\simeq H^c_n\otimes_{H_n}V$ is naturally a left $H^c_n$-module.

\begin{remark}\label{rem:super_reduction}
By the commutation relation, for $n\ge2$, $I_n\coloneqq\sum_{1\le i<j\le n}(c_i-c_j)H^c_n$
is a 2-sided ideal of $H^c_n$ whose quotient superalgebra is
\[H^c_n/I_n\simeq C_1\otimes H_n\otimes(\kk/2a\kk).\]
Now suppose that $2a=0$.
Since $(c_i-c_j)^2=2a=0$ and $(c_i-c_j)(c_i-c_k)=-(c_j-c_k)(c_i-c_j)$
for mutual different $i$, $j$ and $k$,
the ideal $I_n$ is nilpotent. Thus
it is contained in the Jacobson ideal of $H^c_n$, so that
\begin{multline*}
\Irr(H^c_n)=\Irr(H^c_n/I_n)
\simeq\Irr(C_1\otimes H_n)
=\set{V,\Pi V}{V\in\Irr(H_n),aV=0}\\\sqcup\set{V\oplus c_1V}{V\in\Irr(H_n),aV=V}.
\end{multline*}
Hence the classification of simple module of $H^c_n$
is reduced to that of $H_n$.
\end{remark}

The next computation is a key of our theory.
Recall that $m_n=\sum_{w\in\fS_n}T_w\in H_n$.
\begin{lemma}\label{lem:gamma}
Let $\gamma_n^L\coloneqq c_1+q c_2+\dots+q^{n-1}c_n$
and $\gamma_n^R\coloneqq q^{n-1}c_1+q^{n-2} c_2+\dots+c_n$.
Then for $1\le i_1<i_2<\dots<i_r\le n$,
\[m_n c_{i_1}c_{i_2}\dotsm c_{i_r}m_n=\begin{cases*}
\bigl(\frac{a(q-1)}{[2]}\bigr)^s[n]!\,m_n & if $r=2s$,\\
\bigl(\frac{a(q-1)}{[2]}\bigr)^s[n-1]!\,\gamma_n^L m_n & if $r=2s+1$
\end{cases*}\]
(note that even $q$-integers $[2],[4],[6],\dots$ can be divided by $[2]$, so that the coefficients are
polynomials in $\ZZ[a,q]$).
Moreover we have $\gamma_n^L m_n=m_n\gamma_n^R$. 
\end{lemma}
\begin{proof}
Since $H^c_n$ is free over $\kk$,
it suffices to prove for the field of rational functions
$\kk=\QQ(a,q)$ in variables $a$ and $q$,
which contains the universal ring $\ZZ[a,q]$.
If $i_{j-1}<i_j-1$ holds for some $j$, we have
\begin{align*}
m_n\dotsm c_{i_j}\dotsm m_n
&=q^{-1}m_n\dotsm c_{i_j}\dotsm T_{i_j-1}m_n\\
&=q^{-1}m_n T_{i_j-1}\dotsm c_{i_j-1}\dotsm m_n\\
&=m_n\dotsm c_{i_j-1}\dotsm m_n.
\end{align*}
Hence we may assume $i_j=j$.
Then for $r=0$ or $1$, we have $m_n^2=[n]!\,m_n$ and
\[m_n c_1 m_n=(c_1+c_2T_1+\dots+c_nT_{n-1}\dotsm T_2T_1) m'_{n-1} m_n
=[n-1]!\,\gamma_n^L m_n\]
where $m'_{n-1}=m_{(1,n-1)}$. Moreover we have
\begin{align*}
qm_nc_1c_2c_3\dotsm m_n&=m_nc_1c_2c_3\dotsm T_1m_n\\
&=m_nc_1T_1c_1c_3\dotsm m_n\\
&=m_n(T_1c_2+(q-1)(c_1-c_2))c_1c_3\dotsm m_n\\
&=a(q-1)m_nc_3\dotsm m_n-m_nc_1c_2c_3\dotsm m_n
\end{align*}
so that
\[m_nc_1c_2c_3\dotsm m_n=\frac{a(q-1)}{[2]}m_nc_3\dotsm m_n.\]
Hence inductively we obtain the equation.
Similarly as above we have
\[m_nc_nm_n=m_nm_{n-1}(c_n+T_{n-1}c_{n-1}+\dots+T_{n-1}\dotsm T_2T_1c_1)=[n-1]!\,m_n\gamma_n^R\]
which implies $\gamma_n^L m_n=m_n\gamma_n^R$.
\end{proof}

\subsection{Parabolic supermodules}
Analogously to the Iwahori--Hecke algebra,
for each composition $\lambda=(\lambda_1,\lambda_2,\dots,\lambda_r)$ 
we introduce the \term{parabolic subalgebra}
\[H^c_\lambda=\bigoplus_{w\in\fS_\lambda}C_nT_w
=\bigoplus_{w\in\fS_\lambda}T_wC_n\simeq
H^c_{\lambda_1}\otimes H^c_{\lambda_2}\otimes\dots\otimes H^c_{\lambda_r}.\]
Then $H^c_n$ is again a free right $H^c_\lambda$-module with a basis $\set{T_w}{w\in\fD_\lambda}$.
For $m_\lambda=\sum_{w\in\fS_\lambda} T_w$,
$C_n m_\lambda$ is a left (but not right) ideal of $H^c_\lambda$
and the \term{parabolic module} $M^c_\lambda\coloneqq H^c_n m_\lambda\simeq
H^c_n\otimes_{H^c_\lambda}C_n m_\lambda$
is defined as its induced module.
Then by the basis theorem $H^c_n\simeq C_n\otimes H_n$
we have $M^c_\lambda\simeq C_n\otimes M_\lambda$, and
\[M^c_\lambda=\set{x\in H^c_n}{x T_w=q^{\ell(w)}x\text{ for all }w\in\fS_\lambda}.\]
In particular we define the \term{trivial module}
$\unit^c_n\coloneqq M^c_{(n)}\simeq C_n$.
Similarly right modules $M^{c*}_\lambda\coloneqq m_\lambda H^c_n\simeq M^*_\lambda\otimes C_n$ and $\unit^{c*}_n\coloneqq M^{c*}_{(n)}$ are defined. Then we have
\[\Hom_{H^c_n}(M^c_\mu,M^c_\lambda)\simeq M^c_\lambda\cap M^{c*}_\mu\]
equipped with the reversed product
\begin{align*}
\circ_\mu\colon(M^c_\mu\cap M^{c*}_\nu)\otimes(M^c_\lambda\cap M^{c*}_\mu)
&\to M^c_\lambda\cap M^{c*}_\nu,\\
x m_\mu\otimes m_\mu y &\mapsto xm_\mu y.
\end{align*}

With Lemma~\ref{lem:gamma} in mind,
for each composition $\lambda$ we define elements
\begin{align*}
\gamma^L_{\lambda;1}&\coloneqq c_{[1,2,\dots,\lambda_1]},&
\gamma^R_{\lambda;1}&\coloneqq c_{[\lambda_1,\dots,2,1]},\\
\gamma^L_{\lambda;2}&\coloneqq c_{[\lambda_1+1,\lambda_1+2,\dots,\lambda_1+\lambda_2]},&
\gamma^R_{\lambda;2}&\coloneqq c_{[\lambda_1+\lambda_2,\dots,\lambda_1+2,\lambda_1+1]},\\
&\dots&&\dots
\end{align*}
where $c_{[i_1,i_2,\dots,i_r]}\coloneqq c_{i_1}+qc_{i_2}+\dots+q^{r-1}c_{i_r}$.
Then we have $\gamma^L_{\lambda;i}m_\lambda=m_\lambda\gamma^R_{\lambda;i}$.
So let us define the endomorphism $\gamma_{\lambda;i}$ acts on
$M^c_\lambda$ and $M^{c*}_\lambda$ as
\[xm_\lambda\cdot\gamma_{\lambda;i}\coloneqq
xm_\lambda\gamma^R_{\lambda;i}=x\gamma^L_{\lambda;i}m_\lambda,\qquad
\gamma_{\lambda;i}\cdot m_\lambda y\coloneqq
\gamma^L_{\lambda;i}m_\lambda y=m_\lambda\gamma^R_{\lambda;i}y.\]
Note that these endomorphisms anti-commute and we have
\[(\gamma^L_{\lambda;i})^2=(\gamma^R_{\lambda;i})^2=a\qqn{\lambda_i}\]
where $\qqn{k}$ is a $q^2$-integer $1+q^2+\dots+q^{2(k-1)}$.
We can abstractly define a superalgebra consisting of these actions as follows.
\begin{definition}
Let $\lambda$ be a composition.
Let $\Gamma_\lambda$ be a superalgebra generated by odd elements
$\gamma_{\lambda;1},\gamma_{\lambda;2},\dotsc$ with relations
\[(\gamma_{\lambda;i})^2 = a\qqn{\lambda_i},\qquad
\gamma_{\lambda;i} \gamma_{\lambda;j} = -\gamma_{\lambda;j} \gamma_{\lambda;i} \quad\text{for } i\neq j,\qquad
\gamma_{\lambda;i}=0\quad\text{if }\lambda_i=0.\]
\end{definition}
By definition, it is just isomorphic to the Clifford superalgebra
\[\Gamma_\lambda\simeq
C_r(a\qqn{\lambda_{i_1}},a\qqn{\lambda_{i_2}},\dots,a\qqn{\lambda_{i_r}})\]
where $\{i_1,i_2,\dots,i_r\}$ are indices such that $\lambda_i\neq0$.
By the action above $M^c_\lambda$ (resp.\ $M^{c*}_\lambda$)
is now an $(H^c_n,\Gamma_\lambda)$-bimodule (resp.\ a $(\Gamma_\lambda,H^c_n)$-bimodule).
Since the set
\[\set{(\gamma^L_{\lambda;{i_1}})^{p_1}(\gamma^L_{\lambda;{i_2}})^{p_2}
\dotsm(\gamma^L_{\lambda;{i_r}})^{p_r}}{p_k\in\{0,1\}}\]
is linearly independent in $C_n$,
the map $\Gamma_\lambda\to\Gamma_\lambda m_\lambda\subset
M^c_\lambda\cap M^{c*}_\lambda$ is an inclusion of superalgebra.

\subsection{Circled tableaux}

In order to denote elements of the parabolic module $M^c_\lambda$ graphically
we introduce the notion of \term{circled tableau}~\cite{Sagan87}.
Here a circled tableau of shape $\lambda$ is a map $Y(\lambda)\to\{1,2,\dots,\}\sqcup\{\circleA,\circleB,\dots\}$.
From a circled tableau $\sT$ we obtain its underlying ordinary tableau $\sT^\times$ by removing circles from numbers.
The weight of a circled tableau is defined as that of underlying tableau.
We say that a circled tableau is \term{row-standard}
if its underlying tableau is row-standard.
Let $\Tab^c_\lambda$ be the set of row-standard circled tableau of shape $\lambda$.
For $\sT\in\Tab^c_\lambda$ we define the corresponding element
$m_\sT\coloneqq T_w c_{i_1}\dotsm c_{i_r} m_\lambda$
where $w=d(\sT^\times)$ and $i_1,\dots,i_r$ are indices of positions
of circled entries in $\sT$ according to
the top-to-bottom reading order. For example,
\[\text{for }\sT=\young(\circleA24\circleE,3\circleG8,\circleF),\quad
m_\sT=T_3T_4T_6T_7c_1c_4c_6c_8m_{(4,3,1)}.\]
For such $\sT$, we define its length as
$\ell(\sT)\coloneqq\ell(d(\sT^\times))$.
If we focus only on leading terms with respect to this length, we have
\[T_w c_{i_1}\dotsm c_{i_r}=c_{w(i_1)}\dotsm c_{w(i_r)} T_w+(\text{lower terms})\]
so by $M^c_\lambda\simeq C_n\otimes M_\lambda$
the set $\set{m_\sT}{\sT\in\Tab^c_\lambda}$ forms a basis of $M^c_\lambda$.
The action of $T_i$ is described as
\[T_i\cdot m_\sT=\begin{cases*}
m_{s_i\sT} & if $r(i)<r(i+1)$,\\
q m_\sT + (q-1)m_{s_i\sT} & if $r(i)>r(i+1)$
\end{cases*}\]
where $r(i)$ is the index of the row which contains $i$ or $\circleI$
similarly as before, and
$s_i\sT$ is the circled tableau whose underlying tableau is
$(s_i\sT)^\times=s_i(\sT^\times)$ and
which has circles at the same boxes as $\sT$.
If $r(i)=r(i+1)$, putting $j=i+1$ it acts by
\begin{align*}
T_i\cdot\young(ij)&=q\young(ij),&
T_i\cdot\young(i\circleJ)&=\young(\circleI j)+(q-1)\young(i\circleJ),\\
T_i\cdot\young(\circleI j)&=q\young(i\circleJ),&
T_i\cdot\young(\circleI\circleJ)&=a(q-1)\young(ij)-\young(\circleI\circleJ)
\end{align*}
In contrast the action of $c_i$ is hard to describe
due to the commutation relation of $C_n$ and $H_n$, but on leading terms we have
\[c_i\cdot\young(i)=\pm\young(\circleI)+\dotsb,\qquad
c_i\cdot\young(\circleI)=\pm a\young(i)+\dotsb\]
as desired. Here the signs above are taken to be $+$ if
the tableau has even number of circles before this box with respect to the reading order,
and otherwise $-$.
The right action of $\Gamma_\lambda$ is easy: for example,
\[\young(\circleA24\circleE,3\circleG8,\circleF)\cdot\gamma_{(4,3,1);2}=
\young(\circleA24\circleE,\circleC\circleG8,\circleF)
-aq\young(\circleA24\circleE,378,\circleF)
-q^2\young(\circleA24\circleE,3\circleG\circleH,\circleF).\]
Beware the signs due to the exchange of $c_i$ and $c_j$.
\begin{remark}
Usually the shifted form
\[\young(\circleA24\circleE,:3\circleG8,::\circleF)\]
is used in literatures for circled tableaux. We continue to use the non-shifted form
since it seems to be troublesome to change the notations from the previous section.
\end{remark}

Furthermore we introduce the set of row-semistandard circled tableau
$\Tab^c_{\lambda;\mu}$ to denote elements of $M^c_\lambda\cap M^{c*}_\mu$.
We call a circled tableau of shape $\lambda$ and of weight $\mu$
is \term{row-semistandard} if its underlying tableau is row-semistandard
and it does not contain patterns $\young(\circleI i)$ or $\young(\circleI\circleI)$.
In other words, circled numbers must be placed at the rightmost
of a bar $\young(ii\scdots i)$ in a row.
It is also equivalent to say that each row is weakly increasing with respect to the order
\[1<\circleA<2<\circleB<3<\circleC<\dotsb\]
and a circled number can not be adjacent to itself.
For such $\sS\in\Tab^c_{\lambda;\mu}$, we define an element $m_\sS\in M^c_\lambda\cap M^{c*}_\mu$
as follows:
first we make a formal linear combination
of tableaux from $\sS$ by distributing
\[\young(ii\scdots\circleI)\mapsto\young(\circleI i\scdots i)+q\young(i\circleI\scdots i)
\dots+q^{r-1}\young(ii\scdots\circleI)\]
for each circled bar of length $r$,
then by replacing each term $q^l\sR$ with the sum of $q^l m_\sT\in M^c_\lambda$
for all $\sT$ such that $\sT^\times\in\Tab_{\sS^\times}$
and its positions of circles are same as that of $\sR$.
For example, for
\[\sS=\young(1\circleA2\circleC,15\circleE,\circleD)\]
we have
\begin{multline*}
m_\sS=\young(\circleA24\circleE,3\circleG8,\circleF)
+q\young(1\circleB4\circleE,3\circleG8,\circleF)
+q\young(\circleA24\circleE,37\circleH,\circleF)
+q^2\young(1\circleB4\circleE,37\circleH,\circleF)\\
+\young(\circleA34\circleE,2\circleG8,\circleF)
+q\young(1\circleC4\circleE,2\circleG8,\circleF)
+q\young(\circleA34\circleE,27\circleH,\circleF)
+q^2\young(1\circleC4\circleE,27\circleH,\circleF)\\
+\young(\circleB34\circleE,1\circleG8,\circleF)
+q\young(2\circleC4\circleE,1\circleG8,\circleF)
+q\young(\circleB34\circleE,17\circleH,\circleF)
+q^2\young(2\circleC4\circleE,17\circleH,\circleF).
\end{multline*}

\begin{proposition}\label{prop:super_basis_lemma}
The set $\set{m_\sS}{\sS\in\Tab^c_{\lambda;\mu}}$ is linearly independent
in $M^c_\lambda\cap M^{c*}_\mu$.
Moreover if $[2]\in\kk$ is not a zero-divisor,
it is also a basis of $M^c_\lambda\cap M^{c*}_\mu$.
\end{proposition}
\begin{proof}
For each $\sS\in\Tab^c_{\lambda;\mu}$,
take $\sT\in\Tab^c_\lambda$ so that $\sT^\times=(\sS^\times)_\downarrow$
and $\sT$ has a circle at each box whose position is the leftmost of
circled bars $\young(ii\scdots\circleI)$ in $\sS$.
Then the coefficient of $m_\sS$ at the basis element $m_\sT$ is $1$.
Since this map $\sS\mapsto\sT$ is injective,
the set $\set{m_\sS}{\sS\in\Tab^c_{\lambda;\mu}}$ is linearly independent
in $M^c_\lambda$.

On the other hand, let us take $x\in M^c_\lambda$
and write $x=\sum_{\sT\in\Tab^c_\lambda}c_\sT m_\sT$.
Suppose that $x\in M^{c*}_\mu$ and let $s_i\in\fS_\mu$.
Then by the above description of the action of $T_i$,
for $\sT\in\Tab^c_\lambda$ such that $r(i)\neq r(i+1)$,
$T_ix=qx$ implies $c_\sT=c_{s_i\sT}$.
On the other hand, suppose $r(i)=r(i+1)$.
Let $\sT_1$, $\sT_2$, $\sT_3$ and $\sT_4$ be circled tableaux
obtained by replacing $i$ and $j=i+1$ in $\sT$ by
$\young(ij)$, $\young(\circleI j)$, $\young(\circleI j)$ and $\young(\circleI\circleJ)$
respectively.
Then by $(1+T_i)x=[2]x$, we have
$[2](c_{\sT_2}-qc_{\sT_3})=[2]c_{\sT_4}=0$.
Using the assumption that $[2]$ is not a zero-divisor,
we obtain $c_{\sT_2}=qc_{\sT_3}$ and $c_{\sT_4}=0$.
Hence $x$ can be written as a linear combination of $m_\sS$.
It is also clear that this condition is sufficient for that $x\in M^{c*}_\mu$.
\end{proof}
Unfortunately, if the assumption is not satisfied then
this statement may fail.
For example, when $q=1$ and $2=0$ in $\kk$,
the element $c_1c_2m_2$ is incidentally contained in $\unit^c_2\cap\unit^{c*}_2$.
The set $M^c_\lambda\cap M^{c*}_\mu$ is not suitable for our use,
so instead we use a well-behaved set
\[M^c_{\lambda;\mu}\coloneqq\kk\set{m_\sS}{\sS\in\Tab^c_{\lambda;\mu}}
\subset M^c_\lambda\cap M^{c*}_\mu.\]
This free $\kk$-module is preserved by an extension of scalars.
Since the universal ring $\kk=\ZZ[a,q]$ satisfies the assumption,
it is closed under product
\[\circ_\mu\colon M^c_{\mu;\nu}\otimes M^c_{\lambda;\mu}\to M^c_{\lambda;\nu}.\]
By definition we can represent
$\gamma^L_{\lambda;i}m_\lambda=m_\lambda\gamma^R_{\lambda;i}\in M^c_\lambda\cap M^{c*}_\lambda$
by a circled tableau, so
$\Gamma_\lambda m_\lambda$ is contained in $M^c_{\lambda;\lambda}$. Hence
$\Gamma_\mu$ also acts on $M_{\lambda;\mu}$ from left
(resp. $M_{\mu;\nu}$ from right)
and the product $\circ_\mu$ above is $\Gamma_\mu$-bilinear,
so that we can define it as
\[\circ_\mu\colon M^c_{\mu;\nu}\otimes_{\Gamma_\mu} M^c_{\lambda;\mu}\to M^c_{\lambda;\nu}.\]
Let $\rS^c_{r,n}\coloneqq\bigoplus_{\lambda,\mu}M^c_{\lambda;\mu}$
where $\lambda=(\lambda_1,\dots,\lambda_r)$ and $\mu=(\mu_1,\dots,\mu_r)$
run over compositions of at most $r$ components as before.
We call it the \term{queer $q$-Schur superalgebra}.
When $q=1$ and $2\neq0$, it is equal to the Schur superalgebra of type~$\mathsf{Q}$
introduced in \cite{BrundanKleshchev02}.

We finish this subsection with a remark on involution.
For $\sS\in\Tab^c_{\lambda;\mu}$, we can similarly define
an element $m^*_\sS\in M^c_\mu\cap M^{c*}_\lambda$ in the dual manner
by multiplying elements on $m_\lambda$ from right.
When $q=1$, it is actually the dual element of $m_\sS\in W(-a)$
mapped via the anti-homomorphism $*\colon W_n(-a)^\op\to W_n(a)$.
For such $\sS$ we define its dual tableau $\sS^*\in\Tab^c_{\mu;\lambda}$
so that $(\sS^\times)^*=(\sS^*)^\times$ and $\sS$ has $\circleJ$ in its $i$-th row
if and only if $\sS^*$ has $\circleI$ in its $j$-th row.
Then by the commutation relation on Lemma~\ref{lem:gamma} $m^*_\sS$
has the leading term $m_{\sS^*}$ but they are not equal unless $q=1$.
The map $m_\sS\mapsto m_{\sS^*}$
does not either preserve the reversed product in general.

\subsection{Good circled tableaux}
Analogously to the non-super case,
we introduce a filtration into our subcategory of $\lMod{H^c_n}$.
According to this filtration we decompose
the set of simple modules of $H^c_n$ into small parts.

\begin{definition}
For each compositions $\lambda$, $\mu$, and $\nu$, let
\[M^{c\,\nu}_{\lambda;\mu}\coloneqq M^c_{\nu;\mu}\circ_\nu M^c_{\lambda;\nu}
\subset M^c_{\lambda;\mu}.\]
Then we define
\[M^{c\,\ge\nu}_{\lambda;\mu}\coloneqq\sum_{\pi\ge\nu}M^{c\,\pi}_{\lambda;\mu},\quad
M^{c\,>\nu}_{\lambda;\mu}\coloneqq\sum_{\pi>\nu}M^{c\,\pi}_{\lambda;\mu}\]
and finally
\[M^{c(\nu)}_{\lambda;\mu}\coloneqq M^c_{\lambda;\mu}\big/M^{c\,\ge\nu}_{\lambda;\mu}.\]
In particular we let $S^c_{\lambda;\mu}\coloneqq M^{c(\lambda)}_{\lambda;\mu}$
and $S^c_\lambda\coloneqq S^c_{\lambda;(1^n)}$.
\end{definition}

We say that a circled tableau $\sT\in\Tab^c_{\lambda;\mu}$ is \term{good}
if its underlying tableau $\sT^\times$ is good.

\begin{lemma}
$S^c_{\lambda;\mu}$
is spanned by $\set{m_\sT}{\sT\in\Tab^c_{\lambda;\mu}\textup{ which is good}}$.
\end{lemma}
\begin{proof}
Similarly to the proof of Lemma~\ref{lem:good_tableaux},
we prove it inductively by replacing each $m_\sT$
for ungood $\sT\in\Tab^c_{\lambda;\mu}$
with tableaux which have smaller lengths.
However in this case we can not perform this method at a time,
so we do for each number one by one.
Suppose that $\sT$ has ungood $\young(i)$ or $\young(\circleI)$,
which we choose so that $i$ is minimum.
In particular, $i$'s in its $i$-th row are at the leftmost of $\sT$ if exist.
If it has circled $\circleI$ in its $i$-th row,
multiplying $\gamma_{\lambda;i}$ from right we can represent $m_\sT$ by
tableaux without this circle; for example,
\[\young(1111\circleB3,223\circleC,\circleA\circleB3)\cdot\gamma_{\lambda;1}
=\young(111\circleA\circleB3,223\circleC,\circleA\circleB3)
-aq^4\young(111123,223\circleC,\circleA\circleB3)
-q^5\young(1111\circleB\circleC,223\circleC,\circleA\circleB3).\]
Hence we may assume that $i$ in the $i$-th row of $\sT$ is not circled.
We define tableaux $\sT_1,\sT_2$ from $\sT$
by moving up ungood $\young(i)$ and $\young(\circleI)$
as we did in the proof of Lemma~\ref{lem:good_tableaux},
and if such entry is circled we remove this circle from $\sT_2$ and
put on the same box at $\sT_1$.
For example, for
\[\sT=\young(1111\circleB3,223\circleC,\circleA\circleB3)\]
we have
\[\sT_1=\young(111111,2222,\circleA33)\quad\text{and}\quad
\sT_2=\young(11111\circleB3,223\circleC,\circleB3).\]
Taking leading terms we also have
a decomposition $0\equiv m_{\sT_2}\circ_\nu m_{\sT_1}=\pm m_{\sT}+(\text{lower terms})$.
\end{proof}

This leads us the following parallel results.

\begin{corollary*}\label{cor:hom_on_dominance_order_super}
\ 
\begin{enumerate}
\item $S^c_{\lambda;\lambda}$ is spanned by $m_\lambda$ over $\Gamma_\lambda$,
so that it is isomorphic to a quotient superalgebra of $\Gamma_\lambda$.
\item $S^c_{\lambda;\mu}=0$ unless $\lambda\ge\mu$.
\qedhere
\end{enumerate}
\end{corollary*}

\begin{theorem*}
$H^c_n$ and $\rS^c_{r,n}$
are standardly filtered algebras over $\{\Gamma_\lambda\}$ on compositions $\lambda$.
\end{theorem*}

Note that we have a natural map
\[M^{c(\nu)}_{\lambda;\mu}\to\Hom^{(\nu)}_{H^c_n}(M^c_\mu,M^c_\lambda)\]
where the right-hand side is the set of homomorphisms in the quotient category
na\"\i vely defined by using the whole category $\lMod{H^c_n}$,
but in general this map is not surjective nor injective when the assumption
in Proposition~\ref{prop:super_basis_lemma} is not satisfied.
Although we can also define a standard filter using the right-hand side,
this filter is ill-behaved with extension of scalars.
In contrast, we have
\[S^c_{\lambda;\mu}\simeq S^c_{\lambda;\mu}(\ZZ[a,q])\otimes_{\ZZ[a,q]}\kk\]
as desired since $M^c_{\lambda;\mu}$ has a free basis.
So $S^c_{\lambda;\mu}$ is certainly the right definition.

In these modules we have the local transformation lemma
by a similar proof as before.
\begin{lemma*}
Suppose we have an equation
$\sum_\sT c_\sT m_\sT\equiv0$ in $S^c_{\lambda;\mu}$
for some $c_\sT\in\kk$.
For each $\sT\in\Tab_{\lambda;\mu}$ let $\sT^+$ be the tableau obtained by
adding a new common row at the top (resp.\ the bottom) of $\sT$.
Then we have $\sum_\sT c_\sT m_{\sT^+}\equiv0$.
\end{lemma*}

\subsection{Shifted semistandard circled tableaux}
In this subsection, we consider the case $\kk=\QQ(a,q)$.
Recall that a composition $\lambda$ is called a \term{strict partition}
if it is strictly decreasing: $\lambda_1>\lambda_2>\dots>\lambda_r>0=\lambda_{r+1}=\lambda_{r+2}=\dotsb$.

\begin{definition}
A row-semistandard circled tableau $\sT\in\STab^c_{\lambda;\mu}$
is called \term{shifted semistandard} if
its shape $\lambda$ is a strict partition and
it does not contain any of the patterns
\[\young(:i,i),\young(:\circleI,i)\qquad\text{and}\qquad
\young(:j,i),\young(:j,\circleI),\young(:\circleJ,i),\young(:\circleJ,\circleI)
\quad\text{for }i<j\]
(in particular, 
its underlying tableau $\sT^\times$ is semistandard).
In other words, its entries are also weakly increasing
along with each diagonal line so that a non-circled number does not continue.
We denote by $\STab^c_{\lambda;\mu}$ the set of
shifted semistandard circled tableaux of shape $\lambda$ of weight $\mu$.
\end{definition}

A similar notion of \term{generalized shifted tableau} is introduced 
by Sagan~\cite{Sagan87}.
The only difference is that he use the order
\[\circleA<1<\circleB<2<\circleC<3<\dotsb\]
instead of ours.
The set of shifted semistandard circled tableaux is clearly is in bijection
with that of his generalized shifted tableaux by the following circle moving:
\[\young(:::::ii,::ii\circleI,i\circleI)\mapsto
\young(:::::\circleI i,::\circleI ii,ii),\qquad
\young(:::::i\circleI,::ii\circleI,i\circleI)\mapsto
\young(:::::\circleI i,::\circleI ii,\circleI i).\]

\begin{lemma}\label{lem:circle_move}
Let $m,k\in\NN$ such that $m\ge k$.
Let $\lambda\coloneqq(m,k)$ and $\mu\coloneqq(k,m)$.
Then in $S^c_{\lambda;\mu}$ we have
\[\young(11\scdots\circleA2\scdots2,22\scdots2)\equiv
\young(11\scdots12\scdots2,22\scdots\circleB).\]
\end{lemma}
\begin{proof}
By Lemma~\ref{lem:scalar_lemma} and the assumption $q\in\kk^\times$, we have
\begin{align*}
\gamma_{\mu;2}\cdot\young(11\scdots12\scdots2,22\scdots2)
&=(-1)^kq^{-\binom{k}{2}}\gamma_{\mu;2}\cdot\young(22\scdots\scdots\scdots\scdots2,11\scdots1)\\
&=(-1)^kq^{-\binom{k}{2}}\young(22\scdots\scdots\scdots\scdots\circleB,11\scdots1)\\
&=(-1)^kq^{-\binom{k}{2}}\young(22\scdots\scdots\scdots\scdots2,11\scdots1)\cdot\gamma_{\lambda;1}\\
&=\young(11\scdots12\scdots2,22\scdots2)\cdot\gamma_{\lambda;1}.
\end{align*}
If $m=k$, we have
\[\gamma_{\lambda;2}\cdot\young(11\scdots1,22\scdots2)
=\young(11\scdots1,22\scdots\circleB),\qquad
\young(11\scdots1,22\scdots2)\cdot\gamma_{\lambda;1}
=\young(11\scdots\circleA,22\scdots2).\]
Otherwise both-hand sides can be computed by Lemma~\ref{lem:gamma} as
\begin{align*}
\gamma_{\mu;2}\cdot\young(11\scdots12\scdots2,22\scdots2)
&=\young(11\scdots12\scdots2,22\scdots\circleB)+
q^k\young(11\scdots12\scdots\circleB,22\scdots2)
\shortintertext{and}
\young(11\scdots12\scdots2,22\scdots2)\cdot\gamma_{\lambda;1}
&=\young(11\scdots\circleA2\scdots2,22\scdots2)+
q^k\young(11\scdots12\scdots\circleB,22\scdots2)
\end{align*}
so these equations also imply the statement.
\end{proof}

\begin{lemma}\label{lem:strict_partition}
$S^c_{\lambda;\lambda}=0$ unless $\lambda$ is a strict partition.
\end{lemma}
\begin{proof}
If $\lambda$ is not a partition it holds by the same reason
as the non-super case.
Otherwise if $\lambda$ is not a strict partition, it contains $\lambda_i=\lambda_{i+1}>0$.
So it suffices to prove for $\lambda=(k,k)$.
By the lemma above we have
\[(\gamma_{\lambda;1}-\gamma_{\lambda;2})\cdot\young(11\scdots1,22\scdots2)
=\young(11\scdots\circleA,22\scdots2)-\young(11\scdots1,22\scdots\circleB)\equiv0.\]
On the other hand,
$(\gamma_{\lambda;1}-\gamma_{\lambda;2})^2=2a\qqn{k}\in\kk^\times$.
Hence we have $m_\lambda\equiv0$ in $S^c_{\lambda;\lambda}$.
\end{proof}

\begin{lemma}\label{lem:semistandard_super}
$S^c_{\lambda;\mu}$ is spanned by $\set{m_\sT}{\sT\in\STab^c_{\lambda;\mu}}$.
\end{lemma}
\begin{proof}
If $\lambda$ is not a strict partition the statement is clear by the previous lemma,
so we may assume so.
First we prove the statement for two special cases.

Case 1: $\lambda=(m,k)$ and $\mu=(k,k,m-k)$.
Then every good but non-shifted-semistandard circled tableaux have an underlying tableau
\[\young(11\scdots13\scdots3,22\scdots2)\]
and can be made from this tableau by multiplying $\gamma_{\mu;i}$.
By Lemma~\ref{lem:circle_move},
\[\young(11\scdots\circleA3\scdots3,22\scdots2)
\equiv\young(11\scdots13\scdots3,22\scdots\circleB)+(\text{lower terms}).\]
Hence the statement holds since
$(\gamma_{\mu;1}-\gamma_{\mu;2})^2=2a\qqn{k}$ is invertible again
and every good tableau which has smaller length
is shifted semistandard.

Case 2: $\lambda=(m,k)$ and $\mu=(k,l,m-l-l',l')$
where $l<k$ and $l'<m-k$, so
\[\young(11\scdots\scdots\scdots13\scdots4,2\scdots23\scdots3)
\qquad\text{and}\qquad
\young(11\scdots\scdots\scdots13\scdots4,2\scdots23\scdots\circleC)\]
are all the good but non-shifted-semistandard circled tableaux.
Similar to above,
\[\young(11\scdots\scdots\scdots\circleA3\scdots4,2\scdots23\scdots3)
\equiv\young(11\scdots\scdots\scdots13\scdots4,2\scdots\circleB3\scdots3)
+q^l\young(11\scdots\scdots\scdots13\scdots4,2\scdots23\scdots\circleC)+\dotsb.\]
Then by multiplying $\gamma_{\mu;1}$ and $\gamma_{\mu;2}$ from left respectively, we obtain
\[a\qqn{k}\young(11\scdots\scdots\scdots13\scdots4,2\scdots23\scdots3)
\equiv\young(11\scdots\scdots\scdots\circleA3\scdots4,2\scdots\circleB3\scdots3)
+q^l\young(11\scdots\scdots\scdots\circleA3\scdots4,2\scdots23\scdots\circleC)+\dotsb\]
and
\[-\young(11\scdots\scdots\scdots\circleA3\scdots4,2\scdots\circleB3\scdots3)
\equiv a\qqn{l}\young(11\scdots\scdots\scdots13\scdots4,2\scdots23\scdots3)
+q^l\young(11\scdots\scdots\scdots13\scdots4,2\scdots\circleB3\scdots\circleC)+\dotsb\]
so that
\begin{multline*}
a(\qqn{k}+\qqn{l})\young(11\scdots\scdots\scdots13\scdots4,2\scdots23\scdots3)
\\\equiv q^l\young(11\scdots\scdots\scdots13\scdots4,2\scdots\circleB3\scdots\circleC)
-q^l\young(11\scdots\scdots\scdots13\scdots4,2\scdots\circleB3\scdots\circleC)+\dotsb
\end{multline*}
with $a(\qqn{k}+\qqn{l})\in\kk^\times$.
Multiplying $\gamma_{\lambda;1}$ from right we can similarly decompose
the second tableau above.

Now we proceed to a general case.
Let $\sT\in\Tab^c_{\lambda;\mu}$ which is not shifted semistandard.
Then there is a prohibited pattern at boxes $(k,l+1)$ and $(k+1,l)$.
Choose $(k,l)$ so that $\sT$ has no such patterns at
the bottom right boxes of $(k,l)$ except for it.
We prove that we can replace the element $m_\sT$ by
a linear combination of $m_\sR$ where either $\sR$
has no such patterns in this region or $\sR$ satisfies $\ell(\sR^\uparrow)<\ell(\sT^\uparrow)$.
Then by induction it can be written as a linear combination of
shifted standard ones.

First consider the case $\sT(k+1,l)=i,\circleI$ and $\sT(k,l+1)=j,\circleJ$ with $i<j$.
Similar to the proof of Lemma~\ref{lem:semistandard},
we define
\[\nu\coloneqq(\lambda_1,\dots,\lambda_{k-1},l,l,\lambda_k-l,\lambda_{k+1}-l,
\lambda_{k+2},\lambda_{k+3},\dotsc)\]
and $\sT_1\in\Tab_{\lambda;\nu}$, $\sT_2\in\Tab^c_{\nu;\mu}$ by
\begin{align*}
\sT_1(i,j)&=\begin{cases*}
i & if $i<k$ or ($i=k$, $j\le l$) or ($i=k+1$, $j\le l$),\\
i+2 & otherwise,
\end{cases*}
\shortintertext{and}
\sT_2(i,j)&=\begin{cases*}
\sT(i,j) & if $i\le k+1$,\\
\sT(i-2,j+l) & if $i=k+2$ or $i=k+3$,\\
\sT(i-2,j) & if $i>k+3$.
\end{cases*}
\end{align*}
Then the leading term of $m_{\sT_2}\circ_\nu m_{\sT_1}$ is $m_\sT$.
Applying the decomposition in Case 1 above to $m_{\sT_1}$,
we can replace $m_{\sT_2}\circ_\nu m_{\sT_1}$ by a linear combination of
tableaux with smaller lengths so the induction goes forward.

Next consider the other case $\sT(k+1,l)=i$ and $\sT(k,l+1)=i$ or $\circleI$.
Let $(k+1,l_1+1)$ and $(k,l_2)$ be the ends of the bars $\young(ii\scdots i)$
which start from $(k+1,l)$ and $(k,l+1)$ respectively.
In this case we define
\[\nu\coloneqq(\lambda_1,\dots,\lambda_{k-1},l,l_1,l_2-l_1,\lambda_k-l_2,\lambda_{k+1}-l,
\lambda_{k+2},\lambda_{k+3},\dotsc)\]
and $\sT_1\in\Tab_{\lambda;\nu}$, $\sT_2\in\Tab^c_{\nu;\mu}$ by
\begin{align*}
\sT_1(i,j)&=\begin{cases*}
i & if $i<k$ or ($i=k$, $j\le l$) or ($i=k+1$, $j\le l_1$),\\
k+2 & if ($i=k$, $l<j\le l_2$) or ($i=k+1$, $l_1<j\le l$),\\
i+3 & otherwise,
\end{cases*}
\shortintertext{and}
\sT_2(i,j)&=\begin{cases*}
\sT(i,j) & if $i\le k+1$,\\
\sT(k+1,j+l_1) & if $i=k+2$, $j\le l-l_1$\\
\sT(k,j+l_1) & if $i=k+2$, $j> l-l_1$\\
\sT(k,j+l_2) & if $i=k+3$,\\
\sT(k+1,j+l) & if $i=k+4$,\\
\sT(i-3,j) & if $i>k+4$.
\end{cases*}
\end{align*}
In addition, if $\sT(k,l_2)=\circleI$ is circled,
we remove the corresponding circle from $\sT_2$ and
move it to $\sT_1(k,l_2)=i$.
Then the top term of $m_{\sT_2}\circ_\nu m_{\sT_1}$ is again $\pm m_\sT$.
For example, when
\[\sT=\young(1112\circleB\circleC4,1223\circleC,\circleB5)\]
and $(k,l)=(1,3)$, by picking up entries at
\[\young(\ \ \ \bullet\bullet\ \ ,\ \bullet\bullet\ \ ,\ \ )\]
then moving a circle from $\sT_2$ to $\sT_1$ we obtain
\[\sT_1=\young(1113\circleC44,23355,66)\quad\text{and}\quad
\sT_2=\young(111,1,2222,\circleC4,3\circleC,\circleB5).\]
Now according to Case 2 above, up to lower terms, we can replace $\sT_1$
by a linear combination of $m_\sS$
such that $\sS^\times=\sT_1^\times$,
$\sS(k+1,l)=\text{\Large\textcircled{\raisebox{1pt}{\normalsize$\scriptstyle k\!+\!2$}}}$ and
$\sS$ has circles only at the boxes $(k,1),\dots,(k,\lambda_k)$ and $(k+1,1),\dots,(k+1,l)$.
Then the top term of $m_{\sT_2}\circ_\nu m_\sS$ is $\pm m_\sR$,
where $\sR$ also satisfies that $\sR^\times=\sT^\times$, $\sR(k+1,l)=\circleI$
and the positions of circles of $\sR$ and $\sT$ only differ at these boxes.
Hence $\sR$ also does not have bad patterns at the bottom right region of $(k,l)$.
This completes the induction.
\end{proof}

Let $\STab^{c\prime}_{\lambda;\mu}$ be the subset of $\STab^c_{\lambda;\mu}$
consisting of tableaux whose entries in the rightmost of each rows are not circled.
Then clearly $\#\STab^c_{\lambda;\mu}=2^{l(\lambda)}\cdot\#\STab^{c\prime}_{\lambda;\mu}$
where $l(\lambda)$ is the number of non-zero components of $\lambda$.
\begin{corollary}
$S^c_{\lambda;\mu}$ is spanned by $\set{m_\sT}{\sT\in\STab^{c\prime}_{\lambda;\mu}}$
over $\Gamma_\lambda$.
\end{corollary}

By a similar proof we can prove that $S^{c*}_{\lambda;\mu}\coloneqq M^{c(\lambda)}_{\mu;\lambda}$
is also spanned by $\set{m_{\sT^*}}{\sT\in\STab^c_{\lambda;\mu}}$.
Now parallel to Theorem~\ref{thm:cellular_basis}
we obtain the following basis theorem.
\begin{theorem}
When $\kk=\QQ(a,q)$,
$M^c_{\lambda;\mu}$ has a basis
\[\bigsqcup_{\nu\colon\textup{strict partition}}
\set{m_\sS\circ_\nu m_{\sT^*}}{\sS\in\STab^c_{\nu;\mu},\sT\in\STab^{c\prime}_{\nu;\lambda}}.\]
\end{theorem}
\begin{proof}
The proof of that this set spans $M^c_{\lambda;\mu}$ is same as that of Theorem~\ref{thm:cellular_basis}.
For that of linear independence we use the one-to-one correspondence
\[\Tab^c_{\lambda;\mu}\xleftrightarrow{1:1}\bigsqcup_{\nu\colon\textup{strict partition}}
\STab^c_{\nu;\lambda}\times\STab^{c\prime}_{\nu;\mu}\]
induced by Sagan's shifted Knuth correspondence~\cite[Theorem~8.1]{Sagan87}.
\end{proof}

\begin{corollary*}
\begin{enumerate}
\item $S^c_{\lambda;\mu}$
has a basis $\set{m_\sT}{\sT\in\STab^c_{\lambda;\mu}}$.
In particular,
\[S^c_{\lambda;\lambda}\simeq\begin{cases*}
\Gamma_\lambda & if $\lambda$ is a strict partition,\\
0 & otherwise.
\end{cases*}\]
\item The product
\[\circ_{\nu}\colon S^c_{\nu;\mu}\otimes_{\Gamma_\nu} S^{c*}_{\nu;\lambda}
\to M^{c(\nu)}_{\lambda;\mu}\]
is injective.
\item $H^c_n$ and $\rS^c_{r,n}$ are standardly based algebras.
\qedhere
\end{enumerate}
\end{corollary*}

\begin{remark}
The basis theorem above for $H^c_n$ (i.e.\ $\lambda=\mu=(1^n)$)
also holds in the following more weaker conditions:
$\kk$ is an arbitrary commutative ring and $2aq\in\kk^\times$,
and the $q^2$-integers $\qqn{k}$ are also invertible for $1\le k\le n/2$.
Note that we need not to use Case 2 in the proof of Lemma~\ref{lem:semistandard_super}.
This implies that $H^c_n$ is also standardly based over $\{\Gamma_\lambda\}$
in these conditions.
\end{remark}

\subsection{Modules in non-integral rank}
Now let $\kk$ be an arbitrary commutative ring again.
Hereafter in this section, we assume that $q\in\kk^\times$.
In an unpublished work~\cite{Mori04} of the author,
it is proved that there exists a supercategory $\lMod{\uH^c_t}$
which can be regarded as ``the module category of $H^c_t$ for $t\notin\NN$''.
It covers all ordinary module categories $\lMod{H^c_n}$ for all $n\in\NN$,
in the sense that there is a full and surjective functor $\lMod{\uH^c_n}\to\lMod{H^c_n}$.
This is an analogue of the Deligne's category~\cite{Deligne07}
for the symmetric group $\fS_t$.
Here we list results for this category
we need for our purpose.

A \term{fakecomposition} of $n\in\NN$ is a pair
$\lambda=(\lambda_1,\lambda')$ of a number $\lambda_1\in\ZZ$ and a composition $\lambda'$
such that $n=\lambda_1+\abs{\lambda'}$. Here $\lambda_1$ can be a negative number.
For such $\lambda$, we let $\lambda_i\coloneqq\lambda'_{i-1}$ for each $i\ge2$
and write $\lambda=(\lambda_1,\lambda_2,\lambda_3,\dotsc)$.
For a natural number $d\ge\abs{\lambda'}$, let
$\lambda|_d\coloneqq(d-\abs{\lambda'},\lambda_2,\lambda_3,\dotsc)$ be an ordinary composition of $d$.
For two fakecompositions $\lambda,\mu$ of $n$,
we define the set of row-standard circled tableaux of shape $\lambda$ of weight $\mu$ by
\[\uTab^c_{\lambda;\mu}\coloneqq\varinjlim_d\Tab^c_{\lambda|d;\mu|d}.\]
Here the map $\Tab^c_{\lambda|d;\mu|d}\hookrightarrow\Tab^c_{\lambda|d+1;\mu|d+1}$
is inserting $\young(1)$ on the first row of the tableau from left.
The set $\Tab^c_{\lambda|d;\mu|d}$ is stable for a large number $d\gg0$,
so the direct limit converges to a finite set which contains $\Tab^c_{\lambda;\mu}$.
Intuitively we think of Young diagrams whose first rows are very long:
\[\Yboxdim6pt\yng(28,3,1).\]
For example, the set $\uTab^c_{(3,3);(1,3,2)}$ has tableaux such as
\[\young(1111\scdots\scdots112233,11\circleB)
\qquad\text{and}\qquad
\young(1111\scdots\scdots1\circleA223,123)\]
which are not in $\Tab^c_{(3,1);(1,2,1)}$.
The statements below are proved by the author.
\begin{theorem}
Let $\lambda$, $\mu$, and $\nu$ denote fakecompositions of $n\in\NN$.
\begin{enumerate}
\item For each pair of $\lambda,\mu$,
there exists a supermodule $\uM^c_{\lambda;\mu}$ which has a basis
$\set{\um_\sS}{\sS\in\uTab^c_{\lambda;\mu}}$.
\item For each triple $\lambda,\mu,\nu$,
there exists a product map
\[\circ_\mu\colon\uM^c_{\mu;\nu}\otimes\uM^c_{\lambda;\mu}\to\uM^c_{\lambda;\nu}\]
where the coefficient of $\um_\sR\circ_\mu\um_\sS$ on each $\um_\sT$
is given by a polynomial in the variable $[n]$ over $\QQ(a,q)$
contained in $\ZZ[a,q^\pm]$ for all $n\in\NN$.
\item If $\lambda$ and $\mu$ are compositions (i.e.\ $\lambda_1,\mu_1\ge0$)
there exists a surjective map
\begin{align*}
P\colon\uM^c_{\lambda;\mu}&\twoheadrightarrow M^c_{\lambda;\mu}\\
\um_\sS&\mapsto\begin{cases*}
m_\sS & if $\sS\in\Tab^c_{\lambda;\mu}$,\\
0 & otherwise
\end{cases*}
\end{align*}
which respects the product above.
\end{enumerate}
\end{theorem}
Note that the number $\#_{ij}(\sS^\times)$ is also well-defined for $(i,j)\neq(1,1)$.
We also define $\#_{11}(\sS^\times)$ by
\[\#_{11}(\sS^\times)\coloneqq n-\sum_{(i,j)\neq(1,1)}\#_{ij}(\sS^\times).\]
More preceisely, the map $P$
sends $\um_\sS\in\uM^c_{\lambda;\mu}$ to the corresponding
$m_\sS\in M^c_{\lambda;\mu}$ if $\#(\sS^\times)>0$ or
$\#(\sS^\times)=0$ and $\sS$ has no $\circleA$ in its first row, and otherwise zero.
For example, when $\lambda=\mu=(1,1)$,
$P$ is given by
\begin{align*}
\young(1\scdots11,2)&\mapsto\young(1,2),&
\young(1\scdots11,\circleB)&\mapsto\young(1,\circleB),&
\young(1\scdots1\circleA,2)&\mapsto\young(\circleA,2),&
\young(1\scdots1\circleA,\circleB)&\mapsto\young(\circleA,\circleB),\\[2mm]
\young(1\scdots12,1)&\mapsto\young(2,1),&
\young(1\scdots12,\circleA)&\mapsto\young(2,\circleA),&
\young(1\scdots1\circleB,1)&\mapsto\young(\circleB,1),&
\young(1\scdots1\circleB,\circleA)&\mapsto\young(\circleB,\circleA),\\[2mm]
\young(1\scdots\circleA2,1)&\mapsto0,&
\young(1\scdots\circleA2,\circleA)&\mapsto0,&
\young(1\scdots\circleA\circleB,1)&\mapsto0,&
\young(1\scdots\circleA\circleB,\circleA)&\mapsto0.
\end{align*}

The superalgebra $\uGamma_\lambda$ which has generators
$\ugamma_{\lambda;1}$ and $\gamma_{\lambda;2},\gamma_{\lambda;3},\dotsc$
is defined similarly as $\Gamma_\lambda$
but the relation $\ugamma_{\lambda;1}=0$ is omitted even if $\lambda_1=0$.
Hence as an abstract superalgebra, we have
\[\uGamma_\lambda\simeq C_1(a\qqn{\lambda_1})\otimes\Gamma_{\lambda'}.\]
Then there is an inclusion of superalgebra
$\uGamma_\lambda\hookrightarrow\uM^c_{\lambda;\lambda}$,
so that the product can be defined as
\[\circ_\mu\colon\uM^c_{\mu;\nu}\otimes_{\uGamma_\mu}
\uM^c_{\lambda;\mu}\to\uM^c_{\lambda;\nu}.\]

We define the dominance order on the set of fakecompositions so that
$\lambda\le\mu$ if and only if $\lambda|_d\le\mu|_d$ for all $d\gg0$,
then the reversed dominance order is still well-founded.
According to this dominance order
we similarly define the quotient supermodules
$\uM^c_{\lambda;\mu}\twoheadrightarrow\uM^{c(\nu)}_{\lambda;\mu}$
and $\uS^c_{\lambda;\mu}\coloneqq\uM^{c(\lambda)}_{\lambda;\mu}$.
Then for compositions the surjective map
$P\colon\uS^c_{\lambda;\mu}\twoheadrightarrow S^c_{\lambda;\mu}$ is well-defined.
Since the same proof works,
it satisfies the following standardly filtered property again.

\begin{theorem*}
\ 
\begin{enumerate}
\item $\uS^c_{\lambda;\lambda}$ is spanned by $\um_\lambda$ over $\uGamma_\lambda$.
\item $\uS^c_{\lambda;\mu}=0$ unless $\lambda\ge\mu$.
\qedhere
\end{enumerate}
\end{theorem*}

We say that a fakecomposition $\lambda$ is a \term{fakepartition}
if $\lambda'$ is a partition,
and say it is \term{strict} if so is $\lambda'$.
The sets $\underline\STab^c_{\lambda;\mu}$ and
$\underline\STab^{c\prime}_{\lambda;\mu}$ are defined by direct limits
similarly as above.
Then the standardly based structure of the category of parabolic fakemodules
is obtained by the same proofs as for the ordinary case.

\begin{theorem*}
Assume $\kk=\QQ(a,q)$. Then
\begin{enumerate}
\item $\uM^c_{\lambda;\mu}$ has a basis
\[\bigsqcup_{\nu\colon\textup{strict fakepartition}}
\set{\um_\sS\circ_\nu\um_{\sT^*}}{\sS\in\underline\STab^c_{\nu;\mu},\sT\in\underline\STab^{c\prime}_{\nu;\lambda}}.\]
\item $\uS^c_{\lambda;\mu}$ has a basis $\set{\um_\sT}{\sT\in\underline\STab^c_{\lambda;\mu}}$ so
\[\uS^c_{\lambda;\lambda}=\begin{cases*}
\uGamma_\lambda & if $\lambda$ is a strict fakepartition,\\
0 & otherwise.
\end{cases*}\]
\item The product
\[\circ_\nu\colon\uS^c_{\nu;\mu}\otimes_{\uGamma_\lambda}\uS^{c*}_{\nu;\lambda}\to\uM^{c(\nu)}_{\lambda;\mu}\]
is injective.
\qedhere
\end{enumerate}
\end{theorem*}

\subsection{Identification of the quotient superalgebras}
In order to classify the simple modules,
we first determine the quotient superalgebra
$M^c_{\lambda;\lambda}\twoheadrightarrow S^c_{\lambda;\lambda}$.
In this computation the superalgebras
$\uM^c_{\lambda;\lambda}\twoheadrightarrow\uS^c_{\lambda;\lambda}$
which cover them are used.

\begin{lemma}
Let $\lambda=(m,k)$ with $m>k$. Then in $S^c_{\lambda;\lambda}$ we have
\[\young(11\scdots\circleA2\scdots2,1\scdots1)\equiv
(-1)^kq^{\binom{k}{2}}(\gamma_{\lambda;1}-q^{m-k}\gamma_{\lambda;2})m_\lambda.\]
\end{lemma}
\begin{proof}
By \ref{lem:scalar_lemma}, we have
\[\young(11\scdots12\scdots2,1\scdots1)\equiv(-1)^kq^{\binom{k}{2}}m_\lambda.\]
Hence the equation is implied by
\begin{align*}
\gamma_{\lambda;2}\cdot\young(11\scdots12\scdots2,1\scdots1)
&=\young(11\scdots12\scdots\circleB,1\scdots1)
\shortintertext{and}
\young(11\scdots12\scdots2,1\scdots1)\cdot\gamma_{\lambda;1}
&=\young(11\scdots\circleA2\scdots2,1\scdots1)
+q^{m-k}\young(11\scdots12\scdots\circleB,1\scdots1).
\qedhere
\end{align*}
\end{proof}

\begin{corollary*}
For a fakepartition $\lambda=(m,k)$, in $\uS^c_{\lambda;\lambda}$ we have
\[\young(11\scdots\circleA2\scdots2,1\scdots1)\equiv
(-1)^kq^{\binom{k}{2}}(\ugamma_{\lambda;1}-q^{m-k}\gamma_{\lambda;2})\um_\lambda.
\qedhere\]
\end{corollary*}

\begin{lemma}
Let $\lambda$ be a partition. Then
\begin{enumerate}
\item if $\lambda_1>\lambda_2$ then
we have $\uS^c_{\lambda;\lambda}\simeq S^c_{\lambda;\lambda}$,
\item if $\lambda_1=\lambda_2$ then
$\Ker(\uS^c_{\lambda;\lambda}\twoheadrightarrow S^c_{\lambda;\lambda})$
is generated by $(\ugamma_{\lambda;1}-\gamma_{\lambda;2})m_\lambda$
as a 2-sided ideal.
\end{enumerate}
\end{lemma}
\begin{proof}
For the case (2), as in the proof of Lemma~\ref{lem:strict_partition}
we have $(\gamma_{\lambda;1}-\gamma_{\lambda;2})m_\lambda\equiv0$ in $S^c_{\lambda;\lambda}$
so the kernel contains $(\ugamma_{\lambda;1}-\gamma_{\lambda;2})\um_\lambda$.
We prove the converse inclusions.

$\Ker(\uM^c_{\lambda;\lambda}\twoheadrightarrow M^c_{\lambda;\lambda})$
is spanned by $\um_\sT$ for $\sT\in\uTab^c_{\lambda;\lambda}$ which
satisfies either of the condition that $\#_{11}(\sT^\times)<0$ or
that $\#_{11}(\sT^\times)=0$ with $\circleA$ in its first row.
If $\lambda_1>\lambda_2$, we have $\lambda_1-\#_{11}(\sT^\times)>\lambda_2$ for such
$\sT$ so that $\um_\sT\equiv0$
in $\uS^c_{\lambda;\lambda}$ as we did in the proof of
Lemma~\ref{lem:ideal_identification}. This implies
that $\Ker(\uM^c_{\lambda;\lambda}\twoheadrightarrow M^c_{\lambda;\lambda})$
is already zero in $\uS^c_{\lambda;\lambda}$; in other words,
$\uS^c_{\lambda;\lambda}\simeq S^c_{\lambda;\lambda}$.
In the other case $\lambda_1=\lambda_2$, we also have $\um_\sT\equiv0$ if
$\#_{11}(\sT^\times)<0$.
Otherwise by applying local transformation on the second row or below
$\um_\sT$ can be transformed into a linear combination
of tableaux $\sS$ such that $\sT(1,j)=\sS(1,j)$ for all $j$ and
$\#_{21}(\sS^\times)=\lambda_2$, that is, which is in the form
\[\young(1111\scdots\circleA\ \ \ \ ,1111,\ \ \ \ ,\ \ )
\qquad\text{or}\qquad
\young(1111\scdots\circleA\ \ \ \ ,111\circleA,\ \ \ \ ,\ \ ).\]
By the corollary above, for the special case we have
\begin{align*}
\young(1111\scdots\circleA2222,1111,3333,44)&\equiv
(-1)^{\lambda_2}q^{\binom{\lambda_2}{2}}
(\ugamma_{\lambda;1}-\gamma_{\lambda;2})\um_\lambda
\shortintertext{in $\uS^c_{\lambda;\lambda}$, so that}
\young(1111\scdots\circleA2222,111\circleA,3333,44)&\equiv
(-1)^{\lambda_2}q^{\binom{\lambda_2}{2}}
(\ugamma_{\lambda;1}-\gamma_{\lambda;2})\gamma_{\lambda;2}\um_\lambda.
\end{align*}
Every such $\sS$ above can be made by multiplying elements
to these tableaux from left.
Hence the image of $\Ker(\uM^c_{\lambda;\lambda}\twoheadrightarrow M^c_{\lambda;\lambda})$ in $\uS^c_{\lambda;\lambda}$
is generated by $(\ugamma_{\lambda;1}-\gamma_{\lambda;2})\um_\lambda$ as a 2-sided ideal.
\end{proof}

\begin{lemma}
For a fakepartition $\lambda=(\lambda_1,\lambda')$, we have
$\uS^c_{\lambda;\lambda}\simeq C_1(a\qqn{\lambda_1})\otimes S^c_{\lambda';\lambda'}$.
\end{lemma}
\begin{proof}
Since these modules are preserved by extension of scalars,
it suffices to prove for the universal ring $\kk=\ZZ[a,q^\pm]$.
Let
\[V\coloneqq\sum_{\nu>\lambda,\,\nu_1>\lambda_1}\uM^{c\,\nu}_{\lambda;\lambda}
\qquad\text{and}\qquad
W\coloneqq\sum_{\nu>\lambda,\,\nu_1=\lambda_1}\uM^{c\,\nu}_{\lambda;\lambda}\]
so that $\uM^{c\,>\nu}_{\lambda;\lambda}=V+W$. On the other hand, let
\[T\coloneqq\set{\sT\in\Tab^c_{\lambda;\lambda}}{\sT^\times(1,j)=1\text{ for all }j}\] and
\[X\coloneqq\kk\set{\um_\sT}{\sT\in\Tab^c_{\lambda;\lambda}\setminus T},\qquad
Y\coloneqq\kk\set{\um_\sT}{\sT\in T}\]
so that $\uM^c_{\lambda;\lambda}=X\oplus Y$.
Since $\uGamma_\lambda\um_\lambda,W\subset Y$ we have
$\uM^c_{\lambda;\lambda}=V+Y$. Hence
\[\uS^c_{\lambda;\lambda}
=\uM^c_{\lambda;\lambda}/\uM^{c\,>\nu}_{\lambda;\lambda}
=(V+Y)/(V+W)\simeq Y/((V\cap Y)+W).\]
For a $\ZZ[a,q^\pm]$-module $M$,
let $\tilde M\coloneqq M\otimes_{\ZZ[a,q^\pm]}\QQ(a,q)$ be its localization.
By the cellular basis theorem, we have
$\dim\tilde V+\dim\tilde W=\dim\smash{\underline{\tilde M}}^c_{\lambda;\lambda}$
and \[\dim\tilde W=\sum_{\nu>\lambda,\,\nu_1=\lambda_1}\#\underline\STab^c_{\nu;\lambda}\cdot\#\underline\STab^{c\prime}_{\nu;\lambda}.\]
On the other hand, since we can view $\lambda_1$ as a sufficiently large number,
we have a natural one-to-one correspondence
\[\set{\text{strict fakepartition }\nu}{\nu>\lambda,\,\nu_1=\lambda_1}
\onetoone\set{\text{strict partition }\nu'}{\nu'>\lambda'}\]
and for such $\nu$,
\[\underline\STab^c_{\nu;\lambda}\onetoone\{1,\circleA\}\times\STab^c_{\nu';\lambda'}
\qquad\text{and}\qquad
\underline\STab^{c\prime}_{\nu;\lambda}\onetoone\STab^{c\prime}_{\nu';\lambda'}.\]
Then by using the shifted Knuth correspondence for $\lambda'$,
we obtain $\dim\tilde Y=\dim\tilde W$.
Since localization of modules is exact, from the exact sequence
\[0\to V\cap Y\to V\oplus Y\to\uM^c_{\lambda;\lambda}\to0\]
we obtain
\[0\to \widetilde{V\cap Y}\to \tilde V\oplus\tilde Y\to\smash{\underline{\tilde M}}^c_{\lambda;\lambda}\to0.\]
By comparison of dimensions we have $\widetilde{V\cap Y}=0$.
Since $V\cap Y\subset Y$ is a torsion-free module over the integral domain
$\ZZ[a,q^\pm]$, it implies $V\cap Y=0$.
Hence we have
\[\uS^c_{\lambda;\lambda}\simeq Y/W
\simeq C_1(a\qqn{\lambda_1})\otimes
S^c_{\lambda';\lambda'}.
\qedhere\]
\end{proof}

These two lemmas bring us
the following identification of the superalgebras
$S^c_{\lambda;\lambda}$ and $\uS^c_{\lambda;\lambda}$.
\begin{theorem}
Recall the assumption $q\in\kk^\times$.
\begin{enumerate}
\item For a partition $\lambda$, the 2-sided ideal
$\Ker(\Gamma_\lambda\twoheadrightarrow S^c_{\lambda;\lambda})$
is generated by $\gamma_i-\gamma_j$ for all $i,j$ such that $\lambda_i=\lambda_j$.
\item For a fakepartition $\lambda$, the 2-sided ideal
$\Ker(\uGamma_\lambda\twoheadrightarrow\uS^c_{\lambda;\lambda})$
is generated by $\gamma_i-\gamma_j$ for all $i,j$ such that $\lambda_i=\lambda_j$ and
$i,j\ge2$.
\end{enumerate}
\end{theorem}
\begin{proof}
We use a mutual induction for (1) and (2) on the number of components of $\lambda$.
First let $\lambda$ be a fakepartition and suppose that (1) holds for $\lambda'$.
Then
\[\Ker(\uGamma_\lambda\twoheadrightarrow\uS^c_{\lambda;\lambda})
\simeq C_1(a\qqn{\lambda_1})\otimes\Ker(\Gamma_{\lambda'}\twoheadrightarrow S^c_{\lambda';\lambda'})\]
has a generating set above.
Next let $\lambda$ be a partition of $n>0$ and suppose (2) holds for $\lambda$.
We have a commutative square
\[\begin{tikzcd}
\uGamma_\lambda \arrow[two heads]{r} \arrow[two heads]{d} &
\uS^c_{\lambda;\lambda} \arrow[two heads]{d} \\
\Gamma_\lambda \arrow[two heads]{r} &
S^c_{\lambda;\lambda}
\end{tikzcd}\]
where $\uGamma_\lambda\simeq\Gamma_\lambda$ since $\lambda_1>0$.
Hence as a generating set of the kernel of $\Gamma_\lambda\twoheadrightarrow S^c_{\lambda;\lambda}$
we can take the union of that of $\Gamma_\lambda\simeq\uGamma_\lambda\twoheadrightarrow\uS^c_{\lambda;\lambda}$
and that of $\uS^c_{\lambda;\lambda}\twoheadrightarrow S^c_{\lambda;\lambda}$.
\end{proof}

Consequently we obtain the following classification of simple modules of $\rS^c_{r,n}$.
We remark that $S^c_{\lambda;\mu}$ is not free over $\kk$ in general
even if in this case $q\in\kk^\times$.
\begin{theorem}
Suppose $q\in\kk^\times$.
For a partition $\lambda$, let $\Theta_\lambda$ be the 2-sided ideal
generated by $\gamma_{\lambda;i}-\gamma_{\lambda;j}$ above.
Then there is a one-to-one correspondence
\[\Irr(\rS^c_{r,n})\onetoone\bigsqcup_{\nu=(\nu_1,\dots,\nu_r)\colon\textup{partition of $n$}}
\Irr(\Gamma_\lambda/\Theta_\lambda).\]
\end{theorem}

Note that for a partition $\lambda=(k,k)$, we have
\[\Gamma_\lambda/\Theta_\lambda\simeq\Gamma_k\otimes(\kk/2a\qqn{k}\kk)\]
since $a\qqn{k}=\gamma_{\lambda;1}^2\equiv\gamma_{\lambda;1}\gamma_{\lambda;2}
=-\gamma_{\lambda;2}\gamma_{\lambda;1}\equiv-\gamma_{\lambda;1}^2=-a\qqn{k}$.
In addition, clearly
$2a\qqn{k}\kk+2a\qqn{l}\kk=2a\qqn{\gcd\{k,l\}}\kk$.
Thus for a general partition $\lambda$, let $\mu$ be
the strict partition obtained by removing duplicate components of $\lambda$
and let $k_1,\dots,k_r$ be such components, then
\[S^c_{\lambda;\lambda}\simeq\Gamma_\lambda/\Theta_\lambda\simeq\Gamma_\mu\otimes(\kk/2a\qqn{\gcd\{k_1,\dots,k_n\}}\kk).\]
In particular, $S^c_{\lambda;\lambda}=0$
if and only if $2a\qqn{\gcd\{k_1,\dots,k_n\}}\in\kk^\times$.

Remember that when $\kk$ is a field the Clifford superalgebra $\Gamma_\lambda$
has a unique simple module up to parity change.
For $e\ge2$, we say that a partition $\lambda$ is
\term{$e$-strict} if $\lambda_i=\lambda_j$, $i\neq j$ implies
$e|\lambda_i$.
For convention the word $\infty$-strict stands for strict.
For a superalgebra $A$,
let $\Irr(A)/\Pi$ be a quotient set of $\Irr(A)$
on which $V\in\Irr(A)$ is identified with its parity change $\Pi V$.
\begin{corollary*}
Suppose that $\kk$ is a field and $2aq\in\kk^\times$.
Let $e_2$ be the $q^2$-characteristic of $\kk$.
Then there is a one-to-one correspondence
\[\Irr(\rS^c_{r,n})/\Pi\onetoone
\{\nu=(\nu_1,\dots,\nu_r)\colon\textup{$e_2$-strict partition of $n$}\}.
\qedhere\]
\end{corollary*}
The case $2a=0$ is easier.
\begin{corollary*}
Suppose that $\kk$ is a field and $q\in\kk^\times$, $2a=0$.
Then there is a one-to-one correspondence
\[\Irr(\rS^c_{r,n})/\Pi\onetoone
\{\nu=(\nu_1,\dots,\nu_r)\colon\textup{partition of $n$}\}.
\qedhere\]
\end{corollary*}

\subsection{Identification of the ideals}
We keep assuming that $q\in\kk^\times$.
Finally we reach to the classification of simple modules
of the Hecke--Clifford superalgebra $H^c_n$.
Now let $J^c_\lambda\subset\Gamma_\lambda$
be the pullback of the 2-sided ideal
$m_\lambda\cdot S^c_\lambda\subset S^c_{\lambda;\lambda}$ via
the surjective map $\Gamma_\lambda\twoheadrightarrow\Gamma_\lambda/\Theta_\lambda\simeq S^c_{\lambda;\lambda}$.
We determine this ideal as follows.

For $n\in\NN$, let $K_n\subset\kk$ be the ideal generated by the elements
\[\set[\big]{\bigl({\textstyle\frac{a(q-1)}{[2]}}\bigr)^s[n]!}{0\le s\le n/2}.\]
Then by Lemma~\ref{lem:gamma}, we have
$m_n C_n m_n=K_n m_n\oplus K_{n-1}\gamma^L_n m_n$.
Since it is a 2-sided ideal of $\Gamma_n m_n$, the following statement holds.
\begin{lemma*}
There are inclusions $a\qqn{n}K_{n-1}\subset K_n\subset K_{n-1}$.
\end{lemma*}

\begin{lemma}
Let $\lambda=(\lambda_1,\lambda_2,\dots,\lambda_r)$ be a partition.
For each $i$, let
$\Delta_{\lambda;i}\subset\Gamma_\lambda$ be the supermodule
\begin{multline*}
\Delta_{\lambda;i}\coloneqq K_{\lambda_i-\lambda_{i+1}}
\oplus K_{\lambda_i-\lambda_{i+1}}\gamma_{\lambda;i+1}
\oplus K_{\lambda_i-\lambda_{i+1}-1}(\gamma_{\lambda;i}-q^{\lambda_i-\lambda_{i+1}}\gamma_{\lambda;i+1})\\
\oplus K_{\lambda_i-\lambda_{i+1}-1}(\gamma_{\lambda;i}-q^{\lambda_i-\lambda_{i+1}}\gamma_{\lambda;i+1})\gamma_{\lambda;i+1}
\end{multline*}
and let
$\Delta_\lambda\coloneqq\Delta_{\lambda;r}\dotsm
\Delta_{\lambda;2}\Delta_{\lambda;1}$.
Then $\Delta_\lambda\subset\Gamma_\lambda$ is a 2-sided ideal.
\end{lemma}
\begin{proof}
For simplicity we write $\Delta=\Delta_\lambda$,
$\Delta_i=\Delta_{\lambda;i}$ and $\gamma_i=\gamma_{\lambda;i}$.
First we prove
\[\Delta_i\gamma_i\subset\Delta_i+\gamma_{i+1}\Delta_i,\qquad
\Delta_i\gamma_{i+1}\subset\Delta_i,\qquad
\Delta_i\gamma_j=\gamma_j\Delta_i\quad\text{for }j\neq i,i+1.\]
The second and the third inclusions are clear so we prove the first one.
Since $K_n\subset K_{n-1}$, the inclusion
$K_{\lambda_i-\lambda_{i+1}}\gamma_i\subset\Delta_i$
is also obvious. We also have
\begin{align*}
(\gamma_i-q^{\lambda_i-\lambda_{i+1}}\gamma_{i+1})\gamma_i
-q^{\lambda_i-\lambda_{i+1}}\gamma_{i+1}(\gamma_i-q^{\lambda_i-\lambda_{i+1}}\gamma_{i+1})
&=\gamma_i^2-q^{2(\lambda_i-\lambda_{i+1})}\gamma_{i+1}^2\\
&=a\qqn{\lambda_i-\lambda_{i+1}}
\end{align*}
so that
$K_{\lambda_i-\lambda_{i+1}-1}(\gamma_i-q^{\lambda_i-\lambda_{i+1}}\gamma_{i+1})\gamma_i\subset\Delta_i+\gamma_{i+1}\Delta_i$
by $a\qqn{n}K_{n-1}\subset K_n$.
Putting them together we obtain $\Delta_i\gamma_i\subset\Delta_i+\gamma_{i+1}\Delta_i$
as desired.
Then
\begin{align*}
\Delta\gamma_i&=\Delta_r\dotsm\Delta_{i-1}\gamma_i\dotsm\Delta_1
\subset\Delta_r\dotsm\Delta_{i-1}\dotsm\Delta_1=\Delta
\shortintertext{for $i\ge2$, and}
\Delta\gamma_1&\subset\Delta+\Delta_r\dotsm\Delta_2\gamma_2\Delta_1
\subset\Delta+\Delta_r\dotsm\gamma_3\Delta_2\Delta_1\subset\dots\subset\Delta
\end{align*}
so $\Delta$ is a right ideal.
By the equation above we also have inclusions
\[\gamma_i\Delta_i\subset\Delta_i,\qquad
\gamma_{i+1}\Delta_i\subset\Delta_i+\Delta_i\gamma_i\]
which imply that $\Delta$ is also a left ideal in a similar manner.
\end{proof}

\begin{lemma}
For $\lambda=(\lambda_1,\lambda_2,\dots,\lambda_r)$ above, we have
$\Delta_\lambda^r+\Theta_\lambda\subset J^c_\lambda\subset\Delta_\lambda+\Theta_\lambda$.
\end{lemma}
\begin{proof}
Parallel to the proof of Lemma~\ref{lem:ideal_identification}.
So first we prove $J^c_\lambda\subset\Delta_\lambda+\Theta_\lambda$.
Take an arbitrary $\sT\in\Tab^c_\lambda$.
Let $\mu\coloneqq(\lambda_1,1^{n-\lambda_1})$ and define
$\sS\in\Tab^c_{\lambda;\mu}$
which has underlying tableau $\sS^\times=\sT^\times|_\mu$
and for each its bar $\young(11\scdots1)$ it has a circle
if and only if there are odd number of circles in the corresponding boxes in $\sT$.
Let $k\coloneqq\#_{11}(\sS^\times)$ and let $p\coloneqq0$ if $\sS$ does not
have $\circleA$ in its first row, and otherwise $p\coloneqq1$.
Then by Lemma~\ref{lem:gamma} we have
\[m_\mu\cdot m_\sT\in K_{k-p} m_\sS.\]
If $k<\lambda_1-\lambda_2$, we have $m_\sS\equiv0$.
If $k=\lambda_1-\lambda_2$ and $p=1$,
$m_\sS$ can be transformed into a linear combination of tableaux generated by
\begin{align*}
\young(1111\scdots\circleA2222,1111,3333,44)&\equiv(-1)^{\lambda_2}
q^{-\binom{\lambda_2}{2}}m_\lambda
(\gamma_{\lambda;1}-q^{\lambda_1-\lambda_2}\gamma_{\lambda;2})
\shortintertext{or}
\young(1111\scdots\circleA2222,111\circleA,3333,44)&\equiv(-1)^{\lambda_2}
q^{-\binom{\lambda_2}{2}}m_\lambda
(\gamma_{\lambda;1}-q^{\lambda_1-\lambda_2}\gamma_{\lambda;2})\gamma_{\lambda;2}
\end{align*}
as we did before. Hence
\[K_{k-p}m_\sS\subset(\unit^c_{\lambda_1}*S^c_{\lambda'})\cdot\Delta_{\lambda;1}.\]
In the other cases
we have $K_{k-p}\subset K_{\lambda_1-\lambda_2}$ so the inclusion above also holds.
By induction we may assume that
$m_{\lambda'}\cdot S^c_{\lambda'}\subset m_{\lambda'}\Delta_{\lambda'}$.
This implies $m_\lambda\cdot m_\sT\in m_\lambda\Delta_\lambda$ in $S^c_\lambda$.

We can prove the other inclusion by using circled tableaux whose underlying
tableau is $\sR$ in the proof of Lemma~\ref{lem:ideal_identification}.
By putting circles on suitable boxes of $\sR_\downarrow$
we can make arbitrary elements of
\[(\Delta_{\lambda;r}\dotsm\Delta_{\lambda;2}\Delta_{\lambda;1})
(\Delta_{\lambda;r}\dotsm\Delta_{\lambda;2})
\dotsm(\Delta_{\lambda;r}\Delta_{\lambda;r-1})\Delta_{\lambda;r}\supset\Delta_\lambda^r.\]
For example, for $\lambda=(6,4,1)$
{\renewcommand\circleI{\mathcircled{9}}
\renewcommand\circleJ{\mathcircled{10}}
\newcommand\TEN{10}
\begin{multline*}
m_\lambda\young(1\circleB789\circleJ,34\circleE\circleI,6)
=[2][3]!\young(1\circleA222\circleC,11\circleA\circleB,1)
=-[2][3]!\young(1\circleA22\circleB\circleC,1111,2)\gamma_3(\gamma_2-q^3\gamma_3)\\
=\dots
=q^6m_\lambda\cdot\gamma_3\cdot[3]!\cdot(\gamma_1-q^2\gamma_2)
\cdot\gamma_3\cdot[2](\gamma_2-q^3\gamma_3)\cdot1
\end{multline*}
}%
where $\gamma_1-q^2\gamma_2\in\Delta_{\lambda;1}$,
$[3]!,[2](\gamma_2-q^3\gamma_3)\in\Delta_{\lambda;2}$ and $\gamma_3,\gamma_3,1\in\Delta_{\lambda;3}$.
Hence we conclude that $m_\lambda\cdot S^c_{\lambda;\lambda}\supset m_\lambda\Delta_\lambda^r$.
\end{proof}

We state again the main theorem of this paper.

\begin{theorem*}
When $q\in\kk^\times$,
there is a one-to-one correspondence
\[\Irr(H^c_n)\onetoone\bigsqcup_{\lambda\colon\textup{partition of $n$}}\Irr^{\Delta_\lambda+\Theta_\lambda}_{\Theta_\lambda}(\Gamma_\lambda).
\qedhere\]
\end{theorem*}

Now assume that $\kk$ is a field.
By specializing this theorem
we obtain several classifications.
First consider the case $q\neq-1$.
In this case simply $K_n=[n]!\,\kk$.
Let $e$ (resp.\ $e_2$) be a $q$-characteristic (resp.\ $q^2$-) of $\kk$.
Then we have
\[e_2=\begin{cases*}
e & if $e$ is odd,\\
e/2 & if $e$ is even.
\end{cases*}\]
Let $\lambda$ be a partition.
If $\lambda_i>\lambda_{i+1}+e$, we have $\Delta_\lambda=0$ as before.
On the other hand if $\lambda_i<\lambda_{i+1}+e$ we have $\Delta_\lambda=\Gamma_\lambda$.
So suppose $\lambda_i=\lambda_{i+1}+e$ so that $K_{\lambda_i-\lambda_{i+1}}=0$
but $K_{\lambda_i-\lambda_{i+1}-1}=\kk$.
If $e_2|\lambda_i$ then
$\gamma_{\lambda_i}$ and $\gamma_{\lambda;i+1}$ are central nilpotent
so that they are contained in the Jacobson radical of $\Gamma_\lambda$.
Otherwise
\[(\gamma_{\lambda;1}-q^{\lambda_1-\lambda_2}\gamma_{\lambda;2})^2
=a\qqn{\lambda_1}+aq^{2(\lambda_1-\lambda_2)}\qqn{\lambda_2}
=2a\qqn{\lambda_1}\]
is invertible if and only if $2a\neq0$.
When $2a=0$, $K_{\lambda_i-\lambda_{i+1}-1}(\gamma_{\lambda;1}-q^{\lambda_1-\lambda_2}\gamma_{\lambda;2})$
generates a nilpotent ideal so is in the Jacobson radical also in this case.

Summarizing the above, we obtain the following results.
We say that an $e_2$-strict partition $\lambda$ is \term{$e$-restricted} if
\[\begin{cases*}
\lambda_i-\lambda_{i+1}<e & if $e_2|\lambda_i$,\\
\lambda_i-\lambda_{i+1}\le e & otherwise.\\
\end{cases*}\]

\begin{corollary*}
Suppose $\kk$ is a field and $2aq[2]\neq0$.
Then there is a one-to-one correspondence
\[\Irr(H^c_n)/\Pi\onetoone\{\textup{$e$-restricted $e_2$-strict partition of $n$}\}.
\qedhere\]
\end{corollary*}

The result is now coincides with the crystal $B(\Lambda_0)$
of type $\mathsf{A}^{(2)}_{e-1}$ for odd $e$
or of type $\mathsf{D}^{(2)}_{e/2}$ for even $e$
whose descriptions are obtained by
Kang~\cite{Kang03} and Hu~\cite{Hu06} respectively.

\begin{corollary*}
Suppose $\kk$ is a field and $q[2]\neq0$, $2a=0$.
Then there is a one-to-one correspondence
\[\Irr(H^c_n)/\Pi\onetoone\{\textup{$e$-restricted partition of $n$}\}.
\qedhere\]
\end{corollary*}

Next consider the case $q=-1$, so that $[k]=-(k\bmod 2)$ and $\qqn{k}=k$.
First assume that $2a\neq0$.
Let $p$ be the (ordinary) characteristic of $\kk$.
Then we have $K_n=\kk$ if $n<2p$ and otherwise $K_n=0$.
Hence by a similar arguments as above we obtain the following.

\begin{corollary*}
Suppose $\kk$ is a field of characteristic $p\neq2$ and $q=-1$, $a\neq0$.
Then there is a one-to-one correspondence
\[\Irr(H^c_n)/\Pi\onetoone\{\textup{$2p$-restricted $p$-strict partition of $n$}\}.
\qedhere\]
\end{corollary*}
Hence it corresponds to the crystal $B(\Lambda_0)$ of type $\mathsf{D}_p^{(2)}$ again.
S.~Tsuchioka informed the author that this result can be also obtained
from the supercategorification by
the \term{cyclotomic quiver Hecke superalgebra} given in \cite{KangKashiwaraOh13}.
As described in \cite[Theorem~3.13 and Lemma~4.8]{KangKashiwaraTsuchioka11}
it is considered equivalent to $H^c_n$ after localization in some sense.
See also \cite[\S4.6]{KangKashiwaraTsuchioka11} (beware that our $q$ is their $q^2$).

Now finally let $q=-1$ and $2a=0$, so that $K_0=K_1=\kk$ and $K_n=0$ for $n\ge2$.
Similar to the case above for $2a=0$, we obtain the following.

\begin{corollary*}
Suppose $\kk$ is a field and $q=-1$, $2a=0$.
Then there is a one-to-one correspondence
\[\Irr(H^c_n)/\Pi\onetoone\{\textup{$2$-restricted partition of $n$}\}.
\qedhere\]
\end{corollary*}

In fact, the two results for $2a=0$
are already obtained in Remark~\ref{rem:super_reduction}.

\bibliographystyle{model1b-num-names}
\bibliography{ref}

\begin{thebibliography}{46}
\expandafter\ifx\csname natexlab\endcsname\relax\def\natexlab#1{#1}\fi
\providecommand{\url}[1]{\texttt{#1}}
\providecommand{\href}[2]{#2}
\providecommand{\path}[1]{#1}
\providecommand{\DOIprefix}{doi:}
\providecommand{\ArXivprefix}{arXiv:}
\providecommand{\URLprefix}{URL: }
\providecommand{\Pubmedprefix}{pmid:}
\providecommand{\doi}[1]{\href{http://dx.doi.org/#1}{\path{#1}}}
\providecommand{\Pubmed}[1]{\href{pmid:#1}{\path{#1}}}
\providecommand{\bibinfo}[2]{#2}
\ifx\xfnm\relax \def\xfnm[#1]{\unskip,\space#1}\fi
\bibitem[{Ariki(1996)}]{Ariki96}
\bibinfo{author}{S.~Ariki}, \bibinfo{title}{On the decomposition numbers of the
  {H}ecke algebra of {$G(m,1,n)$}}, \bibinfo{journal}{J. Math. Kyoto Univ.}
  \bibinfo{volume}{36} (\bibinfo{year}{1996}) \bibinfo{pages}{789--808}.
\bibitem[{Auslander et~al.(1992)Auslander, Platzeck and
  Todorov}]{AuslanderPlatzeckTodorov92}
\bibinfo{author}{M.~Auslander}, \bibinfo{author}{M.I. Platzeck},
  \bibinfo{author}{G.~Todorov}, \bibinfo{title}{Homological theory of
  idempotent ideals}, \bibinfo{journal}{Trans. Amer. Math. Soc.}
  \bibinfo{volume}{332} (\bibinfo{year}{1992}) \bibinfo{pages}{667--692}.
\bibitem[{Brundan(1998)}]{Brundan98}
\bibinfo{author}{J.~Brundan}, \bibinfo{title}{Modular branching rules and the
  {M}ullineux map for {H}ecke algebras of type {$A$}}, \bibinfo{journal}{Proc.
  London Math. Soc. (3)} \bibinfo{volume}{77} (\bibinfo{year}{1998})
  \bibinfo{pages}{551--581}.
\bibitem[{Brundan and Kleshchev(2001)}]{BrundanKleshchev01}
\bibinfo{author}{J.~Brundan}, \bibinfo{author}{A.~Kleshchev},
  \bibinfo{title}{Hecke-{C}lifford superalgebras, crystals of type
  {$A_{2l}^{(2)}$} and modular branching rules for {$\hat S_n$}},
  \bibinfo{journal}{Represent. Theory} \bibinfo{volume}{5}
  (\bibinfo{year}{2001}) \bibinfo{pages}{317--403}.
\bibitem[{Brundan and Kleshchev(2002)}]{BrundanKleshchev02}
\bibinfo{author}{J.~Brundan}, \bibinfo{author}{A.~Kleshchev},
  \bibinfo{title}{Projective representations of symmetric groups via {S}ergeev
  duality}, \bibinfo{journal}{Math. Z.} \bibinfo{volume}{239}
  (\bibinfo{year}{2002}) \bibinfo{pages}{27--68}.
\bibitem[{Cline et~al.(1988)Cline, Parshall and Scott}]{ClineParshallScott88}
\bibinfo{author}{E.~Cline}, \bibinfo{author}{B.~Parshall},
  \bibinfo{author}{L.~Scott}, \bibinfo{title}{Finite-dimensional algebras and
  highest weight categories}, \bibinfo{journal}{J. Reine Angew. Math.}
  \bibinfo{volume}{391} (\bibinfo{year}{1988}) \bibinfo{pages}{85--99}.
\bibitem[{Deligne(2007)}]{Deligne07}
\bibinfo{author}{P.~Deligne}, \bibinfo{title}{La cat\'egorie des
  repr\'esentations du groupe sym\'etrique {$S_t$}, lorsque {$t$} n'est pas un
  entier naturel}, in: \bibinfo{booktitle}{Algebraic groups and homogeneous
  spaces}, Tata Inst. Fund. Res. Stud. Math., \bibinfo{publisher}{Tata Inst.
  Fund. Res.}, \bibinfo{address}{Mumbai}, \bibinfo{year}{2007}, pp.
  \bibinfo{pages}{209--273}.
\bibitem[{Dipper and James(1986)}]{DipperJames86}
\bibinfo{author}{R.~Dipper}, \bibinfo{author}{G.~James},
  \bibinfo{title}{Representations of {H}ecke algebras of general linear
  groups}, \bibinfo{journal}{Proc. London Math. Soc. (3)} \bibinfo{volume}{52}
  (\bibinfo{year}{1986}) \bibinfo{pages}{20--52}.
\bibitem[{Dipper and James(1987)}]{DipperJames87}
\bibinfo{author}{R.~Dipper}, \bibinfo{author}{G.~James}, \bibinfo{title}{Blocks
  and idempotents of {H}ecke algebras of general linear groups},
  \bibinfo{journal}{Proc. London Math. Soc. (3)} \bibinfo{volume}{54}
  (\bibinfo{year}{1987}) \bibinfo{pages}{57--82}.
\bibitem[{Dipper and James(1989)}]{DipperJames89}
\bibinfo{author}{R.~Dipper}, \bibinfo{author}{G.~James}, \bibinfo{title}{The
  {$q$}-{S}chur algebra}, \bibinfo{journal}{Proc. London Math. Soc. (3)}
  \bibinfo{volume}{59} (\bibinfo{year}{1989}) \bibinfo{pages}{23--50}.
\bibitem[{Du and Rui(1998)}]{DuRui98}
\bibinfo{author}{J.~Du}, \bibinfo{author}{H.~Rui}, \bibinfo{title}{Based
  algebras and standard bases for quasi-hereditary algebras},
  \bibinfo{journal}{Trans. Amer. Math. Soc.} \bibinfo{volume}{350}
  (\bibinfo{year}{1998}) \bibinfo{pages}{3207--3235}.
\bibitem[{Fulton(1997)}]{Fulton97}
\bibinfo{author}{W.~Fulton}, \bibinfo{title}{Young tableaux},
  volume~\bibinfo{volume}{35} of \textit{\bibinfo{series}{London Mathematical
  Society Student Texts}}, \bibinfo{publisher}{Cambridge University Press},
  \bibinfo{address}{Cambridge}, \bibinfo{year}{1997}. \bibinfo{note}{With
  applications to representation theory and geometry}.
\bibitem[{Graham and Lehrer(1996)}]{GrahamLehrer96}
\bibinfo{author}{J.J. Graham}, \bibinfo{author}{G.I. Lehrer},
  \bibinfo{title}{Cellular algebras}, \bibinfo{journal}{Invent. Math.}
  \bibinfo{volume}{123} (\bibinfo{year}{1996}) \bibinfo{pages}{1--34}.
\bibitem[{Grojnowski(1999)}]{Grojnowski99}
\bibinfo{author}{I.~Grojnowski}, \bibinfo{title}{Affine
  {$\widehat{\mathfrak{sl}}_p$} controls the representation theory of the
  symmetric group and related {H}ecke algebras}, \bibinfo{year}{1999}.
  \bibinfo{note}{\href{http://arxiv.org/abs/math/9907129}{arXiv:math/9907129}}.
\bibitem[{Hill et~al.(2011)Hill, Kujawa and Sussan}]{HillKujawaSussan11}
\bibinfo{author}{D.~Hill}, \bibinfo{author}{J.R. Kujawa},
  \bibinfo{author}{J.~Sussan}, \bibinfo{title}{Degenerate affine
  {H}ecke-{C}lifford algebras and type {$Q$} {L}ie superalgebras},
  \bibinfo{journal}{Math. Z.} \bibinfo{volume}{268} (\bibinfo{year}{2011})
  \bibinfo{pages}{1091--1158}.
\bibitem[{Hoefsmit(1974)}]{Hoefsmit74}
\bibinfo{author}{P.N. Hoefsmit}, \bibinfo{title}{Representations of {H}ecke
  algebras of finite groups with {BN}-pairs of classical type},
  \bibinfo{publisher}{ProQuest LLC, Ann Arbor, MI}, \bibinfo{year}{1974}.
  \bibinfo{note}{Thesis (Ph.D.)--The University of British Columbia (Canada)}.
\bibitem[{Hu(2006)}]{Hu06}
\bibinfo{author}{J.~Hu}, \bibinfo{title}{Mullineux involution and twisted
  affine {L}ie algebras}, \bibinfo{journal}{J. Algebra} \bibinfo{volume}{304}
  (\bibinfo{year}{2006}) \bibinfo{pages}{557--576}.
\bibitem[{Humphreys(1990)}]{Humphreys90}
\bibinfo{author}{J.E. Humphreys}, \bibinfo{title}{Reflection groups and
  {C}oxeter groups}, volume~\bibinfo{volume}{29} of
  \textit{\bibinfo{series}{Cambridge Studies in Advanced Mathematics}},
  \bibinfo{publisher}{Cambridge University Press},
  \bibinfo{address}{Cambridge}, \bibinfo{year}{1990}.
\bibitem[{Iglesias and Torrecillas(1995)}]{IglesiasTorrecillas95}
\bibinfo{author}{F.C. Iglesias}, \bibinfo{author}{J.G. Torrecillas},
  \bibinfo{title}{Wide {M}orita contexts}, \bibinfo{journal}{Comm. Algebra}
  \bibinfo{volume}{23} (\bibinfo{year}{1995}) \bibinfo{pages}{601--622}.
\bibitem[{Iglesias and Torrecillas(1998)}]{IglesiasTorrecillas98}
\bibinfo{author}{F.C. Iglesias}, \bibinfo{author}{J.G. Torrecillas},
  \bibinfo{title}{Wide {M}orita contexts and equivalences of comodule
  categories}, \bibinfo{journal}{J. Pure Appl. Algebra} \bibinfo{volume}{131}
  (\bibinfo{year}{1998}) \bibinfo{pages}{213--225}.
\bibitem[{Kang(2003)}]{Kang03}
\bibinfo{author}{S.J. Kang}, \bibinfo{title}{Crystal bases for quantum affine
  algebras and combinatorics of {Y}oung walls}, \bibinfo{journal}{Proc. London
  Math. Soc. (3)} \bibinfo{volume}{86} (\bibinfo{year}{2003})
  \bibinfo{pages}{29--69}.
\bibitem[{Kang et~al.(2013{\natexlab{a}})Kang, Kashiwara and
  Oh}]{KangKashiwaraOh13}
\bibinfo{author}{S.J. Kang}, \bibinfo{author}{M.~Kashiwara},
  \bibinfo{author}{S.j. Oh}, \bibinfo{title}{Supercategorification of quantum
  {K}ac-{M}oody algebras}, \bibinfo{journal}{Adv. Math.} \bibinfo{volume}{242}
  (\bibinfo{year}{2013}{\natexlab{a}}) \bibinfo{pages}{116--162}.
\bibitem[{Kang et~al.(2013{\natexlab{b}})Kang, Kashiwara and
  Oh}]{KangKashiwaraOh13_2}
\bibinfo{author}{S.J. Kang}, \bibinfo{author}{M.~Kashiwara},
  \bibinfo{author}{S.j. Oh}, \bibinfo{title}{Supercategorification of quantum
  {K}ac-{M}oody algebras {II}}, \bibinfo{year}{2013}{\natexlab{b}}.
  \bibinfo{note}{\href{http://arxiv.org/abs/1303.1916}{arXiv:1303.1916}}.
\bibitem[{Kang et~al.(2011)Kang, Kashiwara and
  Tsuchioka}]{KangKashiwaraTsuchioka11}
\bibinfo{author}{S.J. Kang}, \bibinfo{author}{M.~Kashiwara},
  \bibinfo{author}{S.~Tsuchioka}, \bibinfo{title}{Quiver hecke superalgebras},
  \bibinfo{year}{2011}.
  \bibinfo{note}{\href{http://arxiv.org/abs/1107.1039}{arXiv:1107.1039}, to
  appear in Journal f{\"u}r die reine und angewandte Mathematik}.
\bibitem[{Kashiwara(2002)}]{Kashiwara02}
\bibinfo{author}{M.~Kashiwara}, \bibinfo{title}{Bases cristallines des groupes
  quantiques}, volume~\bibinfo{volume}{9} of \textit{\bibinfo{series}{Cours
  Sp\'ecialis\'es [Specialized Courses]}}, \bibinfo{publisher}{Soci\'et\'e
  Math\'ematique de France}, \bibinfo{address}{Paris}, \bibinfo{year}{2002}.
  \bibinfo{note}{Edited by Charles Cochet}.
\bibitem[{Kazhdan and Lusztig(1979)}]{KazhdanLusztig79}
\bibinfo{author}{D.~Kazhdan}, \bibinfo{author}{G.~Lusztig},
  \bibinfo{title}{Representations of {C}oxeter groups and {H}ecke algebras},
  \bibinfo{journal}{Invent. Math.} \bibinfo{volume}{53} (\bibinfo{year}{1979})
  \bibinfo{pages}{165--184}.
\bibitem[{Kelly(1982)}]{Kelly82}
\bibinfo{author}{G.M. Kelly}, \bibinfo{title}{Basic concepts of enriched
  category theory}, volume~\bibinfo{volume}{64} of
  \textit{\bibinfo{series}{London Mathematical Society Lecture Note Series}},
  \bibinfo{publisher}{Cambridge University Press},
  \bibinfo{address}{Cambridge}, \bibinfo{year}{1982}.
\bibitem[{Kleshchev(1995)}]{Kleshchev95}
\bibinfo{author}{A.~Kleshchev}, \bibinfo{title}{Branching rules for modular
  representations of symmetric groups. {II}}, \bibinfo{journal}{J. Reine Angew.
  Math.} \bibinfo{volume}{459} (\bibinfo{year}{1995})
  \bibinfo{pages}{163--212}.
\bibitem[{Kleshchev(2005)}]{Kleshchev05}
\bibinfo{author}{A.~Kleshchev}, \bibinfo{title}{Linear and projective
  representations of symmetric groups}, volume \bibinfo{volume}{163} of
  \textit{\bibinfo{series}{Cambridge Tracts in Mathematics}},
  \bibinfo{publisher}{Cambridge University Press},
  \bibinfo{address}{Cambridge}, \bibinfo{year}{2005}.
\bibitem[{Knuth(1970)}]{Knuth70}
\bibinfo{author}{D.E. Knuth}, \bibinfo{title}{Permutations, matrices, and
  generalized {Y}oung tableaux}, \bibinfo{journal}{Pacific J. Math.}
  \bibinfo{volume}{34} (\bibinfo{year}{1970}) \bibinfo{pages}{709--727}.
\bibitem[{K{\"o}nig and Xi(1998)}]{KonigXi98}
\bibinfo{author}{S.~K{\"o}nig}, \bibinfo{author}{C.~Xi}, \bibinfo{title}{On the
  structure of cellular algebras}, in: \bibinfo{booktitle}{Algebras and
  modules, {II} ({G}eiranger, 1996)}, volume~\bibinfo{volume}{24} of
  \textit{\bibinfo{series}{CMS Conf. Proc.}}, \bibinfo{publisher}{Amer. Math.
  Soc.}, \bibinfo{address}{Providence, RI}, \bibinfo{year}{1998}, pp.
  \bibinfo{pages}{365--386}.
\bibitem[{K{\"o}nig and Xi(1999)}]{KonigXi99}
\bibinfo{author}{S.~K{\"o}nig}, \bibinfo{author}{C.~Xi},
  \bibinfo{title}{Cellular algebras: inflations and {M}orita equivalences},
  \bibinfo{journal}{J. London Math. Soc. (2)} \bibinfo{volume}{60}
  (\bibinfo{year}{1999}) \bibinfo{pages}{700--722}.
\bibitem[{Lascoux et~al.(1996)Lascoux, Leclerc and
  Thibon}]{LascouxLeclercThibon96}
\bibinfo{author}{A.~Lascoux}, \bibinfo{author}{B.~Leclerc},
  \bibinfo{author}{J.Y. Thibon}, \bibinfo{title}{Hecke algebras at roots of
  unity and crystal bases of quantum affine algebras}, \bibinfo{journal}{Comm.
  Math. Phys.} \bibinfo{volume}{181} (\bibinfo{year}{1996})
  \bibinfo{pages}{205--263}.
\bibitem[{Mathas(1999)}]{Mathas99}
\bibinfo{author}{A.~Mathas}, \bibinfo{title}{Iwahori-{H}ecke algebras and
  {S}chur algebras of the symmetric group}, volume~\bibinfo{volume}{15} of
  \textit{\bibinfo{series}{University Lecture Series}},
  \bibinfo{publisher}{American Mathematical Society},
  \bibinfo{address}{Providence, RI}, \bibinfo{year}{1999}.
\bibitem[{Misra and Miwa(1990)}]{MisraMiwa90}
\bibinfo{author}{K.~Misra}, \bibinfo{author}{T.~Miwa}, \bibinfo{title}{Crystal
  base for the basic representation of {$U_q(\mathfrak{sl}(n))$}},
  \bibinfo{journal}{Comm. Math. Phys.} \bibinfo{volume}{134}
  (\bibinfo{year}{1990}) \bibinfo{pages}{79--88}.
\bibitem[{Mitchell(1972)}]{Mitchell72}
\bibinfo{author}{B.~Mitchell}, \bibinfo{title}{Rings with several objects},
  \bibinfo{journal}{Advances in Math.} \bibinfo{volume}{8}
  (\bibinfo{year}{1972}) \bibinfo{pages}{1--161}.
\bibitem[{Mori(2004)}]{Mori04}
\bibinfo{author}{M.~Mori}, \bibinfo{title}{The module categories of the
  {I}wahori--{H}ecke algebra and the {H}ecke--{C}lifford superalgebra in
  non-integral rank}, \bibinfo{year}{2004}. \bibinfo{note}{Unpublished}.
\bibitem[{Morita(1958)}]{Morita58}
\bibinfo{author}{K.~Morita}, \bibinfo{title}{Duality for modules and its
  applications to the theory of rings with minimum condition},
  \bibinfo{journal}{Sci. Rep. Tokyo Kyoiku Daigaku Sect. A} \bibinfo{volume}{6}
  (\bibinfo{year}{1958}) \bibinfo{pages}{83--142}.
\bibitem[{Murphy(1992)}]{Murphy92}
\bibinfo{author}{G.E. Murphy}, \bibinfo{title}{On the representation theory of
  the symmetric groups and associated {H}ecke algebras}, \bibinfo{journal}{J.
  Algebra} \bibinfo{volume}{152} (\bibinfo{year}{1992})
  \bibinfo{pages}{492--513}.
\bibitem[{Murphy(1995)}]{Murphy95}
\bibinfo{author}{G.E. Murphy}, \bibinfo{title}{The representations of {H}ecke
  algebras of type {$A_n$}}, \bibinfo{journal}{J. Algebra}
  \bibinfo{volume}{173} (\bibinfo{year}{1995}) \bibinfo{pages}{97--121}.
\bibitem[{Nicholson and Watters(1988)}]{NicholsonWatters88}
\bibinfo{author}{W.K. Nicholson}, \bibinfo{author}{J.F. Watters},
  \bibinfo{title}{Morita context functors}, \bibinfo{journal}{Math. Proc.
  Cambridge Philos. Soc.} \bibinfo{volume}{103} (\bibinfo{year}{1988})
  \bibinfo{pages}{399--408}.
\bibitem[{Olshanski(1992)}]{Olshanski92}
\bibinfo{author}{G.I. Olshanski}, \bibinfo{title}{Quantized universal
  enveloping superalgebra of type {$Q$} and a super-extension of the {H}ecke
  algebra}, \bibinfo{journal}{Lett. Math. Phys.} \bibinfo{volume}{24}
  (\bibinfo{year}{1992}) \bibinfo{pages}{93--102}.
\bibitem[{Sagan(1987)}]{Sagan87}
\bibinfo{author}{B.E. Sagan}, \bibinfo{title}{Shifted tableaux, {S}chur
  {$Q$}-functions, and a conjecture of {R}. {S}tanley}, \bibinfo{journal}{J.
  Combin. Theory Ser. A} \bibinfo{volume}{45} (\bibinfo{year}{1987})
  \bibinfo{pages}{62--103}.
\bibitem[{Tsuchioka(2010)}]{Tsuchioka10}
\bibinfo{author}{S.~Tsuchioka}, \bibinfo{title}{Hecke-{C}lifford superalgebras
  and crystals of type {$D^{(2)}_l$}}, \bibinfo{journal}{Publ. Res. Inst. Math.
  Sci.} \bibinfo{volume}{46} (\bibinfo{year}{2010}) \bibinfo{pages}{423--471}.
\bibitem[{Wan(2010)}]{Wan10}
\bibinfo{author}{J.~Wan}, \bibinfo{title}{Completely splittable representations
  of affine {H}ecke-{C}lifford algebras}, \bibinfo{journal}{J. Algebraic
  Combin.} \bibinfo{volume}{32} (\bibinfo{year}{2010}) \bibinfo{pages}{15--58}.
\bibitem[{Xi(1999)}]{Xi99}
\bibinfo{author}{C.~Xi}, \bibinfo{title}{Partition algebras are cellular},
  \bibinfo{journal}{Compositio Math.} \bibinfo{volume}{119}
  (\bibinfo{year}{1999}) \bibinfo{pages}{99--109}.

\end{thebibliography}

\end{document}